\documentclass[aos,noinfoline]{imsart}

\RequirePackage[OT1]{fontenc}
\RequirePackage{amsthm,amsmath,natbib}
\RequirePackage[colorlinks,citecolor=blue,urlcolor=blue]{hyperref}

\usepackage[T1]{fontenc}												
\usepackage[utf8x]{inputenc}											
\usepackage[english]{babel}												
\usepackage{amsmath,amstext,amssymb,amsfonts}
\usepackage{amsthm}													
\usepackage{dsfont} 													
\usepackage{graphicx}													
\usepackage{hyperref}													
\usepackage[titletoc,title]{appendix}									
\usepackage{lmodern} 													
\usepackage{xcolor}														
\usepackage{autonum}													
\usepackage[medium]{titlesec}											
\allowdisplaybreaks														

\usepackage[xcolor,pdftex,leftbars]{changebar}							
\cbcolor{black}

\startlocaldefs
\numberwithin{equation}{section}
\theoremstyle{plain}

\theoremstyle{remark}
\newtheorem{rem}{Remark}[section]
\endlocaldefs

\newtheorem{theorem}{Theorem}[section]

\newtheorem{lemma}[theorem]{Lemma}

\newtheorem{proposition}[theorem]{Proposition}
\theoremstyle{definition} 								
\newtheorem{example}[theorem]{Example}
\numberwithin{equation}{section}   						

\newcommand{\ind}[1]{\mathds{1}_{#1}}
\newcommand{\Ind}[1]{\mathds{1}_{\{#1\}}}
\newcommand{\ceil}[1]{\lceil #1 \rceil}
\newcommand{\floor}[1]{\lfloor #1 \rfloor}

\def\N{{\mathbb{N}}}
\def\R{{\mathbb{R}}}
\def\Z{{\mathbb{Z}}}

\def\B{{\mathbb{B}}}

\def\GG{{\mathcal{G}}}
\def\AA{{\mathcal{A}}}
\def\BB{{\mathcal{B}}}

\def\eps{{\varepsilon}}
\def\e{{\rm e}}

\def\ord{{\hbox{o}}}
\def\Ord{{\hbox{O}}}

\def\pA{{P\{\Theta_i\in A\}}}
\def\pnAh{{\hat p_{n,A}}}
\def\pnAhh{{\hat{\hat p}_{n,A}}}
\def\alphanh{{\hat\alpha_n}}
\def\Znp{{Z_n^{(p)}}}
\def\Zp{{Z^{(p)}}}
\def\Zphi{{Z_\phi}}
\def\Znphi{{Z_{n,\phi}}}
\def\Ztn{{Z_n^{(T)}}}
\def\pnfh{{\hat p_{n,A}^f}}
\def\pnbh{{\hat p_{n,A}^b}}
\def\DN{D{\&}N}
	
\begin{document}

\begin{frontmatter}
\title{Cluster based inference for extremes of time series}
\runtitle{Cluster based inference}
\begin{aug}
  \author[A]{\fnms{Holger} \snm{Drees}\ead[label=e1]{drees@math.uni-hamburg.de}},
  \author[B]{\fnms{Anja} \snm{Jan{\ss}en}\ead[label=e2]{anja.janssen@ovgu.de}},
  \author[A]{\fnms{Sebastian} \snm{Neblung}\ead[label=e3,mark]{sebastian.neblung@uni-hamburg.de}}
  \address[A]{University of Hamburg, Department of Mathematics,
 SPST, Bundesstr.\ 55, 20146 Hamburg, Germany, \printead{e1,e3}}
 \address[B]{Otto-von-Guericke University of Magdeburg, Faculty of Mathematics, IMST,
 Universit\"{a}tsplatz~2, 39106 Magdeburg, Germany, \printead{e2}}
\end{aug}

\begin{abstract}
  \; We introduce a new type of estimator for the spectral tail process of a regularly varying time series. The approach is based on a characterizing invariance property of the spectral tail process, which is incorporated into the new estimator via a projection technique. We show uniform asymptotic normality of this estimator, both in the case of known and of unknown index of regular variation. In a simulation study the new procedure shows a more stable performance than previously proposed estimators.
\end{abstract}


\end{frontmatter}												
\section{Introduction}
Statistical methods for assessing environmental hazards, technical failures or financial risks often call for estimating rare events and make use of extreme value theory, see, e.g., \cite{castillo2012extreme, coles2001introduction, embrechts2013} or \cite{finkenstadt2003extreme}. In many of these applications, the temporal dependence structure of sequential observations matters: extremal events may happen in clusters and the length and shape of a cluster are crucial for its impact. For example, heavy rainfall over several days has a different effect on water levels of adjacent rivers than a short cloudburst. Likewise, consecutive large daily losses of a stock will have a more severe impact on an investor's portfolio than a single drop in its value.

In order to have a theoretical framework for such extremal dependence structures, we will assume that our quantities of interest can be described by an $\R^d$-valued  stationary regularly varying time series $(X_t)_{t \in \Z}$. The distribution of its \emph{tail process} $(Y_t)_{t \in \Z}$ is defined by \begin{equation}\label{eq:tailprocess}
\lim_{u \to \infty}P\left(\frac{X_{-m}}{u}\leq x_{-m}, \ldots, \frac{X_{n}}{u}\leq x_{n} \,\Big|\, \|X_0\|>u\right) = P\left\{Y_{-m}\leq x_{-m}, \ldots, Y_n \leq x_n  \right\}
\end{equation}
for all $m,n \in \N_0$ and continuity points $(x_{-m}, \ldots, x_n) \in (\R^d)^{n+m+1}$ of $P^{(Y_t)_{-m \leq t \leq n}}$, where $\| \cdot \|$ denotes an arbitrary but fixed norm on $\R^d$ and all inequalities between vectors are meant componentwise. As shown in \cite{basrak2009}, Theorems 2.1 and 3.1,  $\|Y_0\|$ is Pareto distributed with $P\{\|Y_0\|>x\}=x^{-\alpha}, x\geq 1,$ for some $\alpha>0$ (the so-called index of regular variation), and it is independent of the \emph{spectral tail process}  $\Theta=(\Theta_t)_{t \in \Z}:= (Y_t/\|Y_0\|)_{t \in Z}$, which thus carries all information about the dependence structure of the tail process. It is our aim  to provide estimators for quantities derived from $\Theta$, exemplary done here for $P\{\Theta_i \in A\}$ for suitable Borel sets $A$ in $\R^d$ and $i \in \Z$.

If we observe a finite stretch from a stationary regularly varying time series $(X_t)_{t \in \Z}$, Equation \eqref{eq:tailprocess} together with the definition of $\Theta$ leads to a straightforward estimator of $P\{\Theta_i \in A\}$ given by
\begin{equation}\label{Eq:forwardest}
\hat{p}^f_{n,A}:=\frac{\sum_{t=1}^n\mathds{1}_{\{ \|X_t\|>u_n\}}\ind{A}(X_{t+i}/\|X_t\|)}{\sum_{t=1}^n\mathds{1}_{\{\|X_t\|>u_n\}}},
\end{equation}
where $u_n$ denotes a suitably high threshold, depending on $n$, the number of consecutively observed vectors $(X_t,\ldots,X_{t+i})$. This estimator is called \emph{forward estimator} and asymptotic normality as $n \to \infty$ has been shown under suitable assumptions in the univariate case in \cite{davis2018}, Theorem 3.1.

Now, the assumption of stationarity of $(X_t)_{t \in \Z}$ does not only allow for the construction of consistent estimators, it also has an effect on the structure of the spectral tail process, which can be shown to satisfy the so-called \emph{time change formula}, see \cite{basrak2009}, Theorem 3.1. Namely, for all bounded and measurable $f:(\R^d)^{n+m+1} \to \R$, $n, m \in \N_0$ satisfying $f(x_{-m}, \ldots, x_0, \ldots , x_n)=0$ whenever $x_0=0$, we have
\begin{equation}\label{eq:timechangeform}
E\left[f((\Theta_{t+i})_{-m\leq t\leq n}) \right]=E\left[f\left(\frac{(\Theta_{t})_{-m\leq t\leq n}}{\|\Theta_{-i}\|}\right) \|\Theta_{-i}\|^\alpha \right]
\end{equation}
for all $i \in \Z$.

 For the univariate case and sets of the form $A=(x,\infty)$ or $A=(-\infty, -x)$, $x\geq 0,$ this identity with $f(x_0)=\mathds{1}_{A}(x_0)$ was used in \cite{drees2015} for $i=1$ and in \cite{davis2018} for general $i$ to design an alternative to $\pnfh$, based on the same number $n$ of consecutively observed vectors, given by
\begin{equation}\label{Eq:backwardest}
\hat{p}^b_{n,A}:=\frac{\sum_{t=1}^n\mathds{1}_{\{ \|X_t\|>u_n\}}\ind{A}(X_t/\|X_{t-i}\|) \left\|X_{t-i}/X_t\right\|^{\hat{\alpha}_n}}{\sum_{t=1}^n\mathds{1}_{\{\|X_t\|>u_n\}}},
\end{equation}
where $\hat{\alpha}_n$ is the Hill estimator of $\alpha$, given by
 \begin{equation} \label{eq:Hilldef}
 \alphanh := \frac{\sum_{t=1}^n \Ind{\|X_t\|>u_n}}{\sum_{t=1}^n \log^+(\|X_t\|/u_n)},
 \end{equation}
 with $\log^+ x =\log(x)\vee 0$.
 Simulation studies have shown that this so-called \emph{backward estimator} can have a smaller root mean squared error than the forward estimator in particular for larger values of $x$, see \cite{drees2015} and \cite{davis2018}.

 In this article, we take a more comprehensive approach for using the stationarity of $(X_t)_{t \in \Z}$ when constructing an estimator for the distribution of $\Theta$. It has been shown in \cite{janssen2019} that  the class of all possible spectral tail processes of underlying stationary processes can be characterized with the help of the so-called RS-transform. For a stochastic processes $W=(W_t)_{t \in \Z}$ we write $P^W$ for its distribution and set
 	$$\|W\|_\alpha := \Big(\sum_{t\in\Z}\|W_t\|^\alpha\Big)^{1/\alpha} \in [0,\infty] $$
 for some $\alpha>0$. If $0<\|W\|_\alpha<\infty$ a.s. then the \emph{RS-transform} of $P^W$ is given by
 \begin{equation}\label{eq:RStransform}
 (P^W)^{RS}(B)=E\left(\sum_{j \in \Z}\frac{\|W_j\|^\alpha}{ \|W\|_\alpha^\alpha}\mathds{1}_{B}\left(\frac{\left(W_{t+j}\right)_{t \in \Z}}{\|W_j\|}\right) \right)
 \end{equation}
 for all cylinder sets $B$ in $(\R^d)^{\Z}$. Note that this RS-transform depends on $\alpha$, which we suppress in the notation.
 Theorem 2.4 in \cite{janssen2019} states that processes $\Theta=(\Theta_t)_{t \in \Z}$ with $0<\|\Theta \|_\alpha <\infty$ a.s.\ are spectral tail processes of some underlying stationary time series with index $\alpha$ if and only if their distribution is invariant under the RS-transformation, i.e. $(P^\Theta)^{RS}{=}P^\Theta$.

 Direct calculations furthermore show that for any $P^W$ with $0<\|W\|_\alpha <\infty$ a.s. we have
 $$ ((P^W)^{RS})^{RS}=(P^W)^{RS}, $$
 i.e. for such $P^W$, its RS-transform $(P^W)^{RS}$ satisfies the stated invariance property and is thereby a valid distribution of a spectral tail process of a stationary time series. The RS-transformation can thus be seen as a projection of an initial distribution into the class of admissible limit objects for our estimation problem. Note that common projection estimators would typically transform an initial estimator into its best approximation from the space of admissible quantities by minimizing a suitably chosen norm, see for example \cite{fils2008projection}. In contrast, our approach transforms an initial estimator for $\Theta$ onto the space of admissible processes by making sure that this transformation induces only a random shift in time and scale.

 A straightforward initial empirical version of the distribution of $\Theta$, trimmed to a finite length $2s_n+1, s_n \geq 0,$ can be derived from $n$ consecutive (and overlapping) stretches of $(X_t)_{t \in \Z}$ of this length by
 \begin{equation} \label{eq:def_emp_est}
 \widehat{P^\Theta}=\frac 1{\sum_{t=1}^n  \Ind{\|X_t\|>u_n}} \sum_{t=1}^n \mathds{1}_{\{\|X_t\|>u_n\}}
 \delta_{(\ldots,  0, (\frac{X_{t+j}}{\|X_t\|})_{|j|\leq s_n},  0, \ldots)},
 \end{equation}
 where $\delta_{(\ldots, 0, (\frac{X_{t+j}}{\|X_t\|})_{|j|\leq s_n},  0, \ldots)}$ is the Dirac measure placing mass 1 in the double sided sequence $(z_j)_{j \in \Z}$ with $z_j=X_{t+j}/\|X_t\|, |j| \leq s_n,$ and $z_j=0$, $|j|>s_n$.
   We then apply the RS-transformation to the realization of $\widehat{P^\Theta}$ which yields
  \begin{align}
  	(\widehat{P^\Theta})^{RS}  =&  \frac 1{\sum_{t=1}^n  \Ind{\|X_t\|>u_n}} 	 \sum_{t=1}^n
  	\frac{\Ind{\|X_t\|>u_n}}{\sum_{h=-s_n}^{s_n} \|X_{t+h}\|^\alpha} \sum_{h=-s_n}^{s_n}\|X_{t+h}\|^\alpha  \delta_{(\ldots,  0, (\frac{X_{t+h+j}}{\|X_{t+h}\|})_{j:|h+j|\leq s_n},  0, \ldots)}
  \end{align}
  as a projection based estimator of $P^\Theta$.
  Thus we ensure that our final estimator is derived from a quantity which has all essential structural properties of a spectral tail process. This gives for $P\{\Theta_i\in A\}$ the \emph{projection based estimator}
 \begin{align} \label{eq:pnAhdef}
  \pnAh  &:= (\widehat{P^\Theta})^{RS}\{(x_t)_{t \in \mathbb{Z}} \in (\R^d)^\Z \mid x_i \in A\} =  \frac 1{\sum_{t=1}^n  \Ind{\|X_t\|>u_n}}  \\
 & \times \sum_{t=1}^n
 \frac{\Ind{\|X_t\|>u_n}}{\sum_{h=-s_n}^{s_n} \|X_{t+h}\|^\alpha}\bigg( \sum_{h\in H_n}\|X_{t+h}\|^\alpha  \ind{A}\Big(\frac{X_{t+h+i}}{\|X_{t+h}\|}\Big)
 + \sum_{h\in H_n^c}\|X_{t+h}\|^\alpha \ind{A}(0) \bigg),
 \end{align}
 where $H_n:=\{(-s_n-i)\vee (-s_n),\ldots, (s_n-i)\wedge s_n\}$ and $H_n^c:=\{-s_n,\ldots,s_n\}\setminus H_n$ which equals $\{s_n-i+1,\ldots, s_n\}$ for $i>0$, $\{-s_n,\ldots,-s_n-i-1\}$ for $i<0$, and it is the empty set for $i=0$. We denote the version with estimated $\alpha$ according to \eqref{eq:Hilldef} by
\begin{align} \label{eq:pnAhhdef}
  \pnAhh  & := \frac 1{\sum_{t=1}^n  \Ind{\|X_t\|>u_n}} \sum_{t=1}^n
    \frac{\Ind{\|X_t\|>u_n}}{\sum_{h=-s_n}^{s_n} \|X_{t+h}\|^{\hat\alpha_n}} \times \\
     &  \hspace{3cm} \bigg( \sum_{h\in H_n}\|X_{t+h}\|^{\hat\alpha_n} \ind{A}\Big(\frac{X_{t+h+i}}{\|X_{t+h}\|}\Big)
   + \sum_{h\in H_n^c}\|X_{t+h}\|^{\hat\alpha_n} \ind{A}(0) \bigg).
\end{align}
In Section \ref{section:asymptotics} we state the uniform asymptotic normality of these projection based estimators.
In Subsection \ref{section:SRE} the conditions needed in these results are verified for solutions to stochastic recurrence equations.
The simulations in Section \ref{section:simulations} show that the finite sample performance of $\pnAhh$ is favorable compared to that of $\pnfh$ and $\pnbh$ in several settings. The proofs of the results in Section \ref{section:asymptotics} are given in the Appendix, with some auxiliary lemmas and calculations and further simulations deferred to the Supplement.

Throughout, for vectors $x,y\in\R^q$, we interpret $x\le y$ componentwise, i.e.\ as $x_j\le y_j$ for all $1\le j\le q$, while $x\lneq y$ means $x\leq y$ and $x\ne y$.
Where needed, 0 is also interpreted as a vector with all coordinates equal to 0,  $[x,y):=\times_{j=1}^q [x_j,y_j)$ and $[x,y]:=\times_{j=1}^q [x_j,y_j]$. We denote the positive and negative part of the logarithm by $\log^+x := \log x\vee 0$ and $\log^-x := (-\log x)\vee 0$, respectively. The limit from the left of a function $f$ at $x$ is denoted by $f(x-)$.

\section{Asymptotics}
\label{section:asymptotics}

In this section we establish convergence of the suitably standardized estimators $\pnAh$ and $\pnAhh$ uniformly for all sets $A$ in some family $\AA\subset \B^d$. To this end, note that the estimator with known index of regular variation can be written in the form $\pnAh=T_{n,A}/T_{n,\R^d}$ where
\begin{align}
 T_{n,A} & := \sum_{t=1}^n g_A(W_{n,t}) \qquad \text{with}\qquad
 W_{n,t}  := \bigg(\frac{X_{t+h}}{u_n}\bigg)_{-s_n\le h\le s_n}, \\
 g_A\big((w_h)_{-s\le h\le s}\big) & := \frac{\Ind{\|w_0\|>1}}{\sum_{h=-s}^s \|w_h\|^\alpha} \sum_{h=-s}^s \|w_h\|^\alpha\Big( \ind{A}\Big(\frac{w_{h+i}}{\|w_h\|}\Big)\Ind{|h+i|\le s}+\ind{A}(0)\Ind{|h+i|> s}\Big)
\end{align}
for all $s\in\N$. We will use the theory on the asymptotic behavior of quite general statistics based on sliding blocks of observations, developed by \cite{DN20}, to derive the asymptotics of $(T_{n,A})_{A\in\AA}$ and thus of our estimators.

Throughout, we fix the sequences $s_n\in\N$ and $u_n>0$, $n\in\N$, and the family $\AA$ of Borel sets in $\R^d$. To ensure consistency of the estimators, $u_n$ must tend to $\infty$ sufficiently slowly such that the expected number of exceedances among the observations tends to $\infty$, that is $nv_n\to\infty$ with
\begin{equation} \label{eq:vndef}
  v_n := P\{\|X_1\|>u_n\}.
\end{equation}
We assume that the observed time series is stationary, regularly varying and absolutely regular ($\beta$-mixing)  with $\beta$-mixing coefficients
\begin{equation} \label{eq:betaXdef}
 \beta_{k} := E \Big[
\sup_{B\in\BB_{k+1}^\infty} | P(B|\BB_{-\infty}^0)-P(B)|\Big]
\end{equation}
where $\BB_{-\infty}^0$  and $\BB_{k+1}^\infty$ denote the
$\sigma$-fields generated by $(X_l)_{l\le 0}$ and $(X_l)_{l\ge k+1}$, respectively.
More precisely, we assume
\begin{itemize}
	\item[\bf(RV)] $(X_t)_{t\in\Z}$ is an $\R^d$-valued stationary time series which is regularly varying with index $\alpha>0$, spectral tail process $\Theta=(\Theta_t)_{t\in\Z}$ and tail process $Y=(Y_t)_{t\in\Z}$.
	\item[\bf(S)] The sequences $s_n\in\N$ and $nv_n$ tend to $\infty$ as $n\to\infty$. There exist sequences $l_n,r_n\in\N$, $n\in\N$, such that $s_n\le l_n=\ord(r_n)$, $r_nv_n\to 0$, $r_n=\ord\big(\sqrt{nv_n}\big)$ and $\beta_{l_n}=o(r_n/n)$.
\end{itemize}
Note that if (S) holds for some sequence $(l_n)_{n\in\N}$, then it is fulfilled for any $\tilde l_n\ge l_n$ such that $\tilde l_n=o(r_n)$. Hence, w.l.o.g.\ we may assume that, for any fixed $k\in\N$, $\beta_{l_n-ks_n}=o(r_n/n)$ holds.
While (S) ensures that in the definition of $\pnAh$ summands whose indices $t$ differ by at least $l_n$ are almost independent, it does not suffice to control the size of clusters of extremes. To this end, we use the following two conditions:
\begin{itemize}
	\item[\bf(BC)] For each $c\in(0,1], n\in\N$ and  $1\le k\le r_n$ there exists
$$ e_{n,c}(k)\ge P\big(  \|X_k\|>cu_n\,\big|\, \|X_0\|>cu_n\big) $$
such that $e_{\infty,c}(k)=\lim_{n\to\infty} e_{n,c}(k)$ exists for all $k\in\N$ and $\lim_{n\to\infty} \sum_{k=1}^{r_n}  e_{n,c}(k) = \sum_{k=1}^\infty e_{\infty,c}(k)<\infty$ holds.
\item[\bf(TC)] $\displaystyle	\lim_{m\to\infty}\limsup_{n\to\infty}E\Big( \frac{\sum_{m<|h|\leq s_n} \|X_{h+j}\|^\alpha\Ind{\|X_{h+j}\|\leq cu_n} }{ \sum_{|h|\leq s_n} \|X_{h+j}\|^\alpha}\,\big|\, \|X_0\|>u_n\Big)=0$\\ \hspace*{8cm} for some $c\in (0,1)$ and all $j\in\mathbb{Z}$.
\end{itemize}
By stationarity, Condition (BC) implies the analogous inequality for $k<0$ with $e_{n,c}(k):=e_{n,c}(|k|)$, and it also implies a bound on the second moment of the cluster size. (TC) enables us to truncate clusters to finite lengths  in our asymptotic analysis.

Moreover, the following condition guarantees that $p_A:=\pA$ can be calculated as a limit using the definition of the spectral tail process:
\begin{itemize}
	\item[\bf(C$\mathbf{\Theta}$)] $P\{\Theta_i\in \partial A\}=0$ for all $A\in\AA$.
\end{itemize}	

Our first result states the joint asymptotic normality of finitely many estimators $\pnAh$. Define the process of standardized estimators
\begin{equation}
  \Znp(A) := \sqrt{nv_n}\big(\pnAh-E(g_A(W_{n,0})\mid \|X_0\|>u_n)\big), \quad A\in\AA.
\end{equation}
Since the functions $g_A$ are bounded by 1, the weak convergence \eqref{eq:tailprocess} implies
\begin{align}\label{eq:ewertconv}
E(g_A(W_{n,0})\mid \|X_0\|>u_n) \to E\Big( \|\Theta\|_\alpha^{-\alpha}\sum_{h\in\Z} \|\Theta_h\|^\alpha \ind{A}(\Theta_{h+i}/\|\Theta_h\|)\Big) = \pA,
\end{align}
where the last step holds by the invariance of the spectral tail process under the RS-transformation; see Supplement for details.
\begin{theorem}
	\label{th:fidiconv}
	If the conditions (RV), (S), (BC), (TC) and (C$\Theta$) are met, then the finite dimensional distributions (fidis) of $\Znp$ weakly converge to the fidis of $\Zp(A):= Z(A)-\pA Z(\R^d)$ with $Z$ denoting a centered Gaussian process with covariance function
\begin{align}
  c(A,B) & :=Cov(Z(A),Z(B))\\ & = 
  \sum_{j\in\Z} E\bigg[(\|\Theta_j\|^\alpha\wedge 1) 
  \bigg(\sum_{h\in\Z}\frac{\|\Theta_h\|^\alpha}{\|\Theta\|_\alpha^\alpha} \ind{A}\Big(\frac{\Theta_{h+i}}{\|\Theta_h\|}\Big)\bigg)
  \bigg(\sum_{h\in\Z}\frac{\|\Theta_h\|^\alpha}{\|\Theta\|_\alpha^\alpha}  \ind{B}\Big(\frac{\Theta_{h+i}}{\|\Theta_h\|}\Big) \bigg)
\bigg].  \label{eq:covZ}
\end{align}
 Hence, if, in addition, the bias condition
	\begin{equation}
		\label{eq:biascondition.gA}
	  E(g_A(W_{n,0})\mid \|X_0\|>u_n)-\pA =\ord\big((nv_n)^{-1/2})
	\end{equation}
	holds for all $A\in\AA$, then
$$ 	\sqrt{nv_n}\big(\pnAh-\pA\big)_{A\in\tilde\AA} \,\to\, (\Zp(A))_{A\in\tilde\AA}
$$
weakly for all finite subsets $\tilde\AA$ of $\AA$.
\end{theorem}
Note that \eqref{eq:biascondition.gA} only imposes a rate, while convergence is ensured by \eqref{eq:ewertconv}. Straightforward calculations show that the covariance function of the limit process $Z^{(p)}$ can be represented as follows:
\begin{align}
  Cov\big(Z^{(p)}(A),Z^{(p)}(B)\big) & = E[\chi(\xi(A)-p_A)(\xi(B)-p_B)]
\end{align}
with
\begin{align}
  \chi & :=  \sum_{j\in\Z} (\|\Theta_j\|^\alpha\wedge 1), \qquad 
  \xi  := \sum_{h\in\Z} \frac{\|\Theta_h\|^\alpha}{\|\Theta\|_\alpha^\alpha} \delta_{\Theta_{h+i}/\|\Theta_h\|}.
\end{align}
Note that $\xi$ is a random probability measure with expectation $E[\xi(A)]=(P^{\Theta})^{RS}\{x\in(\R^d)^{\Z}\mid x_i\in A\}=p_A$, whose realizations only depend on the shape of the paths of the spectral tail process, in the sense that they are invariant under a shift and re-scaling of the spectral tail process. In particular, $\xi$  is deterministic if the spectral tail process is the RS-transform of some deterministic process. In that case, the asymptotic variances are all equal to 0 although the spectral tail process may be non-trivial. The expected value of the random weight $\chi$ equals the expected number of values in the tail process with norm larger than 1. It reflects the fact that longer clusters of large observations contribute more summands to the estimator $\pnAh$ than shorter clusters do.


The above convergence will only hold uniformly over all $A\in\AA$ if the family of sets is not too complex. Such a process convergence in the space of bounded functions indexed by $\AA$ (equipped with the supremum norm) is relatively easy to establish using Vapnik-Chervonenkis theory if the sets under consideration are linearly ordered w.r.t.\ inclusion. However, this assumption is too restrictive if multivariate data is observed, i.e.\ if $d>1$. Here we assume that $\AA$ may be indexed by a unit cube of arbitrary dimension in a suitable way:
\begin{itemize}
	\item[\bf(A)] For some $q\in\N$, there exists a map $[0,1]^q\to \AA, t\mapsto A_t$ such that
  \begin{itemize}
     \item[(i)] $\AA=\{A_t|t\in[0,1]^q\}$, $A_{(1,\ldots,1)}=\R^d$ and $A_{(t_1,\ldots, t_q)}=\emptyset$ if $t_j=0$ for some $1\le j\le q$;
     \item[(ii)] for all $1\le j,k\le q$, and all $s_j,t_l\in[0,1]$, ($l\in\{1,\ldots,q\}\setminus \{k\}$)  the mapping $t_k\mapsto A_{(t_1,\ldots, t_q)}\setminus A_{(t_1,\ldots,t_{j-1},s_j,t_{j+1},\ldots,t_q)}$ is non-decreasing on $[0,1]$;
     \item[(iii)] the processes $\big(\sum_{j=1}^{r_n} g_{A_t}(W_{n,j})\big)_{t\in[0,1]^q}$ are separable;
     \item[(iv)] $P\{\Theta_i\in \partial A_t^-\}=0$ for all $t\in[0,1+\iota]^q$ for some $\iota>0$ where $A_t^-:= \bigcup_{s\in[0,t)}A_s$ and $A_t:=A_{t\wedge 1}$ for $t\not\in[0,1]^q$;
     \item[(v)] $P\big\{ \Theta_i \in \bigcap_{s\in(t,1]}A_{s^{(k)}}\setminus A_{t^{(k)}}\}=0$ for all $t\in[0,1)$ and $1\leq k\leq q$ where $t^{(k)}:=(1,\ldots,1,t,1,\ldots,1)$ with $t$ in the $k$-th coordinate;
     \item[(vi)]$P\big\{\|X_0\|>0,X_i/\|X_0\|\in \bigcap_{s\in(t,1]}A_{s^{(k)}}\setminus A_{t^{(k)}}\}=0$ for all $t\in[0,1)$ and $1\leq k\leq q$;
     \item[(vii)] there exists $w\in[0,1]^q$ such that $0\in A_w\setminus\bigcup_{s\lneq w} A_s$.
    \end{itemize}
\end{itemize}

The archetypical example is $\AA=\{(-\infty,t]\mid t\in\R^d\}\cup \{\mathbb{R}^d,\emptyset\}$ for which condition (A) is satisfied, provided the marginal distributions of $\Theta_i$ and $X_i/\|X_0\|$ are continuous.  To see this, let $q=d$ and $A_{(t_1,\ldots,t_d)} := \times_{i=1}^d (-\infty, \varpi(t_i)]\cap\R^d$ for some continuous, increasing mapping $\varpi$ from $[0,1]$ onto $[-\infty,\infty]$ and observe $\bigcap_{s\in(t,1]}A_{s^{(k)}}\setminus A_{t^{(k)}}=\emptyset$ for all $t\in[0,1)$ and $1\leq k\leq d$. Indeed, the above conditions cover almost all natural finite-dimensional families of sets if the distributions of $\Theta_i$ and $X_i/\|X_0\|$ are sufficiently smooth.

\begin{theorem}
	\label{th:procconv}
	If the conditions (RV), (S), (BC), (TC), (C$\Theta$) and (A) are fulfilled, then $\Znp$ weakly converge to the Gaussian process $\Zp$ defined in Theorem \ref{th:fidiconv}. If, in addition,
   \begin{equation} \label{eq:unifbiascond}
	  \sup_{A\in\AA}\big|E(g_A(W_{n,0})\mid \|X_0\|>u_n)-P\{\Theta_i\in A\}\big| =\ord\big((nv_n)^{-1/2}),
	\end{equation}
then \hfill
   $\displaystyle
       \sqrt{nv_n}\big(\pnAh-P\{\Theta_i\in A\}\big)_{A\in\AA} \,\to\, (\Zp(A))_{A\in\AA}.
  $ \hspace*{\fill}
\end{theorem}
\begin{rem}
	\label{rem:processcon.linear}
  In the case $q=1$, assumptions (A) (vi) and (vii) are not needed.
\end{rem}
Next we investigate the uniform asymptotic normality of our estimator $\pnAhh$ in the case when the index of regular variation is unknown. To analyze the joint asymptotic behavior of the denominator of the Hill type estimator $\hat\alpha_n$ and the statistics $T_{n,A}$, we  amend the cluster size condition (BC).  Since the logarithm function occurring in the definition of the Hill estimator is unbounded, further moment conditions are needed. In addition, we require the bias of the Hill type estimator to be negligible.
\begin{itemize}
	\item[\bf(BC')] For each $1\le k\le r_n$ there exists
$$ e'_n(k)\ge E\Big(\max\Big(\log\frac{\|X_0\|}{u_n},1\Big) \max\Big(\log\frac{\|X_k\|}{u_n},\Ind{\|X_k\|>u_n}\Big)\,\Big|\, \|X_0\|>u_n\Big) $$
such that $e'_{\infty}(k)=\lim_{n\to\infty} e'_{n}(k)$ exists for all $k\in\N$ and $\lim_{n\to\infty} \sum_{k=1}^{r_n}  e'_{n}(k) = \sum_{k=1}^\infty e'_{\infty}(k)<\infty$ holds.\\[-1ex]

   \item[\bf(M)]  There exists $\delta>0$ such that the following moment bounds hold:
     \begin{itemize}
       \item[(i)]
       $\displaystyle \sum_{k=1}^{r_n} \bigg[E\bigg( \Big( \log^+\Big(\frac{\|X_0\|}{u_n}\Big) \log^+\Big(\frac{\|X_k\|}{u_n}\Big)\Big)^{1+\delta}\,\Big|\, \|X_0\|>u_n\bigg)\bigg]^{1/(1+\delta)}=\Ord(1),
       $
       \item[(ii)]
         $\displaystyle \limsup_{m\to\infty}E\bigg(\frac{ \sum_{|h|\leq m} |\log\|\Theta_{h}\||^{1+\delta}\|\Theta_{h}\|^\alpha}{\sum_{|k|\leq m}\|\Theta_k\|^\alpha}\bigg)<\infty,
         $
        \item[(iii)] $\displaystyle \lim_{m\to\infty}\limsup_{n\to\infty} E\left(\frac{\sum_{m<|h|\leq s_n} \log^{-}(\|X_{h}\|/u_n)\|X_{h}\|^{\alpha} }{\sum_{|k|\leq s_n} \|X_{k}\|^{\alpha }}  \, \Big|\, \|X_0\|>u_n\right)=0.$
     \end{itemize}

   \item[({\bf H})]
   $\displaystyle E\bigg( \log\Big(\frac{\|X_0\|}{u_n}\Big)\,\Big|\, \|X_0\|>u_n\bigg)-\frac 1\alpha=\ord\big((nv_n)^{-1/2}\big)
   $
\end{itemize}
\begin{theorem}
	\label{th:procconvunknownalpha}
	If the conditions (RV), (S), (BC), (TC), (C$\Theta$), (A), (BC'), (M), (H) and \eqref{eq:unifbiascond} are satisfied and $\log^4 n=\hbox{o}\left(nv_n\right)$, then
   \begin{equation}
       \sqrt{nv_n}\big(\pnAhh-P\{\Theta_i\in A\}\big)_{A\in\AA} \,\to\, \big(Z(A)-(P\{\Theta_i\in A\}-\alpha d_A)Z(\R^d)-\alpha^2 d_A\Zphi\big)_{A\in\AA}
   \end{equation}
   weakly. Here
   $$ d_A :=-\sum_{k\in\Z} E\big[\|\Theta\|_\alpha^{-\alpha}\|\Theta_k\|^\alpha\log \|\Theta_k\|\ind{A}(\Theta_{i})\big]\in\R
   $$
   and
   $\big((Z(A))_{A\in\AA},\Zphi\big)$ is a centered Gaussian process whose covariance function is given by \eqref{eq:covZ} and
   \begin{align}
     Cov(Z(A),\Zphi) & = \sum_{h\in\mathbb{Z}\sum_{k\in\mathbb{Z}}} E\bigg[ \frac{\|\Theta_h\|^\alpha}{\|\Theta\|_\alpha^\alpha} \ind{A}\Big(\frac{\Theta_{h+i}}{\|\Theta_h\|}\Big) \big(\|\Theta_k\|^\alpha\wedge 1\big)\Big(\alpha^{-1}+\log^+\|\Theta_k\|\Big)\bigg],\\
     Var(Z_\phi) & = \alpha^{-1}\sum_{k\in\mathbb{Z}} E\Big[\big(\|\Theta_k\|^\alpha\wedge 1\big)\big(2\alpha^{-1}+\big|\log\|\Theta_k\|\big|\big)\Big].
   \end{align}

\end{theorem}

\begin{rem} \label{rem:cond}
	\begin{itemize}
		\item[(a)] Condition (M) (ii) is equivalent to
		\begin{align}
			\label{eq:condMiiivar}
			\limsup_{m\to\infty}E\bigg(\frac{ \sum_{|h|\leq m} |\log\|Y_{h}\||^{1+\delta}\|Y_{h}\|^\alpha}{\sum_{|k|\leq m}\|Y_k\|^\alpha}\bigg)<\infty.
		\end{align}
		
		\item[(b)] Condition (M) (ii) is implied by the following condition, which may sometimes be easier to check:
	\begin{itemize}
		\item[\bf(M')] There exists $\delta'>\delta$ such that the following moment bounds hold: \begin{itemize}
			\item[(i)] $\displaystyle E\bigg( \sup_{|h|\leq s_n}\Big(\log^+\Big(\frac{\|X_{h}\|}{u_n}\Big)\Big)^{1+\delta'}\,\Big|\, \|X_0\|>u_n\bigg)=\Ord(1),$
		\item[(ii)] $\displaystyle \sum_{|h|\leq s_n}E\bigg( \Big(\log^-\Big(\frac{\|X_{h}\|}{u_n}\Big)\Big)^{1+\delta'}\Big(\frac{\|X_h\|}{u_n}\Big)^\alpha\,\Big|\, \|X_0\|>u_n\bigg)=\Ord(1).$
	\end{itemize}
	\end{itemize}
\nopagebreak
	Moreover, Condition (TC) is implied by (M') (ii); see Supplement for the proofs.
	\end{itemize}
\end{rem}

In general, neither $\pnAhh$ nor $\pnbh$ nor $\pnfh$ has uniformly smallest asymptotic variance, as is illustrated in the next example.
\begin{example} \label{ex:var_comp}
 Define a process $W=(W_t)_{t\in\mathbb{Z}}$  by
\begin{align}
	&P\{W_0=a^{-1},W_1=-1,W_t=0,\,\forall t\notin\{0,1\}\}=p,\\
	&P\{W_0=1,W_1=b^{-1},W_t=0,\,\forall t\notin\{0,1\}\}=1-p,
\end{align}
for some $a> 1$, $b> 1$ and $p=[0,1]$. Fix the index of regular variation to be  $\alpha=1$, and define $\Theta:=(\Theta_t)_{t\in\mathbb{Z}}$ such that $P^\Theta=(P^W)^{RS}$. Note that $\Theta$ has only  four possible realizations and that it vanishes for lags larger than 1 in absolute value. We say $W$ is the shape of $\Theta$, since standardizing a realization of $\Theta$ to sup-norm $1$ and shifting the first non-zero observation to time $0$ gives $W$.

We want to estimate the probability $P\{\Theta_i \in A\}$ for lags $i\in\{-1,0,1\}$ and half lines $A=(x,\infty)$ such that $P\{\Theta_i=x\}=0$. The asymptotic variance of $\pnAhh$ can be calculated using Theorem \ref{th:procconvunknownalpha} and the asymptotic variance of $\pnbh$ and $\pnfh$ can be obtained using Theorem 3.1 of \cite{davis2018}; see Supplement for details.

As shown in Figure \ref{figure:example.var}, where we fix $a=10, b=2$, each of the three estimators can have largest or smallest asymptotic variance depending on the model parameter $p$ (and $a,b$) and the probability we want to estimate (i.e.\ on $i$ and $x$). (Observe that for lag $i=0$ the backward estimator equals the forward estimator.)
For $p\to 0$ or $p\to 1$, the asymptotic variance of the estimator $\pnAh$ with known $\alpha$ (not shown in the plots) tends to 0 since the shape of the spectral tail process becomes deterministic. However, the asymptotic variance of $\pnAhh$ need not vanish due to the remaining variability of the Hill estimator.
\begin{figure}[h]
	\includegraphics[width=1\textwidth]{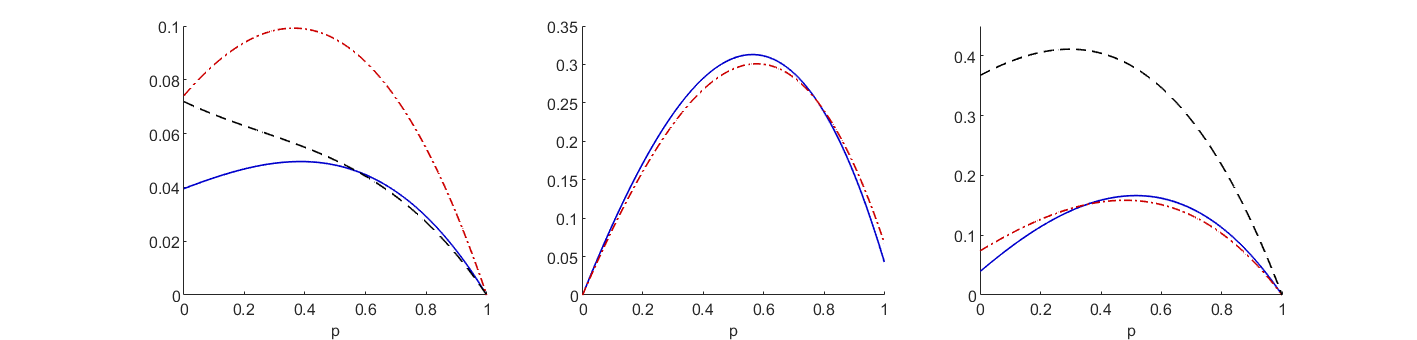}
	\caption{Variances of $\pnAhh$ (blue solid line), $\pnbh$ (black dashed line) and $\pnfh$ (red dashed-dotted line)   as a function of $p\in[0,1]$, for lag $i=-1$ and $A=(3/4,\infty)$ (left), $i=0$, $A=(1/20,\infty)$ (middle) and $i=1$,  $A=(1/20,\infty)$ (right).}
	\label{figure:example.var}
\end{figure}

\end{example}

\subsection{Solutions to stochastic recurrence equations}
\label{section:SRE}
In this subsection we discuss solutions to stochastic recurrence equations, which can e.g.\ be used to analyze GARCH time series (cf.\ \cite{basrak2002garch}). We derive conditions under which the previous theorems apply to this model.
   Consider a stationary solution to  the stochastic recurrence equation
   \begin{align} \label{eq:SREdef}
     X_t=C_t X_{t-1}+D_t, \quad t\in\Z,
   \end{align}
   where  $C_t$ are  random $d\times d$-matrices with non-negative entries and $D_t$ are $[0,\infty)^d$-valued random vectors such that $(C_t,D_t)$, $t\in\Z$, are i.i.d. In what follows, we denote both the Euclidean norm on $\R^d$ and the corresponding operator norm for $d\times d$-matrices by $\|\cdot\|$, and $(C,D)$ is a generic random variable with the same distribution as $(C_1,D_1)$.

   It is well known (\cite{BDM16}, Cor. 4.4.6) that under Condition (SRE) (see below) there exists an a.s.\ unique stationary solution to \eqref{eq:SREdef}, namely
   \begin{equation}
        X_t = \sum_{k=0}^\infty \Pi_{t-k+1,t}   D_{t-k}, \quad \Pi_{j,k}:=C_k\cdots C_j
   \end{equation}
   (with the convention $\Pi_{t+1,t}$ being equal to the identity matrix $I_d$)
   and this solution is regularly varying with index $\alpha$.
   \begin{itemize}
	\item[\bf(SRE)]
      \begin{itemize}
        \item[(i)] $\lim_{m\to\infty} m^{-1} E[\log\|\Pi_{1,m}\|]<0$;
        \item[(ii)] $\rho(s):=\lim_{m\to\infty} m^{-1}\log E[\|\Pi_{1,m}\|^s]\in (0,\infty)$ for some $s>0$;
        \item[(iii)] $E[\|C\|^\alpha\log^+\|C\|]<\infty$, $E[\|D\|^\alpha]<\infty$ with $\alpha$ denoting the unique solution to $\rho(\alpha)=0$;
        \item[(iv)] $P\{D=0\}<1$, $P\{(C_{jk})_{1\le k\le d} \ne 0\}=1, \;\forall\, 1\le j\le d$;
        \item[(v)] the additive subgroup generated by the logarithms of the spectral radii of arbitrary finite products of matrices in the support of $P^C$ is dense in $\R$.
   \end{itemize}
   \end{itemize}
   Observe that under conditions (i) and (ii) the existence of a unique zero of the function $\rho$ is guaranteed (\cite{BDM16}, Lemma 4.4.2). According to \cite{janssen2014}, Theorem 2.1 and Example 6.1, the forward spectral tail process  has the representation
         $\Theta_t=\Pi_{1,t}\Theta_0$, for all $t\ge 0$,
    with $\Theta_0$  independent of $(C_t)_{t\in\Z}$.

   We now discuss conditions that ensure that the limit Theorems \ref{th:fidiconv}, \ref{th:procconv} and \ref{th:procconvunknownalpha} are applicable for the family of sets $\AA=\{[0,x] \mid x\in[0,\infty)^d\}\cup \{[0,\infty)^d,\emptyset\}$, which is distribution determining for $\Theta_i$, $i\ge 0$. We are going to apply results for general Markovian time series established by \cite{KSW19}. To this end, we fix an arbitrary $q\in(0,\alpha)$ and verify Assumption 2.1 of this paper for $\mathbb{Y}:=X$, the functions $g(x)=\|x\|$ and $V(x)=\|x\|^q+1$ and $q_0=q$. Note that, in view of the above discussion, parts (i), (ii) and (v) of Assumption 2.1 are obviously fulfilled.

   If, in addition to (SRE), we assume the following set of conditions (SRE'), then, by Theorems 2.1 and 2.2 and Corollary 2.3 of \cite{Alsmeyer03} (see also \cite{BDM16}, Prop.\ 4.2.1) $(X_t)_{t\in\Z}$ is an aperiodic, positive Harris recurrent $P^{X_0}$-irreducible Feller process:
   \begin{itemize}
	\item[\bf(SRE')]
      \begin{itemize}
        \item[(i)] The interior of the support of $P^{X_0}$ is  not empty.
        \item[(ii)] There is a $\sigma$-finite non-null measure $\nu$ on $\R^d$ and an open set $E_0\in\B^d$ with $P\{X_0\in E_0\}>0$ such that $P^{Cx+D}$ has a non-trivial component that is absolutely continuous w.r.t.\ $\nu$ for all $x\in E_0$.
       \end{itemize}
    \end{itemize}
   Condition (ii) is e.g.\ fulfilled if $(C,D)$ is absolutely continuous w.r.t.\ the Lebesgue measure; see Lemma 4.2.2 of \cite{BDM16} for further sufficient conditions.
    By \cite{MT92}, Theorem 3.4, all compact subsets of $\R^d$ are petite sets, hence also small sets by Theorem 9.4.10 of \cite{douc2018}. Since $V^{-1}[0,u]$ is compact for all $u>0$, part (iv) of Assumption 2.1 of \cite{KSW19} is fulfilled with $m=1$.

   To verify part (iii), check that,  for all $t\ge 0$,
   \begin{equation} \label{eq:SRErec}
     X_t=\Pi_{1,t}X_0+R_t
   \end{equation}
    with $\Pi_{1,t}$ and $X_0$ independent, and $R_t:= \sum_{j=1}^t \Pi_{j+1,t}D_j$. Moreover,
   $\rho(q)<0$ for all $q\in (0,\alpha)$ by the H\"{o}lder inequality and $\rho(\alpha)=0$, so that
   \begin{equation} \label{eq:kappadef}
      \kappa:=E(\|\Pi_{1,m}\|^q)<1
   \end{equation}
    for some sufficiently large $m$. Since $(a+b)^q\le ((1+\eta)a)^q\Ind{b\le \eta a}+ ((1+\eta)b/\eta)^q \Ind{b> \eta a}$ for all $a,b\ge 0$ and $\eta>0$, we may conclude
    \begin{align}
      E(V(X_m)\mid X_0=y) & =  E\big[\|\Pi_{1,m}y+R_m\|^q\big]+1 \\
        & \le E\big[(\|\Pi_{1,m}\| \|y\|+\|R_m\|)^q\big]+1 \\
        & \le (1+\eta)^q \kappa (\|y\|^q+1) + ((1+\eta)/\eta)^q E[\|R_m\|^q]+1 \\
        & =: \tilde\beta V(y) + \tilde b  \label{eq:drift1}
    \end{align}
    with $\tilde \beta<1$ for $\eta$ sufficiently small. By \cite{douc2018}, Proposition 9.2.13 and Theorem 9.3.11, for suitable $M>0$, the set $V^{-1}[0,M]$ is accessible by the $m$-skeleton of the Markov chain.   Thus Theorem 14.1.4 of \cite{douc2018} shows that an analog to \eqref{eq:drift1} holds if $\tilde b$ is replaced with $b\ind{V^{-1}[0,\tilde M]}$ for some $\tilde M>0$ and possibly $\tilde \beta$ with some other constant $\beta\in(0,1)$. Since, by the same reasoning as above, $E(\|X_1\|^q\mid X_0=y)\le (1+\eta)^q E[\|C_1\|^q]\|y\|^q + ((1+\eta)/\eta)^q E[\|R_1\|^q]$, Lemma 2.1 of \cite{mikosch2014deviation} yields part (iii) of Assumption 2.1 of \cite{KSW19}.

    By regular variation of $\|X_0\|$, the remaining part (vi) of Assumption 2.1 of \cite{KSW19} follows if we show $E\big[\|X_0\|^q \Ind{\|X_0\|>u_n}\big] = O(u_n^q v_n)$. This, in turn, is an easy consequence of integration by parts and Karamata's theorem:
    \begin{align}
        E\big[\|X_0\|^q \Ind{\|X_0\|>u_n}\big] & = \int_{u_n}^\infty  z^q\, P^{\|X_0\|}(dz) \\
          & = -z^qP\{\|X_0\|>z\}\big|_{z=u_n}^\infty + q \int_{u_n}^\infty z^{q-1} P\{\|X_0\|>z\}\, dz \\
          & = u_n^q v_n + q \frac{u_n^q v_n}{\alpha-q}(1+\ord(1)).\label{eq:karamata}
    \end{align}

    The drift condition established above and irreducibility imply that the time series is geometrically $\beta$-mixing (see  \cite{KSW19} or \cite{mikosch2014deviation}, p.\ 161), i.e.\ there exist $\tau>0$ and $\sigma\in (0,1)$ such that $\beta_{k}\le \tau\sigma^k$ for all $k\in\N$. It is easy to check that condition (S) is met, provided $v_n=\ord(1/\log n)$ and $(\log^2 n)/n=o(v_n)$ and the block size is chosen such that $s_n=\ord\big(\min(v_n^{-1}, (nv_n)^{1/2})\big)$ (for instance, with $l_n\ge \max(s_n,\log n/|2\log\sigma|)$ and $r_n$ of larger order than $l_n$, but of smaller order than $\min(v_n^{-1}, (nv_n)^{1/2})$).

    Moreover, the following straightforward generalization of Lemma 4.3 of \cite{KSW19} holds: for all functions $\psi:\R^d\to \R$  that vanish on a neighborhood of 0 such that $|\psi(x)|\le c\|x\|^{q/2}$ for some $c>0$ and all $x\in\R^d$, one has
   \begin{equation} \label{eq:KSW_Lemma43}
     \lim_{L\to\infty} \limsup_{n\to\infty} \frac 1{v_n} \sum_{j=L+1}^{r_n} E[|\psi(X_0/u_n)\psi(X_j/u_n)|] = 0,
   \end{equation}
   provided $r_nv_n\to 0$.

    To verify Condition (BC), observe that the operator norm is sub-multiplicative and thus, with $\kappa$ defined in \eqref{eq:kappadef},
    \begin{align}
     E[\|\Pi_{1,k}\|^q] & \le \big( E[\|\Pi_{1,m}\|^q]\big)^{\floor{k/m}} E\big[\|\Pi_{1,k-m\floor{k/m}}\|^q\big]\\
      & \le \big(\kappa^{1/m}\big)^k \max_{0\le j<m} \frac{E[\|\Pi_{1,j}\|^q]}{\kappa^j} =: \tilde\kappa^k c_m,
    \end{align}
    for all $k\in\N$. Therefore, using \eqref{eq:SRErec}, one may first deduce
    \begin{equation} \label{eq:BC_SRE}
      P(\|X_h\|> cu_n\mid\|X_0\|>cu_n) \le P\Big\{\|X_0\|>\frac{cu_n}2\Big\} + 2^q c_m\tilde \kappa^h E\bigg( \Big(\frac{\|X_0\|}{cu_n}\Big)^q\,\Big|\, \|X_0\|>cu_n\bigg)
    \end{equation}
     and then
     (BC)  in the same way as condition (C) in \cite{drees2015}; see Supplement for details.

    According to Remark \ref{rem:cond} (b) it suffices to prove (M') (ii) instead of (TC). To this end we will need even stronger conditions on $v_n$ and $s_n$:
    	\begin{itemize}
    		\item[\bf(SRE'')] $v_n=\ord((\log n)^{-(3+\zeta)})$ for some $\zeta>0$, $(\log^2 n)/n=o(v_n)$,  $s_n=\ord\big(\min(v_n^{-1/(3+\zeta)},$ $ (nv_n)^{1/2})\big)$.
    \end{itemize}
	Choose some $q\in(\alpha/(1+\zeta/2),\alpha)$ and some $\tau\in(1/q,(1+\zeta/2)/\alpha)$, and  define $\eps_h := |h|^{-\tau}$ for $h\in\Z$, so that $(\eps_h^q)_{h\in\Z}$ is summable. Then, with $v_{n,c}:= P\{\|X_0\|>cu_n\}$,
	\begin{align}
		\sum_{h=-s_n}^{s_n} & E\bigg(\Big(\frac{\|X_h\|}{u_n}\Big)^q \Ind{\|X_h\|<u_n}\,\Big|\, \|X_0\|>u_n\bigg)\\
		& \le  \sum_{h=-s_n}^{s_n} \big[ P(\|X_h\|>\eps_h u_n\mid \|X_0\|>u_n)+ \eps_h^q\big]\\
		& \le 2 \sum_{h=1}^{s_n} P(\|X_h\|>\eps_h u_n\mid \|X_0\|>\eps_h u_n) \frac{v_{n,\eps_h}}{v_n} + \Ord(1). \label{eq:Mivbound}
	\end{align}
	By \eqref{eq:BC_SRE}, the $h$-th summand on the right hand side is bounded by
	\begin{align}
		\bigg[v_{n,\eps_h/2}+2^q c_m\tilde\kappa^h E\bigg( \Big(\frac{\|X_0\|}{\eps_h u_n}\Big)^q\,\Big|\, \|X_0\|>\eps_h u_n\bigg)\bigg]\frac{v_{n,\eps_h}}{v_n}.
	\end{align}
	Check that, for all $1\le h\le s_n$,  $\eps_h u_n\ge s_n^{-\tau}u_n$ is of larger order than $s_n^{-\tau} v_n^{-1/q}=(s_n v_n^{1/(\tau q)})^{-\tau}$ which tends to $\infty$ by condition (SRE'') because $\tau q<1+\zeta/2$. Therefore, the regular variation of $\|X_0\|$ combined with the Potter bounds shows that, for all $\alpha^*>\alpha$,
	\eqref{eq:Mivbound} is bounded by
	\begin{align}\label{eq:boundM'ii}
		4\sum_{h=1}^{s_n} \eps_h^{-\alpha^*}\bigg[\Big(\frac{\eps_h}2\Big)^{-\alpha^*}v_n + 2^q c_m\tilde\kappa^h E[\|Y_0\|^q]\bigg] +\Ord(1) = \Ord(s_n^{2\alpha^*\tau+1}v_n)+\Ord(1).
	\end{align}
	In view of (SRE''), this in turn is bounded if $\alpha^*$ is chosen such that $\alpha^*\tau\le 1+\zeta/2$. Now Condition (M') (ii) is obvious, since $|\log x|^{1+\delta}x^\alpha$ can be bounded by a multiple of $x^q$ for all $x\in (0,1)$.

    So far we have verified all assumptions of Theorem \ref{th:fidiconv} except for the bias condition \eqref{eq:biascondition.gA}, which is always fulfilled if $u_n$ is chosen sufficiently large, and Condition (C$\Theta$), which is obviously satisfied if $\Theta_i$ has continuous marginal distributions. In particular, if $i>0$ and $C$ is absolutely continuous, then this is also true for $\Theta_i$ and hence for its marginal distributions. More generally, our results apply to subsets of the family $\AA$ where all sets $(-\infty,y]$ are omitted for which $y_j$ belongs to a neighborhood of some jump point of the $j$-th marginal distribution for some $1\le j\le d$.

    For process convergence established in Theorem \ref{th:procconv}, in addition we need Condition (A). Parts (i) and (ii) are obvious with $ A_t:=\times_{i=1}^d [0,\varpi(t_i)]\cap\R^d$ for some continuous, increasing mapping $\varpi$ from $[0,1]$ onto $[-\infty,\infty]$.
   Then $\bigcap_{s\in(t,1]}A_{s^{(k)}}\setminus A_{t^{(k)}}=\emptyset$, so that (A) (v), (vi) and (vii) are trivial, and (iv) is equivalent to (C$\Theta$). Finally, part (iii) is fulfilled, because the processes are continuous from the right in each coordinate.

   To apply Theorem \ref{th:procconvunknownalpha} on the asymptotic normality of the estimator when $\alpha$ is unknown, we have to check the conditions (BC') and (M), while the bias condition (H) is again satisfied for $u_n$ sufficiently large. The former can be proved in the same way as condition (2.9) of \cite{knezevic2020} using \eqref{eq:KSW_Lemma43}. Likewise, Condition (M) (i) corresponds to (2.11) of \cite{knezevic2020} and it can be established by completely analogous arguments, provided $r_n^{1+\zeta}v_n\to 0$ for some $\zeta>0$, which holds by (SRE''); see Supplement for details.

   Since (M') (ii) has already been verified, in view of Remark \ref{rem:cond} (b), (M) (ii) follow if we establish (M') (i).
   Since, for all $\eta>0$, $\psi(x):=\max\big((\log^+x)^\eta,\ind{[1,\infty)}(x)\big)$ can be bounded by a multiple of $x^{q/2}$ and the time series is stationary,
   \begin{align}
    E\bigg(  \sup_{L<|h|\le s_n} \Big(\log^+ \frac{\|X_h\|}{u_n}\Big)^\eta \,\Big|\, \|X_0\|>u_n\bigg)
    & \le \frac 1{v_n} E\bigg[ \sup_{L<|h|\le s_n} \psi\Big(\frac{\|X_0\|}{u_n}\Big) \psi\Big(\frac{\|X_h\|}{u_n}\Big) \bigg] \\
    & \le  \frac{2}{v_n}\sum_{h=L+1}^{s_n}  E[\psi(\|X_0\|/u_n)\psi(\|X_h\|/u_n)] \\
    & \le 1
   \end{align}
   eventually if $L$ is chosen sufficiently large; cf.\ \eqref{eq:KSW_Lemma43}. Because $E \big((\log^+(\|X_h \|/u_n))^\eta\mid \|X_0\|>u_n\big)$ is bounded for all fixed $h$ due to the regular variation of the time series, the condition (M') (i) is obvious.

   Finally, (M) (iii) can be established  by the same arguments as (M') (ii). Since the denominator is at least $1$ and $\log^-(x)x^\alpha\leq  \tilde{c}_q x^q$ for $q<\alpha$ and some $\tilde{c}_q>0$, the boundedness of \eqref{eq:Mivbound} implies Condition (M) (iii). 

   To sum up, all limit theorems apply if the conditions (SRE), (SRE'), (SRE'') and the bias conditions \eqref{eq:unifbiascond} and (H) are met, and $\Theta_i$ have continuous marginal distributions.

\section{Simulations}
\label{section:simulations}

In this section we present the results of a Monte Carlo simulation study to evaluate the finite sample performance of the projection based estimator $\pnAhh$ with estimated $\alpha$. As competitors we consider the forward estimator $\pnfh$, see \eqref{Eq:forwardest}, and the backward estimator $\pnbh$, see \eqref{Eq:backwardest}.
We study their performance in two well-known real valued time series models:
\begin{itemize}
		
	\item GARCHt: Firstly, we consider  a GARCH(1,1) time series as a typical model for financial data: $X_t=\sigma_t \epsilon_t$ with $\sigma_t^2=0.1+0.14X_{t-1}^2+0.84\sigma_{t-1}^2$ and independent innovations $\epsilon_t$ with Student's $t_{4}$-distribution, standardized to unit variance.
	These parameters ensure that a stationary, regularly varying solution with index $\alpha=2.6$  exists, see \cite{mikosch2000}, Section 2.2, and \cite{davis2018}, Section 4. Its forward spectral tail process $(\Theta_t)_{t \geq 0}$ can be represented as follows:
\begin{equation}
	\Theta_t=\frac{\tilde{\epsilon}_t}{|\tilde{\epsilon}_0|}\prod_{i=1}^{t} (0.14\tilde{\epsilon}_{t-i}^2+0.84)^{1/2},\quad t\ge 0,
	\end{equation}
	where $\tilde{\epsilon}_h$, $h\geq 1$, are i.i.d.\ with a Student's $t_4$-distribution, standardized to unit variance; we denote the corresponding density by $f_\epsilon$. The random variable $\tilde{\epsilon}_0$ is independent of $(\tilde{\epsilon}_t)_{t \geq 1}$ with density $x\mapsto f_\epsilon(x)|x|^{2.6}/E[|\epsilon_1|^{2.6}]$, see Proposition 6.2 of \cite{ehlert2015}.
	\item SRE: The second model is defined as a stationary solution to the univariate stochastic recurrence equation $X_t=C_tX_{t-1}+D_t$, $t\in\Z$, with i.i.d.\ $\R^2$-valued $(C_t,D_t)$. We choose $C_t$ and $D_t$ to be independent with $C_t\sim\mathcal{N}(1/3,8/9)$ and $D_t\sim \mathcal{N}(-10,1)$; see \cite{drees2015}, Sections 5.2 and 6 for a discussion of this model. (We have chosen this model, which does not exactly fit in the setting of Section \ref{section:SRE}, to enable a comparison with this reference.) The unique  stationary solution is regularly varying with index $\alpha=2$. Moreover, the forward spectral tail processes is of the form
    \begin{equation}
	\Theta_t=\Theta_0\prod_{h=1}^{t}{C}_h, \quad t\ge 0,
	\end{equation}
where $\Theta_0$ is uniformly distributed on $\{-1,1\}$ and independent of  $(C_h)_{h\ge 1}$; see \cite{janssen2014} and \cite{goldie1991}, Theorem 4.1.	
\end{itemize}
We consider the family of sets $A=(-\infty,x]$, $x\in\R$; the corresponding probabilities describe the cdf of $\Theta_i$.
For all lags $i\in\{1,\ldots,10\}$, the cdf of $\Theta_i$ is approximately calculated on $[-2,2]$ (in steps of $0.01$)  via $10^7$ Monte Carlo repetitions. For $1000$ time series of length $2000$ we calculate the three estimates of $P\{\Theta_i\leq x\}$, $x\in[-2,2]$. In each simulation, we choose the empirical $95\%$ quantile of the absolute values as threshold $u_n$.

In addition, for the projection based estimator $\pnAhh$ one has to choose the block length, determined by $s_n$. If $s_n$ is chosen too small, then the estimated distribution of $\Theta_i$ has substantial mass at 0, an artefact caused by the cut off of the time series at lags larger than $s_n$ in absolute value   in \eqref{eq:def_emp_est}. In the results presented below, we have chosen $s_n=30$, but in most cases the performance of the projection based estimator turned out to be quite stable for a wide range of values of $s_n$ including the interval  $\{20, \ldots, 50\}$; see Supplement for details.
\begin{figure}[b]
	\includegraphics[width=1\textwidth]{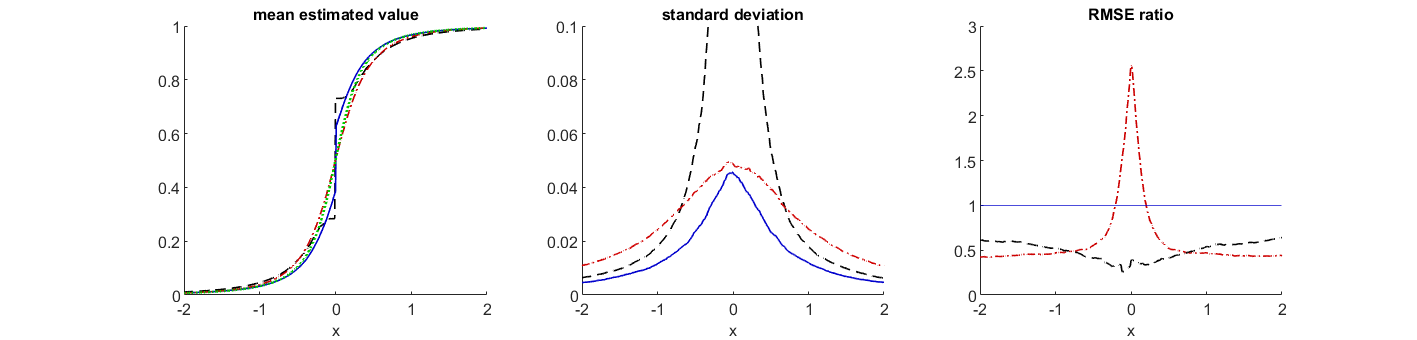}
	\includegraphics[width=1\textwidth]{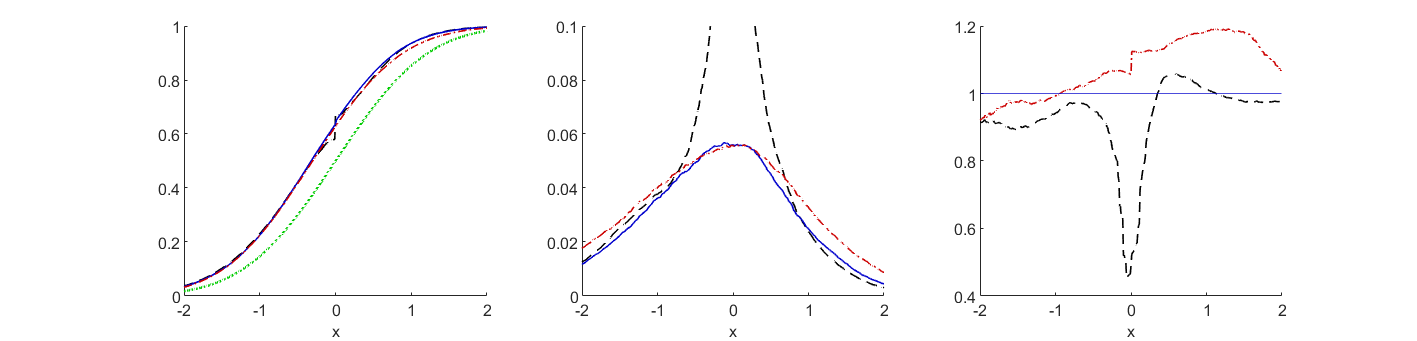}
\caption{Mean (left), standard deviation (middle) and relative efficiency w.r.t.\  $\pnAhh$ (right), of $\pnAhh$ (blue solid line), $\pnbh$ (black dashed line) and $\pnfh$ (red dashed-dotted line) for
GARCHt model and lag $i=10$ (top) and SRE model and lag $i=1$ (bottom); the true cdf is indicated by the green dotted line.}
	\label{figure:pb.1.95}
\end{figure}

First we examine the GARCHt model for lag $i=10$, where the Hill type estimator of $\alpha$ exhibits a bias of about $0.014$ and a standard deviation of about $0.461$. The top row of Figure \ref{figure:pb.1.95} shows mean and standard deviation of all three estimators.  In addition, the relative efficiencies of the forward and backward estimator w.r.t.\ the newly proposed estimator, i.e.\ the ratio between the root mean squared error (RMSE) of the respective estimator and the RMSE of $\pnAhh$, is shown in the right column.
By construction, both the projection based estimator and the backward estimator of the distribution of $\Theta_i$ typically have mass at $0$, leading to a bias of the estimated cdf at small values of $|x|$ and thus to an underperformance w.r.t.\ the forward estimator for $|x|<0.2$. In contrast,  for larger values of $|x|$ the relative efficiency of the forward estimator w.r.t.\ $\pnAhh$ is less than $1/2$, and the backward estimator is outperformed by the newly proposed estimator for all $x$.


The bottom row of Figure \ref{figure:pb.1.95} shows the results for the SRE model and lag $i=1$, when the estimator of $\alpha$ has bias $\approx 0.20$ and standard deviation $\approx 0.385$. For all three estimators we observe a strong bias caused by $D_t$, which adds a negative drift to the time series on finite levels, while it vanishes in the limit. The standard deviation of $\pnAhh$ is smaller than or equal to that of $\pnfh$, but as the RMSE is dominated by the bias (which is somewhat smaller for $\pnfh$), the relative efficiency of the forward estimator is larger than 1 for most values of $x$. In contrast, the backward estimator, with its huge standard deviation near the origin, is less efficient than the projection based estimator for most $x$. If one evaluates the performance of $\pnAhh$, $\pnfh$ and $\pnbh$ as estimators of the ``pre-asymptotic'' probabilities $P(X_i/|X_0|\le x\mid |X_0|>u_n)$ (as it has been done in \cite{davis2018}), then the bias almost vanishes and for most values of $x$ the projection based estimator outperforms the forward estimator; see Supplement for details.

Further simulation results for other lags, thresholds  and models can be found in the Supplement. Overall, the projection based estimator tends to have the smallest variance of the three estimators, in particular for higher lags $i$ and larger values of $|x|$. Usually it shows the most robust behavior which combines the good performance of the forward estimators for small values of $|x|$ with the improvements of the backward estimator for large $|x|$.

\appendix
\section{Proofs}
\subsection{Proof of Theorem \ref{th:fidiconv}.}

First we prove weak convergence of the empirical processes $(\Ztn(A))_{A\in\mathcal{A}\cup \{\mathbb{R}^d\}}$ associated with the sums $T_{n,A}$ of sliding blocks statistics defined by
\begin{align}
\Ztn(A)&=\frac{1}{\sqrt{nv_n}}\left(T_{n,A}-E[T_{n,A}]\right).
\end{align}

\begin{proposition}
	\label{prop:asymptotic.of.tng}
	If the conditions (RV), (S), (BC), (TC) and (C$\Theta$) are satisfied, then the fidis of $\Ztn$ weakly converge to the fidis of the centered Gaussian process $Z$ defined in Theorem \ref{th:fidiconv}.
\end{proposition}

	To prove this proposition, we want to apply Theorem 2.1 of \cite{DN20} (abbreviated to \DN\, in what follows) to empirical processes based on the transformed observations $X'_{n,t}:=X_{t-s_n}/u_n$, $1\le t\le n':=n+2s_n$, and block lengths $s'_n:=2s_n+1$ and $l_n':=2s_n'-1$. 	We will consider the family  $\mathcal{G}=\left\{g_A| A\in \mathcal{A}\cup \{\mathbb{R}^d\}\right\}$ of bounded functions $g_A$ taking values in $[0,1]$, and  use the normalization $b_n(g)=(nv_n/p_n)^{1/2}$ for all $g\in\mathcal{G}$ with $p_n=P\{\max_{1\leq t\leq r_n} \|X_t\|>u_n\}$.

We start with a verification of the convergence of the normalized covariances, i.e.\ condition (C) of \DN. To this end, we need some technical lemmas. The proof of the first lemma is given in the Supplement.

\begin{lemma} \label{lemma:AC}
  If (RV) and (BC) holds, then the so-called anti clustering condition
  \begin{align}
	\label{eq:anticlustering bedingung}
	\lim_{m\to\infty}\limsup_{n\to\infty}P\Big(\max_{m\leq |t|\leq r_n}\|X_t\|>cu_n \,\Big|\, \|X_0\|>cu_n \Big)=0
	\end{align}
	is satisfied for all $c\in(0,\infty)$.
\end{lemma}
\begin{lemma}  \label{lemma:umformulierung.e3}
If (RV) holds, then condition (C$\Theta$) is satisfied if and only if
$$P\left\{\exists t\in\mathbb{Z}: \frac{Y_{t+i}}{\|Y_{t}\|}\in \partial A, \|Y_t\|>0\right\}=0, \qquad \forall\, A\in\mathcal{A}.$$
\end{lemma}

\begin{proof}
    Since $P^\Theta=(P^\Theta)^{RS}$ and $\Theta_t=Y_t/\|Y_0\|$, $P\{\Theta_{i}\in \partial A\}=0$ holds if and only if
    \begin{equation}
	0=P\left\{ \sum_{t\in\mathbb{Z}} \frac{\|\Theta_t\|^{\alpha}}{\|\Theta\|_{\alpha}^{\alpha}} \mathds{1}_{\partial A}\left(\frac{\Theta_{t+i}}{\|\Theta_t\|}\right)>0 \right\}
	=P\left\{\exists t\in\mathbb{Z}: \frac{Y_{t+i}}{\|Y_{t}\|}\in \partial A, \|Y_t\|>0\right\}. \qedhere
	\end{equation}
\end{proof}

In the following proofs we consider functions on the sequence space $l_\alpha=\{(x_t)_{t\in\mathbb{Z}}\mid \sum_{t\in\Z}\|x_t\|^\alpha<\infty \}$ (equipped with the supremum norm). For arbitrary $r\in\N$, $\mathbb{R}^{2r+1}$ is embedded in $l_\alpha$ by the mapping $(z_i)_{|i|\leq r}\mapsto (z_t)_{t\in\mathbb{Z}}$ with ${z}_t:=0$ for $|t|> r$. Note that \eqref{eq:anticlustering bedingung}  ensures that the realizations of the spectral tail process a.s.\ belong to $l_\alpha$ (cf.\ Remark 2.3 of \cite{janssen2019}).
The following result establishes condition (C) of D{\&}N.
\begin{lemma}
	\label{lemma:covarianz konkreter schaetzer}
	If the conditions (RV), (S), (BC), (TC) and (C$\Theta$) are satisfied, then
	\begin{equation}
	\frac{1}{r_nv_n}Cov\bigg(\sum_{t=1}^{r_n} g_{A}(W_{n,t}),\sum_{t=1}^{r_n} g_{B}(W_{n,t})\bigg)\to c(A,B),\quad \forall\, A,B\in\mathcal{A}\cup \{\mathbb{R}^d\},
	\end{equation}
	with $c$ defined in Theorem \ref{th:fidiconv}.
\end{lemma}

\begin{proof}
	 First note that
	\begin{align}
	Cov&\bigg(\sum_{t=1}^{r_n} g_{A}(W_{n,t}),\sum_{t=1}^{r_n} g_{B}(W_{n,t})\bigg)\\
	&= E\bigg[ \sum_{t=1}^{r_n}\sum_{s=1}^{r_n} g_{A}(W_{n,t})g_{B}(W_{n,s})\bigg]
	- E\bigg[ \sum_{t=1}^{r_n} g_{A}(W_{n,t})\bigg]E\bigg[\sum_{t=1}^{r_n}g_{B}(W_{n,t})\bigg]=:I+II,
    \label{eq:cov konkreter schaetzer zweite summe}
	\end{align}
	where
	\begin{align}  \label{eq:termIInegligible}
	\frac{|II|}{r_nv_n}\leq \frac 1{r_nv_n} \bigg(E\Big[\sum_{t=1}^{r_n}\mathds{1}_{\left\{\|X_t\|>u_n\right\}} \Big]\bigg)^2 = r_nv_n \to 0.
	\end{align}
	By stationarity
	\begin{align}
	I & = \sum_{s=1}^{r_n} \sum_{t=1}^{r_n} E\big[ g_A(W_{n,s}) g_B(W_{n,t})\big] \\
    & = \sum_{j=-r_n}^{r_n} (r_n-|j|)  E\big[ g_A(W_{n,j}) g_B(W_{n,0})\big] \\
     & = r_nv_n \sum_{j=-r_n}^{r_n} \Big(1-\frac{|j|}{r_n}\Big) E\Bigg( \mathds{1}_{\{\|X_j\|>u_n\}}\\
    & \hspace{1.3cm} \times  \bigg(\sum_{h\in H_{n}}\frac{\|X_{j+h}\|^{\alpha}}{\sum_{k=-s_n}^{s_n} \|X_{j+k}\|^{\alpha }} \mathds{1}_{  A}\Big(\frac{X_{j+h+i}}{\|X_{j+h}\|}\Big)
    + \sum_{h\in H_{n}^c}\frac{\|X_{j+h}\|^{\alpha}}{\sum_{k=-s_n}^{s_n} \|X_{j+k}\|^{\alpha }} \mathds{1}_{A}(0)\bigg)\\
	& \hspace{0.8cm}\times \bigg(\sum_{l\in H_n}\frac{\|X_{l}\|^{\alpha}}{\sum_{k=-s_n}^{s_n} \|X_{k}\|^{\alpha }}\mathds{1}_{  B}\Big(\frac{X_{l+i}}{\|X_{l}\|}\Big)+ \sum_{l\in H_n^c}\frac{\|X_{l}\|^{\alpha}}{\sum_{k=-s_n}^{s_n} \|X_{k}\|^{\alpha }}\mathds{1}_{B}(0)\bigg)  \,\bigg|\, \|X_0\|> u_n \Bigg).
  \label{eq:cov konkreter schaetzer erste summe}
	\end{align}

 Define a function $f=f_{A,B}:l_{\alpha}\times l_{\alpha}\to[0,1]$ by
	\begin{align}
	f\big(&(y_t)_{t\in\mathbb{Z}},(z_{t})_{t\in\mathbb{Z}}\big) := \mathds{1}_{\{\|y_0\|>1\}}\mathds{1}_{\{\|z_0\|>1\}} \\
 & \hspace{2cm}\times \Bigg(\sum_{h\in\mathbb{Z}}\frac{\|z_{h}\|^{\alpha}}{\sum_{k\in\mathbb{Z}} \|z_{k}\|^{\alpha }} \mathds{1}_{A}\Big(\frac{z_{h+i}}{\|z_{h}\|}\Big) \bigg) \bigg(\sum_{l\in\mathbb{Z}}\frac{\|y_{l}\|^{\alpha}}{\sum_{k\in\mathbb{Z}} \|y_{k}\|^{\alpha }} \mathds{1}_{B}\Big(\frac{y_{l+i}}{\|y_{l}\|}\Big)\Bigg),
\label{eq:continuous.mapping.function.for.gA}
	\end{align}
with the convention $0/0:=0$.
Using Condition (TC), one may replace sums over $\{-s_n,\ldots,s_n\}$ with sums over $\{-m,\ldots,m\}$ to conclude from \eqref{eq:tailprocess} and (C$\Theta$) that
   \begin{align}\label{eq:expecconv.f}
   	E\big(f\big((X_t/u_n)_{-s_n\le t\le s_n},(X_{t+j}/u_n)_{-s_n\le t\le s_n}\big)\mid \|X_0\|>u_n\big) \to E\big[f\big((Y_t)_{t\in\Z},(Y_{t+j})_{t\in\Z}\big)\big];
   \end{align}
see Supplement for details.
  Moreover,  condition (BC) implies
  \begin{align}
 E\big(f\big((X_t/u_n)_{-s_n\le t\le s_n},(X_{t+j}/u_n)_{-s_n\le t\le s_n}\big)\mid \|X_0\|>u_n\big) & \le E\big(\Ind{\|X_j\|>u_n}\mid \|X_0\|>u_n\big)\\
 &\le e_{n,1}(|j|).
  \end{align}
	Hence, using Pratt's lemma \citep{pratt1960}, we may conclude from \eqref{eq:cov konkreter schaetzer erste summe} and \eqref{eq:expecconv.f} that
	\begin{align}
	\frac{1}{r_nv_n}I&=\sum_{j=-r_n}^{r_n} \Big(1-\frac{|j|}{r_n}\Big)  E\big(f\big((X_t/u_n)_{-s_n\le t\le s_n},(X_{t+j}/u_n)_{-s_n\le t\le s_n}\big)\mid \|X_0\|>u_n\big)\\
	&\rightarrow \sum_{j\in\Z} E\big[f\big((Y_t)_{t\in\Z},(Y_{t+j})_{t\in\Z}\big)\big].
	\end{align}
	
	To sum up, we have shown that
	\begin{align}
	\frac{1}{r_nv_n}&Cov\bigg(\sum_{t=1}^{r_n} g_{A}(W_{n,t}),\sum_{t=1}^{r_n} g_{B}(W_{n,t})\bigg)\\
    &\rightarrow \sum_{j\in\Z}E\bigg[  \mathds{1}_{\{\|Y_j\|>1\}} \bigg(\sum_{h\in\mathbb{Z}}\frac{\|Y_{j+h}\|^{\alpha}}{\|Y\|_\alpha^{\alpha }} \mathds{1}_{A}\Big(\frac{Y_{j+h+i}}{\|Y_{j+h}\|}\Big) \bigg) \bigg(\sum_{l\in\mathbb{Z}}\frac{\|Y_{l}\|^{\alpha}}{\|Y\|_\alpha^{\alpha }} \mathds{1}_{ B}\Big(\frac{Y_{l+i}}{\|Y_{l}\|}\Big)\bigg)\bigg]\\
	&=
	\sum_{j\in\mathbb{Z}}E\bigg[  \mathds{1}_{\{\|Y_j\|>1\}} \bigg(\sum_{k\in\mathbb{Z}}\frac{\|Y_{k}\|^{\alpha}}{\|Y\|_\alpha^{\alpha }} \mathds{1}_{A}\Big(\frac{Y_{k+i}}{\|Y_{k}\|}\Big) \bigg) \bigg(\sum_{l\in\mathbb{Z}}\frac{\|Y_{l}\|^{\alpha}}{\|Y\|_\alpha^{\alpha }} \mathds{1}_{B}\Big(\frac{Y_{l+i}}{\|Y_{l}\|}\Big)\bigg)\bigg].
\label{eq:covlimitrep1}
	\end{align}
	Since $Y_t=\Theta_t\|Y_0\|$ with $(\Theta_t)_{t\in\Z}$ and $\|Y_0\|$ independent and $P\{\|Y_0\|>x\}=x^{-\alpha}\wedge 1$, the limit can be rewritten as
	\begin{align}	
	& \sum_{j\in\Z} E\bigg[  \mathds{1}_{\{\|\Theta_j\|\|Y_0\|>1\}} \bigg(\sum_{k\in\mathbb{Z}}\frac{\|\Theta_{k}\|^{\alpha}}{\|\Theta\|_{\alpha}^{\alpha }} \mathds{1}_{B}\Big(\frac{\Theta_{k+i}}{\|\Theta_{k}\|}\Big) \bigg) \bigg(\sum_{l\in\mathbb{Z}}\frac{\|\Theta_{l}\|^{\alpha}}{\|\Theta\|_{\alpha}^{\alpha }} \mathds{1}_{ A}\Big(\frac{\Theta_{l+i}}{\|\Theta_{l}\|}\Big)\bigg)\bigg]\\
	&= \sum_{j\in\Z}\int_{l_\alpha} P{\{\|Y_0\|>\|\theta_j\|^{-1}\}} \bigg(\sum_{k\in\mathbb{Z}}\frac{\|\theta_{k}\|^{\alpha}}{\|\theta\|_{\alpha}^{\alpha }} \mathds{1}_{B}\Big(\frac{\theta_{k+i}}{\|\theta_{k}\|}\Big) \bigg)  \Big(\sum_{l\in\mathbb{Z}}\frac{\|\theta_{l}\|^{\alpha}}{\|\theta\|_{\alpha}^{\alpha }} \mathds{1}_{ A}\Big(\frac{\theta_{l+i}}{\|\theta_{l}\|}\Big)\bigg) \,P^{\Theta}(d\theta)\\
	&=\sum_{j\in\Z} E\bigg[  ( \|\Theta_j\|^{\alpha}\wedge 1) \bigg(\sum_{k\in\mathbb{Z}}\frac{\|\Theta_{k}\|^{\alpha}}{\|\Theta\|_{\alpha}^{\alpha }} \mathds{1}_{B}\Big(\frac{\Theta_{k+i}}{\|\Theta_{k}\|}\Big) \bigg) \bigg(\sum_{l\in\mathbb{Z}}\frac{\|\Theta_{l}\|^{\alpha}}{\|\Theta\|_{\alpha}^{\alpha }} \mathds{1}_{A}\Big(\frac{\Theta_{l+i}}{\|\Theta_{l}\|}\Big)\bigg)\bigg]. &  \qedhere
	\end{align}
\end{proof}

\begin{proof}[Proof of Proposition \ref{prop:asymptotic.of.tng}]
 As mentioned above, the assertion follows readily if we verify the conditions of Theorem 2.1 of \DN. Note that $p_n$ is of the same order as $r_nv_n$ according to Proposition 4.2 of \cite{basrak2009}, since (BC) implies the anti clustering condition \eqref{eq:anticlustering bedingung} (which in fact just controls the cluster size).
	Hence the conditions (A1), (A2) and (MX) of \DN\, immediately follow from our condition (S) and stationarity. Moreover, condition
	(D0) of \DN\, is trivially fulfilled for finite families of sets $A\in\AA$, which suffices here.
	
	As $\mathds{1}_{\{g(W_{n,j})\neq 0 \}}\leq \mathds{1}_{\{ \|X_{j}\|> u_n \}}=\mathds{1}_{\{g_{\R^d}(W_{n,j})\neq 0 \}}$ for all $j=1,\ldots,r_n$ and all $g\in\mathcal{G}$, condition (2.6) of \DN\ is an obvious consequence of Lemma \ref{lemma:covarianz konkreter schaetzer} which yields
	\begin{equation}
	\label{eq:bed.3.15.forTnga.verifikation}
	N_n:=\frac{1}{r_nv_n}E\bigg[\Big(\sum_{j=1}^{r_n} \mathds{1}_{\{ \|X_{j}\|> u_n \}}\Big)^2  \bigg]
	=\Ord(1).
	\end{equation}
	Since the remaining  condition (C) of \DN\ has been proved in Lemma \ref{lemma:covarianz konkreter schaetzer}, the assertion follows.
\end{proof}

\begin{proof}[Proof of Theorem \ref{th:fidiconv}]
	Since $\pnAh=T_{n,A}/T_{n,\mathbb{R}^d} $ and $E[T_{n,A}]=nv_nE(g_A(W_{n,0})\mid \|X_0\|>u_n)$, direct calculations show that
	\begin{align}
	&\sqrt{nv_n}\big(\pnAh-E(g_A(W_{n,0})\mid \|X_0\|>u_n)\big) \\
    & =\sqrt{nv_n} \bigg(\frac{\sqrt{nv_n}\Ztn(A)+nv_nE(g_A(W_{n,0})\mid \|X_0\|>u_n)}{\sqrt{nv_n}\Ztn(\R^d)+nv_n }-E(g_A(W_{n,0})\mid \|X_0\|>u_n)\bigg)\\
	&=\frac{\Ztn(A)-E(g_A(W_{n,0})\mid \|X_0\|>u_n)\Ztn(\mathbb{R}^d)}{1+(nv_n)^{-1/2}\Ztn(\mathbb{R}^d) }\\
& \rightarrow Z(A)-\pA Z(\mathbb{R}^d)
	=\Zp (A)
	\end{align}
uniformly over any finite family of sets $A\in\AA$, where in the last step we have used Proposition \ref{prop:asymptotic.of.tng},  \eqref{eq:ewertconv} and $nv_n\to\infty$.
\end{proof}

\subsection{Proof of Theorem \ref{th:procconv}.}

	The assertion can be concluded from the following proposition by the continuous mapping argument used in the proof of Theorem \ref{th:fidiconv}.

\begin{proposition}
	\label{prop:procconvt}
	If the conditions (RV), (S), (BC), (TC), (C$\Theta$) and (A) are fulfilled, then $\Ztn$ weakly converge to the Gaussian process $Z$ defined in Theorem \ref{th:fidiconv}.
\end{proposition}

\begin{proof}
	We are going to apply Theorem 2.3 (i) of \cite{DN20} (D\&N). Since condition (D0) of \DN\ follows immediately from condition (A) (iii), and the other conditions used in Theorem 2.1 of \DN\ have already been verified in the proof of Proposition \ref{prop:asymptotic.of.tng}, only the conditions (D1) and (D2) of \DN\ remain to be shown  for some semi-metric on $\GG$ such that $\mathcal{G}$ is totally bounded. 	
	
	First we verify that $\mathcal{G}$ is totally bounded w.r.t. $\tilde{\rho}(g_A,g_B):=\rho (A,B)$, where the semi-metric $\rho$ on $\mathcal{A}$ is given by
	\begin{align}
	\rho(A,B)&:=\sum_{j\in\mathbb{Z}}E\bigg[  \mathds{1}_{\left\{\|Y_j\|>1\right\}} \sum_{h\in\mathbb{Z}}\frac{\|Y_{h}\|^{\alpha}}{\|Y\|_{\alpha}^{\alpha }} \bigg|\mathds{1}_{ A}\left(\frac{Y_{h+i}}{\|Y_{h}\|}\right)-\mathds{1}_{B}\left(\frac{Y_{h+i}}{\|Y_{h}\|}\right) \bigg| \bigg].
	\end{align}
Note that $\rho(A,B)=\rho(A\backslash B,\emptyset)$ for $B\subset A$
and $\rho(\bigcup_{n\in\mathbb{N}}B_n,\emptyset)=\sum_{n\in\mathbb{N}}\rho(B_n,\emptyset)$ for disjoint sets $B_n\in \mathbb{B}^d$, $n\in\mathbb{N}$.
	
	Fix $\delta>0$ and, for the time being, $k\in\{1,\ldots,q\}$. Recall that, for $t\in[0,1+\iota]$, $t^{(k)}\in[0,1]^q$ denotes the vector with $k$-th coordinate equal to $t\wedge 1$ and all other coordinates equal to 1. By condition (A) (ii) the mapping $H_k:t\mapsto \rho(A_{t^{(k)}},\emptyset)$ is non-decreasing. Observe that the arguments used in the proof of Lemma \ref{lemma:umformulierung.e3} applied to condition (A)~(v) show that $P\{\exists h\in\Z: Y_{h+i}/\|Y_h\|\in \bigcap_{s\in(t,1]}A_{s^{(k)}}\setminus A_{t^{(k)}}\}=0$ for all $t\in[0,1)$. By the monotone convergence theorem, this implies that $H_k$ is right-continuous and thus it is the measure generating function of a measure on $[0,1]$, which is finite because $\rho(A_1,\emptyset) =c(\mathbb{R}^d,\mathbb{R}^d)<\infty $ by Lemma \ref{lemma:covarianz konkreter schaetzer}. Thus, there exist $J_k<\infty$ and a partition $\mathcal{T}_k:=\big\{[s_{k,j-1},s_{k,j})\mid 1\le j\le J_k\big\}\cup\{\{1\}\}$
 of $[0,1]$ (with $s_{k,0}=0$ and $s_{k,J_k}=1$) such that $H_k(s_{k,j}-)-H_k(s_{k,j-1})\le\delta$ for all $1\le j\le J_k$. (For instance, one may define the interval boundaries iteratively by
 $ s_{k,j}=\inf\big\{t\in (s_{k,j-1},1]\mid \rho\big(A_{t^{(k)}},A_{s_{k,j-1}^{(k)}}\big)> \delta\big\}$.
  Observe that although it is possible that the measure pertaining to $H_k$ has mass greater than $\delta$ at some of the boundary points $s_{k,j}$, such jumps of $H_k$ do not play any role in  the following calculations.)

  We now consider the finite cover of $[0,1]^q$ by sets in
   $$\mathcal{T}^{(\delta)}:=\{\times_{k=1}^q T_k\mid T_k\in\mathcal{T}_k,\;\forall\,1\le k\le q\}.$$
   For any set $T=\times_{k=1}^q T_k\in \mathcal{T}^{(\delta)}$ define $K_T:=\big\{1\le k\le q\mid T_k\ne\{1\}\big\}$,
   \begin{equation} \label{eq:setboundsdef}
     \underline{A}_T := \bigcap_{s\in T} A_s = A_{(\min T_1,\ldots, \min T_q)}\in\AA \quad \text{and} \quad \bar A_T := \bigcup_{s\in T} A_s\in\tilde\AA:=\{A_t^- \mid t\in[0,1+\iota]^q\}.
   \end{equation}
   Check that $\bar A_T\setminus\underline{A}_T \subset \bigcup_{k\in K_T} \big( \bigcup_{s\in T_k} A_{s^{(k)}} \setminus A_{(\min T_k)^{(k)}}\big)$. Hence, by construction of $\mathcal{T}^{(\delta)}$, for all $t\in T$,
 	\begin{align}  \label{eq:rhobound1}
	\rho\big(A_t,\underline{A}_T \big) &\leq \rho\big( \bar A_T\setminus \underline{A}_T,\emptyset \big)
\le \sum_{k\in K_T} H_k(\sup T_k-)-H_k(\min T_k) \le q\delta.
	\end{align}
   Therefore, all sets of the form $\{A_t|t\in T\}$ have a radius of at most $q\delta$ w.r.t.\ $\rho$, which shows that
    $\AA$ is totally bounded w.r.t.\ $\rho$ and thus $\GG$ is totally bounded w.r.t.\ $\tilde\rho$.

	Next we turn to the continuity condition (D1) of \DN. Let $V_{n}(g_A):=\sum_{t=1}^{r_n}g_A(W_{n,t})$. According to \eqref{eq:termIInegligible} and \eqref{eq:covlimitrep1}, one has
	\begin{align}
	&\frac{1}{r_nv_n}E\left[ ( V_n(g_A)- V_n(g_B))^2\right]\\
	&\rightarrow   \sum_{j\in\mathbb{Z}}E\bigg[  \mathds{1}_{\{\|Y_j\|>1\}}  \bigg(\sum_{l\in\mathbb{Z}}\frac{\|Y_{l}\|^{\alpha}}{\|Y\|_{\alpha}^{\alpha }} \left(\mathds{1}_{ A}\left(\frac{Y_{l+i}}{\|Y_{l}\|}\right)-\mathds{1}_{B}\left(\frac{Y_{l+i}}{\|Y_{l}\|}\right)\right)\bigg)^2\bigg]\\
	&\leq  \sum_{j\in\mathbb{Z}}E\Bigg[  \mathds{1}_{\left\{\|Y_j\|>1\right\}}  \sum_{l\in\mathbb{Z}}\frac{\|Y_{l}\|^{\alpha}}{\|Y\|_{\alpha}^{\alpha }} \bigg|\mathds{1}_{ A}\left(\frac{Y_{l+i}}{\|Y_{l}\|}\right)-\mathds{1}_{B}\left(\frac{Y_{l+i}}{\|Y_{l}\|}\right)\bigg|\Bigg]
	=\rho(A,B)
	\end{align}
	 for all $A,B\in\mathcal{A}$. In view of condition (A) (iv), one can prove by the same arguments that this convergence even holds for all $A,B\in\AA\cup\tilde\AA$. In particular, for all fixed $\delta,\eps>0$ one has eventually
\begin{align} \label{eq:bound.for.d1}
	\frac{1}{r_nv_n}E\left[ \big( V_n(g_{\bar A_T})- V_n(g_{\underline{A}_S})\big)^2\right] \le \rho(\bar A_T,\underline{A}_S)+\frac{\eps}2
\end{align}	
	uniformly for all sets $S,T\in\mathcal{T}^{(\delta)}$.

Now consider two sets $A_s,A_t\in\mathcal{A}$ with $\rho(A_s,A_t)\le\delta$.
Because  $\mathcal{T}^{(\delta)}$ defines a partition of $[0,1]^q$, there exist unique sets $S,T\in\mathcal{T}^{(\delta)}$ such that $s\in S$ and $t\in T$.
Then, obviously $\underline{A}_T\subset A_t\subset \bar A_T$ and $\underline{A}_S\subset A_s\subset \bar A_S$. Hence, we may conclude from \eqref{eq:bound.for.d1} and \eqref{eq:rhobound1} that eventually
	\begin{align}
	\frac{1}{r_nv_n}& E\left[ ( V_n(g_{A_t})- V_n(g_{A_s}))^2\right]\\
& \le \frac{1}{r_nv_n}E\left[ \max\big(( V_n(g_{\bar A_T})- V_n(g_{\underline{A}_S}))^2, ( V_n(g_{\underline{A}_T})- V_n(g_{\bar{A}_S}))^2 \big) \right]\\
	&\leq \frac{1}{r_nv_n}E\left[ ( V_n(g_{\bar A_T}) - V_n(g_{\underline{A}_S}))^2\right] + \frac{1}{r_nv_n}E\left[ ( V_n(g_{\underline{A}_T})- V_n(g_{\bar{A}_S}))^2\right]\\
	&\leq \rho(\bar A_T,\underline{A}_S) + \rho(\underline{A}_T,\bar{A}_S) + \eps\\
	&\leq \rho(\bar A_T,A_t)+\rho(A_t,A_s)+\rho(A_s,\underline{A}_S) + \rho(\underline{A}_T,A_t) +\rho(A_t,A_s) +\rho(A_s,\bar{A}_S)+  \eps\\
    &\leq 2\big(\rho(\bar A_T,\underline{A}_T)+\rho(\bar A_S,\underline{A}_S)+ \rho(A_t,A_s)\big)+  \eps\\
    & \le (4q+2)\delta + \eps.
  \end{align}
	Therefore,
	\begin{align}
	&\lim_{\delta\downarrow 0}\limsup_{n\to\infty} \sup_{s,t\in[0,1]^q:\rho(A_s,A_t)\le\delta}\frac{1}{r_nv_n}E\left[ ( V_n(g_{A_s})- V_n(g_{A_t}))^2\right]=0
	\end{align}
	and condition (D1) is satisfied.
	
	Finally we turn to condition (D2) of \DN.
	By condition (A) (vii) there exists $w=(w_k)_{1\le k\le q}\in [0,1]^q$ such that  $0\in A_w$, but $0\not\in A_s$ for all $s\lneq w$. By the monotonicity assumption (A) (ii) it follows that $0\in A_s$ if $s\ge w$.  Next we show that here even equivalence holds. To this end, suppose that $0\in A_s$ for some $s$ for which $s\ge w$ does not hold. Then on the one hand $0\in A_s\setminus A_{s\wedge w}$, but on the other hand $0\not\in A_{s\vee w}\setminus A_w$. However, this contradicts condition (A) (ii), which implies that  $A_s\setminus A_{s\wedge w}\subset A_{s\vee w}\setminus A_w$, as can be easily seen by successively replacing each $s_k<w_k$ with $w_k$.

   Fix some $n\in\N$ and let $\eps>0$ be arbitrary. In the same way as above, one may conclude from conditions (A) (ii) and (vi) that,	for all $1\le k\le q$, the function $F_k:[0,1]\to \mathbb{R}$,
	$$t\mapsto \frac{1}{r_nv_n}E\bigg[\bigg(\sum_{l=1}^{r_n}\ind{\{\|X_l\|>u_n\}}\sum_{h\in H_{n}}\frac{\|X_{l+h}\|^\alpha}{\sum_{j=-s_n}^{s_n}\|X_{l+j}\|^\alpha}\ind{A_{t^{(k)}}}
\Big(\frac{X_{l+h+i}}{\|X_{l+h}\|}\Big)\bigg)^2\bigg]$$
	is the measure generating function of a measure on $[0,1]$ with mass $F_k(1)\le N_n$ (cf.\ \eqref{eq:bed.3.15.forTnga.verifikation}).
	Hence, one can choose $t_{k,j}\in[0,1]$, $0\le j\le J_k^*\le\ceil{N_n(q/\eps)^2}+1$, such that $t_{k,0}:=0$, $t_{k,J_k^*}=1$, $w_k\in \{t_{k,j}|1\le j\le J_k^*\}$, $t_{k,j-1}< t_{k,j}$ and $F_k(t_{k,j}-)-F_k(t_{k,j-1})\le (\eps/q)^{2}$ for all $1\le j\le J_k^*$. These points define a partition $\mathcal{T}^*_k:=\{[t_{k,j-1},t_{k,j})|1\le j\le J_k^*\}\cup\{\{1\}\}$ of $[0,1]$ and hence $\mathcal{T}^*:=\{\times_{k=1}^q T_k\mid T_k\in\mathcal{T}_k^*,\forall\,1\le k\le q\}$ is a partition of $[0,1]^q$.

Now we define brackets
 $$\mathcal{A}^{\eps,n}_T:=\{A_t\mid t\in T\}$$
	for all $T\in\mathcal{T}^*$, so that the set $\mathcal{A}^{\epsilon,n}=\{\mathcal{A}^{\eps,n}_T \mid  T\in\mathcal{T}^*\}$ is a partition of $\mathcal{A}$.	
	Since $0\in A_s$ if and only if $s\geq w$, the above construction  ensures that, for all $\bar\AA\in \mathcal{A}^{\epsilon,n}$, either all or none of the sets in $\bar\AA$ contain 0. Consequently,
 $\ind{\{0\in A\}}=\ind{\{0\in B\}}$ for all $A,B\in\bar\AA$.
	
Using the Cauchy-Schwarz inequality for sums and $(a-b)^2\le a^2-b^2$ for all $a\ge b\ge 0$, we conclude similarly as above that, for all $T=\times_{k=1}^q T_k\in\mathcal{T}^* $ and $A_{T_k}^+:=\bigcup_{s\in T_k} A_{s^{(k)}}$,
	\begin{align}
	&\frac{1}{r_nv_n}E\bigg[\sup_{A,B\in\mathcal{A}^{\eps,n}_T} \Big(\sum_{l=1}^{r_n}\big(g_A(W_{n,l})-g_B(W_{n,l})\big)\Big)^2\bigg]\\
   & \le \frac{1}{r_nv_n}E\bigg[ \Big(\sum_{l=1}^{r_n}\big(g_{\bar A_T}(W_{n,l})-g_{\underline{A}_T}(W_{n,l})\big)\Big)^2\bigg]\\
  &\leq \frac{1}{r_nv_n} E\bigg[\bigg( \sum_{k\in K_T}\sum_{l=1}^{r_n}\ind{\{\|X_l\|>u_n\}}\sum_{h\in H_n}\frac{\|X_{l+h}\|^\alpha}{\sum_{j=-s_n}^{s_n}\|X_{l+j}\|^\alpha} \ind{A_{T_k}^+\setminus A_{(\min T_{k})^{(k)}}} \Big(\frac{X_{t+h+i}}{\|X_{t+h}\|}\Big)\bigg)^2\bigg]\\
	&\leq \frac{|K_T|}{r_nv_n}\sum_{k\in K_T}E\bigg[\bigg(\sum_{l=1}^{r_n}\ind{\{\|X_l\|>u_n\}}\sum_{h\in H_n}\frac{\|X_{l+h}\|^\alpha}{\sum_{j=-s_n}^{s_n}\|X_{l+j}\|^\alpha} \ind{A_{T_k}^+\setminus A_{(\min T_{k})^{(k)}}} \Big(\frac{X_{t+h+i}}{\|X_{t+h}\|}\Big)\bigg)^2\bigg]\\
	&\leq |K_T|\sum_{k\in K_T}\frac{1}{r_nv_n} E\bigg[\bigg(\sum_{l=1}^{r_n}\ind{\{\|X_l\|>u_n\}}\sum_{h\in H_n}\frac{\|X_{l+h}\|^\alpha}{\sum_{j=-s_n}^{s_n}\|X_{l+j}\|^\alpha} \ind{A_{T_k}^+} \Big(\frac{X_{t+h+i}}{\|X_{t+h}\|}\Big)\bigg)^2\bigg]\\
	&\hspace{2cm}-E\bigg[\bigg(\sum_{l=1}^{r_n}\ind{\{\|X_l\|>u_n\}}\sum_{h\in H_n}\frac{\|X_{l+h}\|^\alpha}{\sum_{j=-s_n}^{s_n}\|X_{l+j}\|^\alpha} \ind{ A_{(\min T_{k})^{(k)}}} \Big(\frac{X_{t+h+i}}{\|X_{t+h}\|}\Big)\bigg)^2\bigg] \Bigg)\\
& = |K_T|\sum_{k\in K_T} \big( F_k(\sup T_k-)-F_k(\min{T}_k)\big)
	\le |K_T|\sum_{k\in K_T}\eps^2 q^{-2}
	\le \eps^2.
\label{eq:calculation.for.d2.brackets}
	\end{align}
	Hence, $\mathcal{A}^{\eps,n}_T$ is indeed an $\eps$-bracket wr.t.\ the $L_2^n$-metric considered in condition (D2) of D{\&}N. Since there are $\prod_{k=1}^q (J_k^*+1)$ brackets,
 the bracketing number $N_{[\cdot]}(\eps,\mathcal{A},L_2^n)$ can be bounded by $\prod_{k=1}^q (J_k^*+1) \leq (N_n q^2/\eps^2+3)^q$. Now recall from \eqref{eq:bed.3.15.forTnga.verifikation} that $N_n$ is stochastically bounded. Thus,  for sufficient large $n$  and a suitable constant $c_1>0$,  it follows
	\begin{align}
	&\int_{0}^{\tau_n} \sqrt{ \log N_{[\cdot]}(\eps,\mathcal{A},L_2^n)}\,d\eps
	\leq \int_{0}^{\tau_n} \sqrt{c_1 +\log N_n - 2q\log \eps}\, d\eps
	\end{align}
	which tends to 0 in probability as $n\to\infty$ for all sequences $\tau_n \to 0$. Hence, (D2) is satisfied.
	
	To sum up, we have shown that all conditions of Theorem 2.3 (i) of \DN\ are fulfilled, which yields the assertion.
\end{proof}

\begin{proof}[Proof of Remark \ref{rem:processcon.linear}]
	The remark follows by completely analogous arguments as  Theorem \ref{th:procconv}, if  we  apply part (ii) of Theorem 2.3 of \DN\ (instead of part (i)). Then condition (D2) of \DN\ is not needed any more, but instead we have to verify (D3) of D{\&}N. Note that, by (A) (ii), the sets in $\AA$ are linearly ordered in the case $q=1$. The same then holds true for the functions in $\GG$, which is thus a $VC(2)$-class. This in turn implies (D3) of \DN\ (cf.\ Remark 2.11 of \cite{drees2010}).

Note that the conditions (A) (vi) and (vii) were only used for the proof of (D2) and are hence not needed in the case $q=1$.
\end{proof}

\subsection{Proof of Theorem \ref{th:procconvunknownalpha}.}

Define the function $\phi$ on $l_\alpha$ by $\phi((x_h)_{h\in\Z}):=\log^+ \|x_0\|$ and let $b_n(\phi):=(nv_n/p_n)^{1/2}$.

	The weak convergence of the first summand in the representation
	\begin{align}
	\sqrt{nv_n}& \big(\pnAhh-\pA\big)_{A\in\mathcal{A}}\\
 &= \sqrt{nv_n}(\pnAh-\pA)_{A\in\mathcal{A}} +\sqrt{nv_n}\big(\pnAhh-\pnAh\big)_{A\in\mathcal{A}}
   \label{eq:unknown.alpha.erste.zerlegung}
	\end{align}
	has already be shown in Theorem \ref{th:procconv}.
	
	To deal with the second term, we first analyze the asymptotic behavior of the estimator  $\alphanh=T_{n,\mathbb{R}^d}/\tilde T_{n,\phi}$ with $\tilde T_{n,\phi}:=\sum_{t=1}^{n}\phi(W_{n,t})=\sum_{t=1}^{n}\log^+(\|X_{t}\|/u_n)$.  Let $\Znphi:=(nv_n)^{-1/2}(\tilde T_{n,\phi}-E[\tilde T_{n,\phi}])$. With similar arguments as in the proof of Proposition \ref{prop:procconvt}, we will show weak convergence of $(\Ztn,\Znphi)$ to the centered Gaussian process $(Z,\Zphi)$ defined in Theorem \ref{th:procconvunknownalpha}. Note that to this end  it suffices to prove convergence of the fidis, because asymptotic tightness follows immediately from Proposition \ref{prop:procconvt}. Since $\phi$ is unbounded, here we apply a modification of Theorem 2.4 of \cite{DN20} (D{\&}N) (see Theorem 7.1 in the Supplement) instead of Theorem 2.3.
	To ease the formulas, in what follows we use the notation
	$$ X_{n,t} := \frac{X_t}{u_n}. $$
	Check that $E\big((\log^+\|X_{n,0}\|)^{\eta} \mid \|X_0\|>u_n\big)=\Ord(1)$ for $\eta\geq 0$ due to regular variation of $\|X_0\|$. By condition (BC') and stationarity, we obtain
	\begin{align}	
 \frac{1}{r_nv_n} & E\bigg(\Big(\sum_{t=1}^{r_n}\log^+\|X_{n,t}\|\Big)^{2}\bigg)
	= \sum_{s=1}^{r_n}\sum_{t=1}^{r_n}\frac{1}{r_nv_n} E\big(\log^+\|X_{n,s}\|\cdot\log^+\|X_{n,t}\|\big)\\
	&\leq 2\sum_{k=0}^{r_n}E\big(\log^+\|X_{n,0}\|\cdot\log^+\|X_{n,k}\|\mid \|X_0\|>u_n\big)
	= 2\sum_{k=1}^{r_n}e'_n(k)+\Ord(1)=\Ord(1),
	\end{align}
	that is, condition (7.2) of Theorem 7.1 in the Supplement is satisfied. Markov's inequality yields
	\begin{align}
		P\Big\{\sum_{t=1}^{r_n}\log^+\|X_{n,t}\|>\epsilon \sqrt{nv_n}\Big\} \leq \frac{1}{\eps^2nv_n}E\bigg(\Big(\sum_{t=1}^{r_n}\log^+\|X_{n,t}\|\Big)^2\bigg)
		=\Ord\Big(\frac{r_n}{n}\Big)
	\end{align}
	for all $\eps>0$. Applying H\"{o}lder's inequality, stationarity and condition (M) (i), we conclude for $\delta >0$
	\begin{align}
		&E\bigg(\Big(\sum_{t=1}^{r_n}\log^+\|X_{n,t}\|\Big)^2\Ind{\sum_{j=1}^{r_n}\log^+\|X_{n,j}\|>\eps \sqrt{nv_n}}\bigg)\\
		&\leq \sum_{s=1}^{r_n}\sum_{t=1}^{r_n}E\big((\log^+\|X_{n,s}\|\cdot\log^+\|X_{n,t}\|)^{1+\delta} \big)^{1/(1+\delta)}P\Big\{\sum_{j=1}^{r_n}\log^+\|X_{n,j}\|>\eps \sqrt{nv_n}\Big\}^{\delta/(1+\delta)}\\
		&\le 2r_nv_n^{1/(1+\delta)}\sum_{k=0}^{r_n}E\big((\log^+\|X_{n,0}\|\cdot\log^+\|X_{n,k}\|)^{1+\delta} \mid \|X_0\|>u_n\big)^{1/(1+\delta)} \cdot\Ord\Big(\Big(\frac{r_n}n\Big)^{\delta/(1+\delta)}\Big)\\
		&=\Ord\bigg(r_nv_n^{1/(1+\delta)}\Big(\frac{r_n}{n}\Big)^{\delta/(1+\delta)}\bigg)
		= \Ord\bigg( r_nv_n \Big(\frac{r_n}{nv_n}\Big)^{\delta/(1+\delta)}\bigg)=\ord(r_nv_n).
	\end{align}
	Thus condition (L2) of \DN\ is satisfied for $\phi$, which implies condition (L) of \DN\ (cf.\ Supplement of that paper). All remaining conditions of Theorem 7.1  (i) in the Supplement have been verified in the proof of  Proposition \ref{prop:asymptotic.of.tng}, except for the convergence of the standardized variance of $\Znphi$ and the covariance of $\Znphi$ and $\Ztn(A)$ for all $A\in\mathcal{A}$, which is proved in Lemma \ref{lemma:covaraince.pb.unknown.alpha}. 
	Hence we may conclude that
	\begin{align}
	\big((\Ztn(A))_{A\in \mathcal{A}},\Znphi\big)\xrightarrow{w} \big((Z(A))_{A\in \mathcal{A}},\Zphi\big).
	\end{align}

	By similar arguments as in the proof of Theorem \ref{th:fidiconv} using the bias conditions \eqref{eq:unifbiascond} and (H), we conclude
	\begin{equation}
	\label{eq:joint.convergence.pb.alpha}
	\sqrt{nv_n}\big((\pnAh-\pA)_{A\in\mathcal{A}}, \alphanh-\alpha \big)\xrightarrow{w} \big((\Zp(A))_{A\in\AA},Z_\alpha\big),
	\end{equation}
	with $Z_\alpha:=\alpha Z(\mathbb{R}^d)-\alpha^2\Zphi$ (cf.\ \cite{drees2015}, Lemma 4.4).
	
	For $A\in\AA$, define a function $p_{n,A}:(0,\infty)\to \mathbb{R}$ by
	\begin{align}
	p_{n,A}(a)
	& := \frac 1{\sum_{t=1}^n  \Ind{\|X_t\|>u_n}} \sum_{t=1}^n
	\Ind{\|X_t\|>u_n} \sum_{h= -s_n}^{s_n} \frac{\|X_{n,t+h}\|^a}{\sum_{k=-s_n}^{s_n} \|X_{n,t+k}\|^a}\\
	&  \hspace{3cm}\times \bigg( \Ind{h\in H_n} \ind{A}\Big(\frac{X_{t+h+i}}{\|X_{t+h}\|}\Big)
	+ \Ind{h\in H_n^c} \ind{A}(0) \bigg)
	\end{align}
	so that $\pnAh=p_{n,A}(\alpha)$ and $\pnAhh=p_{n,A}(\alphanh)$. A Taylor expansion yields
	\begin{align}
	&\pnAhh-\pnAh\\
	&= \frac{\hat{\alpha}_n-\alpha}{\sum_{t=1}^{n}\mathds{1}_{\{\|X_t\|>u_n\}}} \sum_{t=1}^{n} \mathds{1}_{\{\|X_t\|>u_n\}} \sum_{h=-s_n}^{s_n}\bigg( \frac{\log(\|X_{n,t+h}\|)\|X_{n,t+h}\|^{\alpha}}{\sum_{k=-s_n}^{s_n} \|X_{n,t+k}\|^{\alpha }}\\
	&\hspace{4.5cm}  - \frac{\|X_{n,t+h}\|^{\alpha} \sum_{k=-s_n}^{s_n}\log(\|X_{n,t+k}\|)\|X_{n,t+k}\|^{\alpha}}{(\sum_{k=-s_n}^{s_n} \|X_{n,t+k}\|^{\alpha })^2}\bigg)\\
	&\hspace{4.5cm}\times \bigg( \Ind{h\in H_n} \ind{A}\Big(\frac{X_{t+h+i}}{\|X_{t+h}\|}\Big)
	+ \Ind{h\in H_n^c} \ind{A}(0) \bigg) \\
	&  +\frac{(\alphanh-\alpha)^2}{2\sum_{t=1}^{n}\mathds{1}_{\{\|X_t\|>u_n\}}} \sum_{t=1}^{n} \mathds{1}_{\{\|X_t\|>u_n\}} \sum_{h=-s_n}^{s_n}\bigg( \frac{\log^2(\|X_{n,t+h}\|)\|X_{n,t+h}\|^{\bar{\alpha}}}{\sum_{k=-s_n}^{s_n} \|X_{n,t+k}\|^{\bar{\alpha} }}\\
	&\hspace{4cm} -2\frac{\log(\|X_{n,t+h}\|)\|X_{n,t+h}\|^{\bar{\alpha}} \sum_{k=-s_n}^{s_n}\log(\|X_{n,t+k}\|)\|X_{n,t+k}\|^{\bar{\alpha}}}{(\sum_{k=-s_n}^{s_n} \|X_{n,t+k}\|^{\bar{\alpha} })^2}\\
	& \hspace{4cm}-\frac{\|X_{n,t+h}\|^{\bar{\alpha}} \sum_{k=-s_n}^{s_n}\log^2(\|X_{n,t+k}\|)\|X_{n,t+k}\|^{\bar{\alpha}}}{(\sum_{k=-s_n}^{s_n} \|X_{n,t+k}\|^{\bar{\alpha} })^2}\\
    & \hspace{4cm} +2\frac{\|X_{n,t+h}\|^{\bar{\alpha}} (\sum_{k=-s_n}^{s_n}\log(\|X_{n,t+k}\|)\|X_{n,t+k}\|^{\bar{\alpha}})^2}{(\sum_{k=-s_n}^{s_n} \|X_{n,t+k}\|^{\bar{\alpha} })^3}\bigg)\\
	&\hspace{4.5cm}\times \bigg( \Ind{h\in H_n} \ind{A}\Big(\frac{X_{t+h+i}}{\|X_{t+h}\|}\Big)
	+ \Ind{h\in H_n^c} \ind{A}(0) \bigg) \label{eq:unknown.alpha.zweite.zerlegung}\\
	&=:I(A) + II(A),	
	\end{align}
	with $\bar{\alpha}=\bar{\alpha}_n=\lambda_n \alpha +(1-\lambda_n)\alphanh$ for some (random) $\lambda_n\in(0,1)$. 
Define functions $f_A^{(m)}$ on $l_\alpha\to\R$ by
\begin{align}
	f_A^{(m)}&\big((w_h)_{h\in\Z}\big)
	:=\Ind{\|w_0\|>1} \sum_{|h|\leq m}\bigg( \frac{\log(\|w_h\|)\|w_h\|^{\alpha}}{\sum_{|k|\leq m} \|w_k\|^{\alpha }}\\
   &\hspace{5cm}
  -\frac{\|w_h\|^{\alpha} \sum_{|k|\leq m} \log(\|w_k\|)\|w_k\|^{\alpha}}{(\sum_{|k|\leq m} \|w_k\|^{\alpha })^2}\bigg) \ind{A}\Big(\frac{w_{h+i}}{\|w_h\|}\Big)
  \label{eq:def.f.unknown.alpha}
	\end{align}
 and $f_A:=f_A^{(\infty)}$ with the convention $\sum_{|h|\leq \infty}=\sum_{h\in\mathbb{Z}}$.
In the Supplement it is shown that the constant $d_A$ defined in Theorem \ref{th:procconvunknownalpha} equals $E[f_A(Y)]$ (with ($Y=(Y_h)_{h\in\Z}$). Let
	\begin{align}
	\label{eq:def.cna}
	d_n(A):=\frac{1}{nv_n}\sum_{t=1}^{n}f_A(W_{n,t}).
	\end{align}
A combination of Lemma \ref{lemma:unknown.alpha.cna.gleichmaessig} and Lemma \ref{lemma:unknown.alpha.approximating.error} (given below) shows that $\sup_{A\in\AA}|d_n(A)-d_A|=\ord_P(1)$. Because $nv_n/\sum_{t=1}^{n}\mathds{1}_{\{\|X_t\|>u_n\}}\to 1$ in probability (which is immediate from  Proposition \ref{prop:asymptotic.of.tng}), it follows from convergence \eqref{eq:joint.convergence.pb.alpha} that
\begin{align}
	\sqrt{nv_n}I(A) =\sqrt{nv_n}(\hat{\alpha}_n-\alpha)\frac{nv_n}{\sum_{t=1}^{n}\mathds{1}_{\left\{\|X_t\|>u_n\right\}}}d_n(A) \xrightarrow{w} d_AZ_\alpha
\end{align}
uniformly for all $A\in\mathcal{A}$.
By Lemma \ref{lemma:unknow.alpha.part.ii}, the term $II(A)$ is asymptotically negligible. Hence, we may conclude from \eqref{eq:unknown.alpha.zweite.zerlegung} that
	
	\begin{align}
	\sqrt{nv_n}\left(\pnAhh-\pnAh\right)_{A\in\mathcal{A}}\rightarrow (d_AZ_\alpha)_{A\in\mathcal{A}}.
	\end{align}
    Indeed, the arguments show that this convergence holds jointly with the first component in \eqref{eq:joint.convergence.pb.alpha}.
	Hence, by \eqref{eq:unknown.alpha.erste.zerlegung} and the continuous mapping theorem
	\begin{align}
	\sqrt{nv_n}\left(\pnAhh-\pA\right)_{A\in\mathcal{A}}
	&\xrightarrow{w}(\Zp(A)+d_AZ_\alpha)_{A\in\mathcal{A}}\\
	&=\big(Z(A)-(P\{\Theta_i\in A\}-\alpha d_A)Z(\R^d)-\alpha^2 d_A\Zphi\big)_{A\in\AA}. \qed
	\end{align}

The Lemmas \ref{lemma:unknown.alpha.ewertkonvergenz}--\ref{lemma:unknown.alpha.approximating.error} establish uniform convergence of $d_n(A)-d_A$ to 0 in probability.
\begin{lemma}
	\label{lemma:unknown.alpha.ewertkonvergenz}
	Under Conditions  (RV), (S) and (C$\Theta$), we have
	\begin{align}
\lim_{n\to\infty}\frac{1}{v_n}E[f_A^{(m)}(W_{n,0})]=E[f_A^{(m)}(Y)]=:d_A^{(m)}\in\R
	\end{align}
	for all $A\in\mathcal{A}$ and all $m\in\mathbb{N}$, with $f_A^{(m)}$ defined by \eqref{eq:def.f.unknown.alpha}.
\end{lemma}

\begin{proof}
   Since $P\{Y_{h+i}/\|Y_{h}\|\in \partial A,\|Y_h\|>0\}=0$ for all $h\in\mathbb{Z}$ and all $A\in\mathcal{A}$ by Lemma \ref{lemma:umformulierung.e3}, the function $f_A^{(m)}$ is $P^{Y}$-a.s. continuous as finite sum of continuous functions. Because
   $\mathcal{L}((X_h/u_n)_{|h|\leq m}\mid \|X_0\|>u_n)\xrightarrow{w}\mathcal{L} ((Y_h)_{|h|\leq m})$ by the definition of the tail process,
the assertion
	$
 v_n^{-1}E[f_A^{(m)}(W_{n,0})] = E\big(f_A^{(m)}(W_{n,0})\mid \|X_0\|>u_n) \to
	d_A\in\R$ follows, provided  the random variables $v_n^{-1}f_A^{(m)}(W_{n,0}), n\in\N$, are uniformly integrable.
	 This, in turn, is implied by the uniform moment bound
	\begin{align}
	\sup_{n\in\mathbb{N}}E\big(|f_A^{(m)}(W_{n,0})|^{1+\eta} \mid\|X_0\|>u_n\big)<\infty
	\end{align}
	for some $\eta >0$. Due to the Minkowski inequality, we can verify this bound separately for both summands in the definition of $f_A^{(m)}$, whose moduli are maximal for $A=\R^d$. Since $s_n>m$ for $n$ large enough, it thus suffices to show that
	\begin{align}
E\bigg(\bigg|\frac{\sum_{h=-m}^{m} \log(\|X_{n,h}\|) \|X_{n,h}\|^{\alpha}}{\sum_{k=-m}^{m} \|X_{n,k}\|^{\alpha }}\bigg|^{1+\eta}\,\Big|\, \|X_0\|>u_n\bigg) \label{eq:unknown.alpha.1.unif.int.1}
	\end{align}
is bounded.
	Using Jensen's inequality and the fact that the denominator is at least 1 if $\|X_0\|>u_n$, this expectation  can be bounded by
	\begin{align}
	& E\bigg(  \frac{\sum_{h=-m}^{m} |\log\|X_{n,h}\||^{1+\eta} \|X_{n,h}\|^{\alpha}}{\sum_{k=-m}^{m} \|X_{n,k}\|^{\alpha }} \,\Big|\, \|X_0\|>u_n\bigg)\\
   & \le \sum_{h=-m}^{m} \Big[ E\big(  (\log^+\|X_{n,h}\|)^{1+\eta} \mid \|X_0\|>u_n\big)
   + E\big((\log^-\|X_{n,h}\|)^{1+\eta} \|X_{n,h}\|^{\alpha} \mid \|X_0\|>u_n\big)\Big].
  \end{align}
The first expectation is bounded because the time series is regularly varying and $(\log^+ x)^{1+\eta}=\ord(x^{\alpha-\eps})$ for some $\eps>0$ (cf.\ \cite{KS20}, Section 2.3.3), while the second is bounded, because $(\log^- x)^{1+\eta} x^\alpha$ is a bounded function. Since $m$ is fixed, \eqref{eq:unknown.alpha.1.unif.int.1} is bounded, which concludes the proof.
\end{proof}

\begin{lemma}
	\label{lemma:unknown.alpha.cna.convergence}
	If the conditions (RV), (S), (C$\Theta$)
are fulfilled, then
	\begin{align}
	d_n^{(m)}(A):=\frac{1}{nv_n}\sum_{t=1}^{n}f_A^{(m)}(W_{n,t}) \rightarrow d_A^{(m)}
	\end{align}
	in probability for all $A\in\mathcal{A}$ and all $m\in\mathbb{N}$.
\end{lemma}
\begin{proof}
   Rewrite $d_n^{(m)}(A)$ as
	\begin{align}
	\label{eq:unknow.alpha.1.decompose}
	d_n^{(m)}(A)
	=\frac{r_n}{n}\sum_{l=1}^{2} \sum_{j=1}^{m_{n,l} } F_{A,n}^{(m)}(l,j) + \frac{1}{nv_n}\sum_{t=m_nr_n+1}^{n}f_A^{(m)}(W_{n,t}) ,
	\end{align}
	with $m_n:=\floor{n/r_n}$, $m_{n,l}:=\max\{j\in\mathbb{N}_0\mid 2(j-1)+l\le n/r_n\} \sim n/(2r_n)$ and
   $$ F_{n,A}^{(m)}(l,j) := \frac 1{r_nv_n}\sum_{t=1}^{r_n} f_A^{(m)}(W_{n,(2(j-1)+l-1)r_n+t}). $$
    We first prove
    \begin{align}
	\label{eq:gesetz.grosse.zahlen.bed.unknown.alpha.qweies.moment}
	E[(F^{(m)}_{n,A}(1,1))^2]=\ord(n/r_n).
	\end{align}
Check that
\begin{align}
	&E\bigg[\bigg(\sum_{t=1}^{r_n}|f_A^{(m)}(W_{n,t})|\bigg)^2\bigg]\\
	&\leq 4 E\bigg[\bigg(\sum_{t=1}^{r_n}\mathds{1}_{\left\{\|X_t\|>u_n\right\}}\sum_{h=-m}^{m} \frac{|\log \|X_{n,t+h}\||\|X_{n,t+h}\|^{\alpha}}{\sum_{k=-m}^{m} \|X_{n,t+k}\|^{\alpha }}\bigg)^2\bigg]\\
	&\leq 8\Bigg( E\bigg[\bigg(\sum_{t=1}^{r_n}\mathds{1}_{\{\|X_t\|>u_n\}} \sum_{h=-m}^{m} \frac{|\log\|X_{n,t+h}\||\|X_{n,t+h}\|^{\alpha}}{\sum_{k=-m}^{m} \|X_{n,t+k}\|^{\alpha }}\mathds{1}_{\{\|X_{t+h}\|< u_n\}}\bigg)^2\bigg]\\
  & \hspace{1cm} + r_n^2 E\bigg[\bigg( \frac 1{r_n} \sum_{t=1}^{r_n}\mathds{1}_{\{\|X_t\|>u_n\}}\sum_{h=-m}^{m} \frac{\log(\|X_{n,t+h}\|)\|X_{n,t+h}\|^{\alpha}}{\sum_{k=-m}^{m} \|X_{n,t+k}\|^{\alpha }}\mathds{1}_{\{\|X_{t+h}\|\geq u_n\}}\bigg)^2\bigg]\Bigg)\\	
	&=:8(E_1+E_2).  \label{eq:expectFnarep}
	\end{align}
	According to \eqref{eq:bed.3.15.forTnga.verifikation}, $E_1$ is at most of the order $r_nv_n=\ord(nr_nv_n^2)$, since $\log^-(x)x^\alpha$ is bounded and the denominator is at least 1. 
	Using Jensen's inequality twice and stationarity, 
    we may conclude for the second term
	\begin{align}
	E_2 & \le r_n^2 E\bigg[\frac 1{r_n}  \sum_{t=1}^{r_n}
\bigg(\mathds{1}_{\{\|X_t\|>u_n\}} \sum_{h=-m}^{m} \frac{\log^+(\|X_{n,t+h}\|)\|X_{n,t+h}\|^{\alpha}}{\sum_{k=-m}^{m} \|X_{n,t+k}\|^{\alpha }}\bigg)^2\bigg]\\	
	&=r_n^2 E\bigg[\Ind{\|X_0\|>u_n}\bigg(\sum_{h=-m}^{m} \frac{\log^+(\|X_{n,h}\|)\|X_{n,h}\|^{\alpha}}{\sum_{k=-m}^{m} \|X_{n,k}\|^{\alpha }}\bigg)^2 \bigg]\\
	&\le r_n^2v_n E\bigg(\sum_{h=-m}^{m} \frac{ (\log^+(\|X_{n,h}\|))^{2}\|X_{n,h}\|^{\alpha}}{\sum_{k=-m}^{m} \|X_{n,k}\|^{\alpha }} \,\bigg|\, \|X_0\|>u_n\bigg)\\
	&\leq r_n^2v_n \sum_{h=-m}^{m}E\Big(\big( \log^+\|X_{n,h}\|\big)^2\,\big|\, \|X_0\|>u_n\Big)
	=\Ord(r_n^2v_n)=\ord(nr_nv_n^2),
	\end{align}
where the last two steps follow from regular variation as before (cf.\ \cite{KS20}, Section 2.3.3) and $r_n=\ord(\sqrt{nv_n})$.
	
  Combine this result with the bound on $E_1$ and \eqref{eq:expectFnarep}	to conclude
  \eqref{eq:gesetz.grosse.zahlen.bed.unknown.alpha.qweies.moment}.

  These calculations also show that the last summand in \eqref{eq:unknow.alpha.1.decompose} vanishes in $L_2$-norm and thus in probability:
  \begin{align}
	E&\bigg[\bigg(\frac{1}{nv_n}\sum_{t=m_nr_n+1}^{n}f_A^{(m)}(W_{n,t})\bigg)^2\bigg]
	\leq \frac{1}{(nv_n)^2}E\bigg[\bigg( \sum_{t=1}^{r_n}|f_A^{(m)}(W_{n,t})|\bigg)^2\bigg]    =\ord(r_n/n)=\ord(1).
  \end{align}
	
   Denote independent copies of $F_{n,A}^{(m)}(l,j)$ by $F_{n,A}^{(m)*}(l,j)$. Since, for fixed $l\in\{1,2\}$ and different $j\in\{1,\ldots, m_{n,l}\}$, the random variables  $F_{n,A}^{(m)}(l,j)$  are measurable functions of random variables $X_t$ which are separated in time by at least $r_n-2s_n-1$ observations, Eberlein's inequality (\cite{eberlein1984}) combined with condition (S) shows that the total variation distance between the joint distribution of the $F_{n,A}^{(m)}(l,j)$ and the joint distribution of the $F_{n,A}^{(m)*}(l,j)$ tends to 0:
	\begin{align}
	\big\|P^{(F^{(m)*}_{n,A}(l,j))_{1\le j\le m_{n,l}}} -P^{(F_{n,A}^{(m)}(l,j))_{1\le j\le m_{n,l}}} \big\|_{TV} \leq m_{n,l}\beta_{r_n-2s_n-1}=\Ord(n\beta_{l_n}/r_n) \to 0
	\end{align}
	for $l\in\{1,2\}$. Thus, in view of \eqref{eq:unknow.alpha.1.decompose}, $d_n^{(m)}(A)$ converges to $d_A^{(m)}$ in probability if
\begin{align}  \label{eq:FAnindepconv}
  \frac{r_n}{n}\sum_{j=1}^{m_{n,l} } F_{n,A}^{(m)*}(l,j) \to \frac{d_A^{(m)}}2
\end{align}
in probability for $l\in\{1,2\}$. This, however, follows by Chebyshev's inequality from \eqref{eq:gesetz.grosse.zahlen.bed.unknown.alpha.qweies.moment} and Lemma \ref{lemma:unknown.alpha.ewertkonvergenz}, which shows that
  $$ \frac{r_n}{n}\sum_{j=1}^{m_{n,l} } E\big[F_{n,A}^{(m)*}(l,j)] = \frac{r_n m_{n,l}}{n}\frac{E[f_A^{(m)}(W_{n,0})]}{v_n}\to \frac{d_A^{(m)}}2. \qedhere$$
\end{proof}

\begin{lemma}
	\label{lemma:unknown.alpha.cna.gleichmaessig}
	Suppose (RV), (S), (C$\Theta$), and (A)
 are satisfied. Then $\sup_{A\in\mathcal{A}}|d_A^{(m)}|<\infty$ and $\sup_{A\in\mathcal{A}}|d_n^{(m)}(A)-d_A^{(m)}|=\ord_P(1)$ for all $m\in\mathbb{N}$.
\end{lemma}
\begin{proof}
Fix $m\in\mathbb{N}$ and let
	\begin{align}
	f_{A}^{(I,\pm)}(w) & = \Ind{\|w_0\|>1}   \sum_{|h|\leq m} \frac{\log^{\pm}(\|w_h\|)\|w_h\|^{\alpha}}{\sum_{|k|\leq m} \|w_k\|^{\alpha }} \ind{A} \Big(\frac{w_{h+i}}{\|w_h\|}\Big),\\
	 f_{A}^{(II,\pm)}(w) & = \Ind{\|w_0\|>1}  \sum_{|h|\leq m}\frac{\|w_h\|^{\alpha} \sum_{|k|\leq m}\log^{\pm}(\|w_k\|)\|w_k\|^{\alpha}}{(\sum_{|k|\leq m} \|w_k\|^{\alpha })^2}  \ind{A}\Big(\frac{w_{h+i}}{\|w_h\|}\Big).
	\end{align}
	(To improve readability, here we suppress the dependence on $m$ in the notation.)
	In the proof of Lemma \ref{lemma:unknown.alpha.ewertkonvergenz}, it has been shown that the families of random variables $v_n^{-1}f_{A}^{(\sharp,\pm)}(W_{n,0})$, $n\in\mathbb{N}$, are uniformly integrable for $\sharp\in\{I,II\}$ (cf.\ \eqref{eq:unknown.alpha.1.unif.int.1}). Therefore, by the same arguments as used in the proof of Lemma \ref{lemma:unknown.alpha.cna.convergence}, it follows that
	\begin{align}
	d_n^{(\sharp,\pm)}(A) :=
  \frac{1}{nv_n}\sum_{t=1}^{n}f_A^{(\sharp,\pm)}(W_{n,t})\xrightarrow{P} 
  E\big[f_A^{(\sharp,\pm)}((Y_t)_{t\in\Z})\big]=:d_{A}^{(\sharp,\pm)}<\infty
	\end{align}
	for all $A\in\mathcal{A}$ and $\sharp\in\{I,II\}$. Indeed, this even holds for all $A\in\AA\cup\tilde\AA=\{A_t,A_t^-\mid t\in[0,1+\iota]^q\}$ as assumption (A) (iv) ensures the necessary continuity property for sets $A_t^-$. Since $f_{A}^{(m)}=f_{A}^{(I,+)}-f_{A}^{(I,-)}-f_{A}^{(II,+)}+f_{A}^{(II,-)}$ and $d_A^{(\sharp,\pm)}\le d_{\R^d}^{(\sharp,\pm)}<\infty$, it suffices to prove $\sup_{A\in\mathcal{A}}|d_{n}^{(\sharp,\pm)}(A)-d_{A}^{(\sharp,\pm)}|=\ord_P(1)$ for all $\sharp\in\{I,II\}$.
	
	Fix some $\eps>0$ and, for the time being, some $k\in\{1,\ldots,q\}$.
	By assumptions (A) (ii) and (v), $[0,1]\ni t\mapsto E\big[f_{A_{t^{(k)}}}^{(\sharp,\pm)}(Y)\big]$ (with $Y=(Y_h)_{h\in\Z}$) is non-decreasing and right-continuous. Thus, by the same arguments as used in the proof of Proposition \ref{prop:procconvt} for the function $H_k$, there exist $J_k\in\N$ and $0=:t_{k,0}<t_{k,1}<\ldots<t_{k,J_k}:=1$ such that
\begin{align}
   E\bigg[f_{\cup_{s< t_{k,j}}A_{s^{(k)}}}^{(\sharp,\pm)}(Y) -f_{A_{t_{k,j-1}^{(k)}}}^{(\sharp,\pm)}(Y)\bigg]<\eps
\end{align}
   for all $1\le j\le J_k$. Let $\mathcal{T}_k:=\{[t_{k,j-1},t_{k,j})\mid 1\le j\le J_k\}\cup\{\{1\}\}$.

   We now define brackets for $(f_{A}^{(\sharp,\pm)})_{A\in\mathcal{A}}$ by
	$$[{T_1,\ldots,T_q} ]:=\{f_{A_{(s_1,\ldots,s_l)}}^{(\sharp,\pm)} \mid s_k\in T_k,\;\forall\, 1\le k\le q\}$$
	for all $T_k\in\mathcal{T}_k$, $1\le k\le q$. Recall  definition \eqref{eq:setboundsdef} of $\bar A_T$ and $\underline{A}_T$. By construction, for all $s,t\in T:=\times_{k=1}^q T_k$,
   \begin{align}
     E  \big[\big|f_{A_s}^{(\sharp,\pm)}(Y) -f_{A_t}^{(\sharp,\pm)}(Y)\big|\big]
       & \le E \big[f_{\bar A_T}^{(\sharp,\pm)}(Y) -f_{\underline{A}_T}^{(\sharp,\pm)}(Y)\big]\\
        &\le \sum_{k=1}^q E \big[f_{\cup_{r\in T_k} A_{r^{(k)}}}^{(\sharp,\pm)}(Y)-f_{A_{(\min T_k)^{(k)}}}^{(\sharp,\pm)}(Y)\big] <q\eps.
   \end{align}
  This proves that, for arbitrary $\eps>0$, $(f_A^{(\sharp,\pm)})_{A\in\AA}$ is covered by finitely many $q\eps$-brackets w.r.t.\ the $L_1(P^{Y})$-norm.
	Since $d_{n}^{(\sharp,\pm)}(A)$ converges pointwise for all $A\in\mathcal{A}\cup\tilde\AA$, the  proof of Theorem 2.4.1 of \cite{vanderVaart1996} shows that
	$\sup_{A\in\mathcal{A}}|d_n^{(\sharp,\pm)}(A)- d_A^{(\sharp,\pm)}|=\ord_P(1)$ for all $\sharp\in\{I,II\}$.
\end{proof}

\begin{lemma}
	\label{lemma:unknown.alpha.approximating.error}
	\begin{itemize}
		\item[(i)] Suppose the Conditions (RV), (S), (BC), (BC'), (TC), and (M)~(ii), (iii) are satisfied. Then $$\lim_{m\to\infty}\limsup_{n\to\infty}\sup_{A\in\AA} E|d_n^{(m)}(A)-d_n(A)|=0.$$
		\item[(ii)] Under Conditions (RV), (S) and (M) (ii) $$\lim_{m\to\infty}\sup_{A\in\AA}|d^{(m)}_A-d_{A}|=0.$$
	\end{itemize}
\end{lemma}
\begin{proof}
	By stationarity one has for sufficiently large $n$ (such that $\{-m,\ldots,m\}\subset H_n$)
	\begin{align}
	&E(|d_n^{(m)}(A)-d_n(A)|)\leq E(|f_A^{(m)}(W_{n,0})-f_A(W_{n,0})|\mid \|X_0\|>u_n)\\
	&\leq  E\Big( \Big| \frac{\sum_{|h|\leq m} \log(\|X_{n,h}\|)\|X_{n,h}\|^{\alpha} \ind{A} \big(\frac{X_{h+i}}{\|X_{h}\|}\big)}{\sum_{|k|\leq m} \|X_{n,k}\|^{\alpha }} \\
	&\hspace{0,5cm}- \frac{\sum_{|h|\leq s_n}\log(\|X_{n,h}\|)\|X_{n,h}\|^{\alpha}\Big( \Ind{h\in H_n} \ind{A}\big(\frac{X_{h+i}}{\|X_{h}\|}\big)
	+ \Ind{h\in H_n^c} \ind{A}(0) \Big)}{\sum_{|k|\leq s_n} \|X_{n,k}\|^{\alpha }}\Big| \, \Big|\, \|X_0\|>u_n\Big)\\
	&\hspace{0.2cm}+ E\Bigg(  \Big| \frac{\sum_{|h|\leq m}\|X_{n,h}\|^\alpha \ind{A} \big(\frac{X_{h+i}}{\|X_{h}\|}\big)\sum_{|k|\leq m}\log(\|X_{n,k}\|)\|X_{n,k}\|^\alpha }{(\sum_{|k|\leq m} \|X_{n,k}\|^{\alpha })^2} \\
	&\hspace{0.2cm}-  \frac{\sum_{|h|\leq s_n}\|X_{n,h}\|^\alpha\Big( \Ind{h\in H_n} \ind{A}\big(\frac{X_{h+i}}{\|X_{h}\|}\big)
	+ \Ind{h\in H_n^c} \ind{A}(0) \Big) \sum_{|k|\leq s_n}\log(\|X_{n,k}\|)\|X_{n,k}\|^\alpha }{(\sum_{|k|\leq s_n} \|X_{n,k}\|^{\alpha })^2}\, \Big| \, \\
	&\hspace{2cm}   \Big|\, \|X_0\|>u_n\Bigg)\\
	&=: T_1+T_2.
	\end{align}
	In the following calculations we will use identities of the type
	\begin{align}
		\label{eq:sum-difference}
		&\sum_{|h|\leq s_n}a_h\sum_{|k|\leq m}b_k-\sum_{|h|\leq m}a_h\sum_{|k|\leq s_n}b_k
		= \sum_{m<|h|\leq s_n}a_h\sum_{|k|\leq m}b_k-\sum_{|h|\leq m}a_h\sum_{m<|k|\leq s_n}b_k
	\end{align}
	for $a_h,b_k\in\mathbb{R}$. Reducing the fractions to a common denominator, using \eqref{eq:sum-difference} and the triangle inequality and bounding all indicators by $1$, we obtain
	\begin{align}
		T_1&\leq E\bigg( \frac{\sum_{|h|\leq m} |\log\|X_{n,h}\||\|X_{n,h}\|^{\alpha} \sum_{m<|k|\leq s_n} \|X_{n,k}\|^{\alpha }}{\sum_{|k|\leq m} \|X_{n,k}\|^{\alpha }\sum_{|k|\leq s_n} \|X_{n,k}\|^{\alpha }}\,\Big|\, \|X_0\|>u_n\bigg) \\
		&\hspace{0,5cm}+E\bigg(\frac{\sum_{m<|h|\leq s_n} |\log\|X_{n,h}\||\|X_{n,h}\|^{\alpha} }{\sum_{|k|\leq s_n} \|X_{n,k}\|^{\alpha }}   \,\Big|\, \|X_0\|>u_n\bigg)\\
		&=:T_{1,1}+T_{1,2}.
	\end{align}
	Applying the H\"{o}lder inequality for expectations and for sums, for $\delta>0$, we can bound $T_{1,1}$ by
	\begin{align}
		T_{1,1}\leq& \Bigg( E\bigg( \frac{\sum_{|h|\leq m} |\log\|X_{n,h}\||^{1+\delta}\|X_{n,h}\|^{\alpha} }{\sum_{|k|\leq m} \|X_{n,k}\|^{\alpha }}\,\Big|\, \|X_0\|>u_n\bigg)\Bigg)^{1/(1+\delta)}\\
		&\hspace{3cm}\times \Bigg(E\bigg( \bigg(\frac{ \sum_{m<|k|\leq s_n} \|X_{n,k}\|^{\alpha }}{\sum_{|k|\leq s_n} \|X_{n,k}\|^{\alpha }}\bigg)^{(1+\delta)/\delta}\,\Big|\, \|X_0\|>u_n\bigg)\Bigg)^{\delta/(1+\delta)}.
	\end{align}
	Because of uniform integrability (cf. \eqref{eq:unknown.alpha.1.unif.int.1}), the first expectation converges to the expectation that is bounded in \eqref{eq:condMiiivar}. The second expectation is bounded by
	\begin{align}
		&E\left( \frac{ \sum_{m<|k|\leq s_n} \|X_{n,k}\|^{\alpha }}{\sum_{|k|\leq s_n} \|X_{n,k}\|^{\alpha }}\,\Big|\, \|X_0\|>u_n\right)\\
		&\leq E\left( \frac{ \sum_{m<|k|\leq s_n} \|X_{n,k}\|^{\alpha }\Ind{\|X_k\|\leq cu_n}}{\sum_{|k|\leq s_n} \|X_{n,k}\|^{\alpha }}\,\Big|\, \|X_0\|>u_n\right) \\
		&\hspace{0,5cm}+ \sum_{m<|k|\leq s_n} P(\|X_{k}\|>cu_n\mid \|X_0\|>cu_n) \frac{P(\|X_0\|>cu_n)}{P(\|X_0\|>u_n)} \\
		&\leq E\left( \frac{ \sum_{m<|k|\leq s_n} \|X_{n,k}\|^{\alpha }\Ind{\|X_k\|\leq cu_n}}{\sum_{|k|\leq s_n} \|X_{n,k}\|^{\alpha }}\,\Big|\, \|X_0\|>u_n\right) + 4c^{-\alpha}\sum_{h=m}^{s_n}e_{n,c}(h),
	\end{align}
	for sufficiently large $n$ and the constant $c\in (0,1)$ from Condition (TC), and hence $\lim_{m\to\infty}\limsup_{n\to\infty}T_{1,1}=0$ by condition (M) (ii), (BC) and (TC). Condition (BC') implies
	\begin{align}
		&E\bigg(\frac{\sum_{m<|h|\leq s_n} \log^{+}(\|X_{n,h}\|)\|X_{n,h}\|^{\alpha} }{\sum_{|k|\leq s_n} \|X_{n,k}\|^{\alpha }}   \,\Big|\, \|X_0\|>u_n\bigg)\\
		&\leq \sum_{m<|h|\leq s_n} E\left( \log^{+}\|X_{n,h}\|   \mid \|X_0\|>u_n\right)
		\leq 2\sum_{m<h\leq s_n} e_n'(h),
	\end{align}
 	and thus $\lim_{m\to\infty}\limsup_{n\to\infty}T_{1,2}=0$ by condition (BC') and (M) (iii).

 	Similarly, one can show that
 	\begin{align}
 	&T_2\leq 2E\bigg( \frac{ \sum_{|k|\leq m}|\log\|X_{n,k}\||\|X_{n,k}\|^\alpha   \sum_{m<|h|\leq s_n} \|X_{n,h}\|^\alpha} {\sum_{|k|\leq m} \|X_{n,k}\|^{\alpha } \sum_{|k|\leq s_n} \|X_{n,k}\|^{\alpha }}      \,\Big|\, \|X_0\|>u_n\bigg)\\
 	&\hspace{1cm}+E\bigg( \frac{ \sum_{m<|k|\leq s_n}|\log\|X_{n,k}\||\|X_{n,k}\|^\alpha }{\sum_{|k|\leq s_n} \|X_{n,k}\|^{\alpha }}   \,\Big|\, \|X_0\|>u_n\bigg)\\
 	&\hspace{1cm}+E\bigg( \frac{\sum_{m<|h|\leq s_n}\|X_{n,h}\|^\alpha \sum_{|k|\leq m}|\log\|X_{n,k}\||\|X_{n,k}\|^\alpha }{(\sum_{|k|\leq s_n} \|X_{n,k}\|^{\alpha })^2}   \,\Big|\, \|X_0\|>u_n\bigg)\\
 	&\le 3T_{1,1}+T_{1,2},
 	\end{align}
 	so that $\lim_{m\to\infty}\limsup_{n\to\infty}T_2=0$. Since the bounds on $T_1$ and $T_2$ do not depend on $A$, all convergences hold uniformly in $A$, which proves  assertion (i).	
	Assertion (ii) follows by similar arguments, which are given in detail in the Supplement.
\end{proof}

The next two lemmas establish the asymptotic negligibility of $II(A)$. The proof of the first lemma is given in the Supplement.
\begin{lemma}
	\label{lemma:maximizing.m}
	For $m\in\N$ and $a=(a_1,\ldots,a_m)\in [0,1]^m$ let
	\begin{align}
		M(a)& :=\frac{\sum_{k=1}^{m}a_k \log^2a_k}{1+\sum_{k=1}^{m}a_k},\qquad
		\tilde M(a)  := \frac{\sum_{k=1}^{m}a_k|\log a_k|}{1+\sum_{k=1}^{m}a_k}
	\end{align}
	with $0\log^i(0):=0$ for $i\in\{1,2\}$.
	Then $\sup_{a\in[0,1]^m} M(a)=\Ord(\log^2 m)$ and $\sup_{a\in[0,1]^m} \tilde M(a)$ $=\Ord(\log m)$
	as $m\to\infty$.
\end{lemma}

\begin{lemma}
	\label{lemma:unknow.alpha.part.ii}
	If the conditions (RV), (S), (BC), (TC), (C$\Theta$), (BC'), (M) (i), (H) and $\log^4 n=\hbox{o}(nv_n)$ are fulfilled, then the term $II(A)$ defined in \eqref{eq:unknown.alpha.zweite.zerlegung} is uniformly negligible:
	\begin{align}
	\sup_{A\in\mathcal{A}}|II(A)|=\ord_P\big((nv_n)^{-1/2}\big).
	\end{align}
\end{lemma}
\begin{proof}
	Because $(\hat{\alpha}_n-\alpha)^2=\Ord_P((nv_n)^{-1})$ by \eqref{eq:joint.convergence.pb.alpha} and $nv_n/(\sum_{t=1}^{n}\mathds{1}_{\{\|X_t\|>u_n\}})\to 1$ in probability, due to the weak convergence in Proposition \ref{prop:asymptotic.of.tng} for $\mathbb{R}^d$, it suffices to show that
	\begin{align} \label{eq:unknow.alpha.second.taylor}
  \frac{2\sum_{t=1}^{n} \Ind{\|X_t\|>u_n}}{(\hat{\alpha}_n-\alpha)^2}|II(A)|
	&\le 2\sum_{t=1}^{n} \Ind{\|X_t\|>u_n} \sum_{h=-s_n}^{s_n} \frac{\log^2(\|X_{n,t+h}\|)\|X_{n,t+h}\|^{\bar{\alpha}}}{\sum_{k=-s_n}^{s_n} \|X_{n,t+k}\|^{\bar{\alpha} }}\\
   & \hspace{0,5cm} + 4 \sum_{t=1}^{n} \Ind{\|X_t\|>u_n} \bigg(\frac{ \sum_{h=-s_n}^{s_n}|\log(\|X_{n,t+h}\|)|\|X_{n,t+h}\|^{\bar{\alpha}}}{\sum_{k=-s_n}^{s_n} \|X_{n,t+k}\|^{\bar{\alpha} }}\bigg)^2\\
  & =: 2T_1 +4T_2 = \ord_P\big((nv_n)^{3/2}\big).
\label{eq:unknow.alpha.second.taylor2}
	\end{align}

	For this purpose we distinguish whether $\|X_{t+h}\|$ exceeds $u_n$ or not.
	Define the set $B_n:=\big\{\max_{1\leq t\leq n}\|X_{n,t}\|\le n^{2/\alpha}\big\}$. The regular variation of $\|X_0\|$ implies, for all $\epsilon\in(0,1/2)$ and sufficient large $n$,
	\begin{align}
	P(B_n^c)&\leq n P\big\{\|X_{n,0}\|>n^{2/\alpha}\big\}\leq nP\big\{\|X_0\|>n^{2/\alpha}\big\}=\ord\big(n(n^{2/\alpha})^{-\alpha(1-\epsilon)}\big)
	=\ord(1).
	\end{align}
	 Since on the set $B_n$
	\begin{align}
	0\leq \log(\|X_{n,t}\|)\mathds{1}_{\{\|X_t\|\geq u_n\}} \leq \frac{2}{\alpha}\log n
	\end{align}
	for all $1\leq t\leq n$, we may conclude that, with probability tending to 1,
	\begin{align}
	  \sum_{t=1}^n & \Ind{\|X_t\|>u_n} \sum_{h=-s_n}^{s_n} \frac{(\log^2(\|X_{n,t+h}\|)\|X_{n,t+h}\|^{\bar{\alpha}}\Ind{\|X_{t+h}\|>u_n}}{\sum_{k=-s_n}^{s_n} \|X_{n,t+k}\|^{\bar{\alpha} }}\\
	 & \leq \frac{4}{\alpha^2}\log^2 n \sum_{t=1}^n \Ind{\|X_t\|>u_n}
= \ord\big((nv_n)^{3/2}\big),
\label{eq:unknow.alpha.2.taylor.I.gro}
	\end{align}
	since $\log^2 n =\hbox{o}(\sqrt{nv_n})$ by assumption.

   Next, observe that by the first assertion of Lemma \ref{lemma:maximizing.m}
  \begin{align}
	& \sum_{t=1}^{n} \Ind{\|X_t\|>u_n} \sum_{h=-s_n}^{s_n} \frac{\log^2(\|X_{n,t+h}\|)\|X_{n,t+h}\|^{\bar{\alpha}}\Ind{\|X_{t+h}\|\le u_n}}{\sum_{k=-s_n}^{s_n} \|X_{n,t+k}\|^{\bar{\alpha} }}\\
	&\leq \sum_{t=1}^{n} \Ind{\|X_t\|>u_n}\frac{1}{\bar{\alpha}^2}\sum_{h=-s_n, h\neq 0}^{s_n} \frac{\log^2(\|X_{n,t+h}\|^{\bar{\alpha}})\|X_{n,t+h}\|^{\bar{\alpha}} \Ind{\|X_{n,t+h}\|\le 1}}{1+\sum_{k=-s_n,k\neq 0}^{s_n} \|X_{n,t+k}\|^{\bar{\alpha} }}\\
   & = \Ord(\log^2 s_n) \sum_{t=1}^{n} \Ind{\|X_t\|>u_n} = \ord\big((nv_n)^{3/2}\big).
\label{eq:unknown.alpha.2.taylor.kleiner1.I}
	\end{align}
Combine \eqref{eq:unknow.alpha.2.taylor.I.gro} and \eqref{eq:unknown.alpha.2.taylor.kleiner1.I} to see that $T_1$ is stochastically of smaller order than $(nv_n)^{3/2}$.

Similarly as in  \eqref{eq:unknow.alpha.2.taylor.I.gro}, on the set $B_n$ we obtain
	\begin{align}
	&\sum_{t=1}^{n} \Ind{\|X_t\|>u_n} \bigg( \frac{ \sum_{h=-s_n}^{s_n}|\log(\|X_{n,t+h}\|)|\|X_{n,t+h}\|^{\bar{\alpha}} \Ind{\|X_{t+h}\|>u_n}}{\sum_{k=-s_n}^{s_n} \|X_{n,t+k}\|^{\bar{\alpha} }}\bigg)^2\\
	& \leq \frac{4}{\alpha^2}\log^2 n \sum_{t=1}^n \Ind{\|X_t\|>u_n}
= \ord\big((nv_n)^{3/2}\big).\label{eq:unknow.alpha.2.taylor.IV.gro}
	\end{align}
	Moreover, by the second assertion of Lemma  \ref{lemma:maximizing.m}
  \begin{align}
	&\sum_{t=1}^{n} \Ind{\|X_t\|>u_n} \bigg(\frac{ \sum_{h=-s_n}^{s_n}|\log\|X_{n,t+h}\||\|X_{n,t+h}\|^{\bar{\alpha}} \Ind{\|X_{t+h}\|\le u_n}}{\sum_{k=-s_n}^{s_n} \|X_{n,t+k}\|^{\bar{\alpha} }}\bigg)^2\\
	& = \Ord (\log^2 s_n) \sum_{t=1}^{n} \Ind{\|X_t\|>u_n} = \ord\big((nv_n)^{3/2}\big).
  \label{eq:secondtermnegativepart}
  \end{align}
  A combination of both these results show that $T_2$, which can be bounded by double the sum of the left-hand sides of \eqref{eq:unknow.alpha.2.taylor.IV.gro} and \eqref{eq:secondtermnegativepart}, is stochastically of smaller order than $(nv_n)^{3/2}$.
  This concludes the proof.
\end{proof}

In the last lemma, the proof of which is given in the Supplement, the limit behavior of the covariance between $(Z(A))_{A\in\mathcal{A}}$ and $\Zphi$ and of  the variance of $\Zphi$ is analyzed.
\begin{lemma}
	\label{lemma:covaraince.pb.unknown.alpha}
	Suppose the conditions (RV), (S), (BC), (TC), (C$\Theta$), and (BC') are satisfied. Then
	\begin{enumerate}
		\item[(i)]
		\begin{align}
		\frac{1}{r_nv_n}&Var\bigg(\sum_{t=1}^{r_n}\phi(W_{n,t})\bigg)
		=\alpha^{-1}\sum_{k\in\mathbb{Z}}E\left[(1\wedge \|\Theta_k\|^\alpha)(|\log(\|\Theta_k\|)|+2\alpha^{-1})\right],
		\end{align}
		
		\item[(ii)] For all $A\in\mathcal{A}$,
		\begin{align}
		\frac{1}{r_nv_n}&Cov\bigg(\sum_{j=1}^{r_n}\phi(W_{n,j}),\sum_{t=1}^{r_n}g_A(W_{n,t})\bigg)\\
		&=\sum_{k\in\mathbb{Z}}E\left[\sum_{h\in\mathbb{Z}}\frac{\|\Theta_h\|^\alpha}{\|\Theta\|_\alpha^\alpha}\ind{ A}\left(\frac{\Theta_{h+i}}{\|\Theta_h\|}\right) (1\wedge \|\Theta_k\|^\alpha)(\log^+\|\Theta_k\|+\alpha^{-1})\right].
		\end{align}
		
	\end{enumerate}
\end{lemma}


\bibliography{Dissbib_070321}
\bibliographystyle{agsm}
\end{document}


\begin{frontmatter}
\title{Supplement to the paper\\Cluster based inference for extremes of time series}
\runtitle{Cluster based inference}
\begin{aug}
	\author[A]{\fnms{Holger} \snm{Drees}\ead[label=e1]{drees@math.uni-hamburg.de}},
	\author[B]{\fnms{Anja} \snm{Jan{\ss}en}\ead[label=e2]{anja.janssen@ovgu.de}},
	\author[A]{\fnms{Sebastian} \snm{Neblung}\ead[label=e3,mark]{sebastian.neblung@uni-hamburg.de}}
	\address[A]{University of Hamburg, Department of Mathematics,
		SPST, Bundesstr.\ 55, 20146 Hamburg, Germany, \printead{e1,e3}}
	\address[B]{Otto-von-Guericke University of Magdeburg, Faculty of Mathematics, IMST,
		Universit\"{a}tsplatz~2, 39106 Magdeburg, Germany, \printead{e2}}
\end{aug}

\begin{abstract}
	\; We continue the simulation study from section 3 of the main paper and provide additional simulation results in another model. Furthermore, we specify the variance example at the end of Section 2 by giving the concrete variance formulas. We verify three conditions of Theorem 2.3 for solutions to stochastic recurrence equations. Moreover, we state a modified version of the sliding block limit theorem from \cite{DN20} which is used for the proof of Theorem 2.3 of the main paper.  Finally, we provide the proofs for equation (2.1), equation (A.4), Remark 2.2, Lemma A.2, Lemma A.9 (ii), Lemma A.10 and Lemma A.12 of the main paper.
\end{abstract}

\begin{keyword}[class=AMS]
\kwd[Primary ]{62G32}  \kwd[; secondary ]{62M10, 62G05, 60F17}
\end{keyword}

\begin{keyword}
\kwd{cluster of extremes, empirical processes, extreme value analysis,   time series, uniform central limit theorems}
\end{keyword}
\end{frontmatter}

To avoid confusion, we continue the section and figure numbering from the main article.

\setcounter{section}{3}
\setcounter{figure}{2}

\section{Further simulations.}

In this section, we present additional simulation results about the finite sample performance of the projection based estimator $\pnAhh$ and its competitors $\pnfh$ and $\pnbh$ of the cdf of $\Theta_i$. We keep the Monte Carlo setting (sample size, threshold, block length, number of Monte Carlo simulations)  used in the paper.

In Section 3, among other things, the estimators are discussed  in the GARCHt model for lag $i=10$, where $\pnAhh$ performs best. Figure \ref{figure:garcht-5-95} shows the corresponding results for smaller lags, namely $i=1$ and $i=5$, respectively. Because $\pnAhh$ and $\pnbh$ have point mass at $0$, $\pnfh$ typically has  smallest bias among all estimators  for small $|x|$, which is also reflected in the RMSE.
In contrast, $\pnAhh$ has smallest variance for most values of $x$ and this effect becomes more pronounced as the lag increases.
Consequently, the projection based estimator outperforms the backward estimator, in particular as $i$ increases, and it has smallest RMSE for $|x|>0.3$ and all lags under consideration.

\begin{figure}[h!]
	\includegraphics[width=1\textwidth]{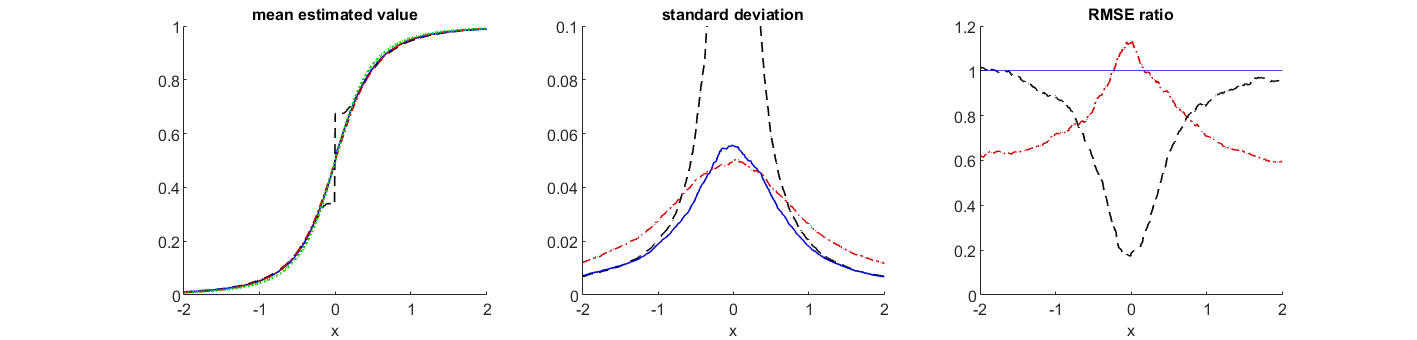}
	\includegraphics[width=1\textwidth]{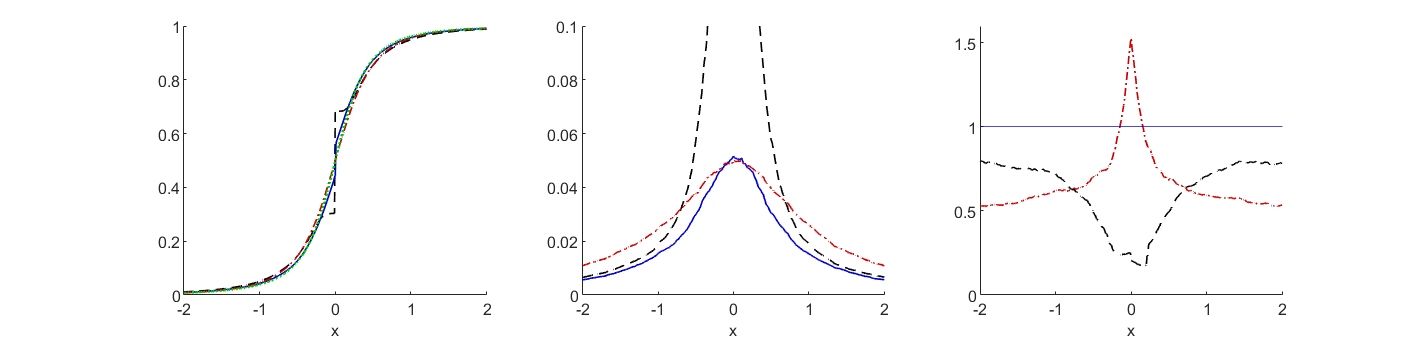}
	\caption{Mean (left), standard deviation (middle) and relative efficiency w.r.t.\  $\pnAhh$ (right), of $\pnAhh$ (blue solid line), $\pnbh$ (black dashed line) and $\pnfh$ (red dashed-dotted line) in the
GARCHt model for lags $i=1$ (top) and  $i=5$ (bottom); the true cdf is indicated by the green dotted line.}
	\label{figure:garcht-5-95}
\end{figure}

While in the (asymmetric) SRE model (see Figure 2 in the paper for lag 1 and Figure~\ref{figure:sr-10-95} for lags $i\in\{5,10\}$) the variances of the three estimators behave similarly, the large bias, which is due to the influence of $D_t$, often dominates the RMSE, leading to an underperformance of $\pnAhh$. This is particularly true for larger lags and negative $x$, while for $x>0$ and lag 10, $\pnAhh$ has minimal RMSE.

Note that here the difference between the conditional distribution of the exceedances over the threshold $F_{|X|}^{\leftarrow}(0.95)$ and its limit distribution obscures almost all real differences between the estimators. It is thus instructive to compare, in addition, the performance of $\pnAhh$, $\pnfh$ and $\pnbh$ interpreted as estimators of the pre-asymptotic cdf $P(X_i/|X_0|\le x\mid |X_0|>F_{|X|}^{\leftarrow}(0.95))$ (rather than its limit $P\{\Theta_i\le x\}$), like it was done in \cite{davis2018} for the backward estimator. The corresponding plots for lag $i=1$ are shown in Figure \ref{figure:sr-1-95-pre}. Of course, in the left plot only the true cdf has changed and the middle plot is exactly the same as in the bottom line of Figure 2, but the RMSE of the projection based estimator is much more strongly reduced than that of the forward estimators for $x$ not too close to 0. So, in this setting, $\pnAhh$  overall performs best as an estimator of the pre-asymptotic cdf.
\begin{figure}[t!]
	\includegraphics[width=1\textwidth]{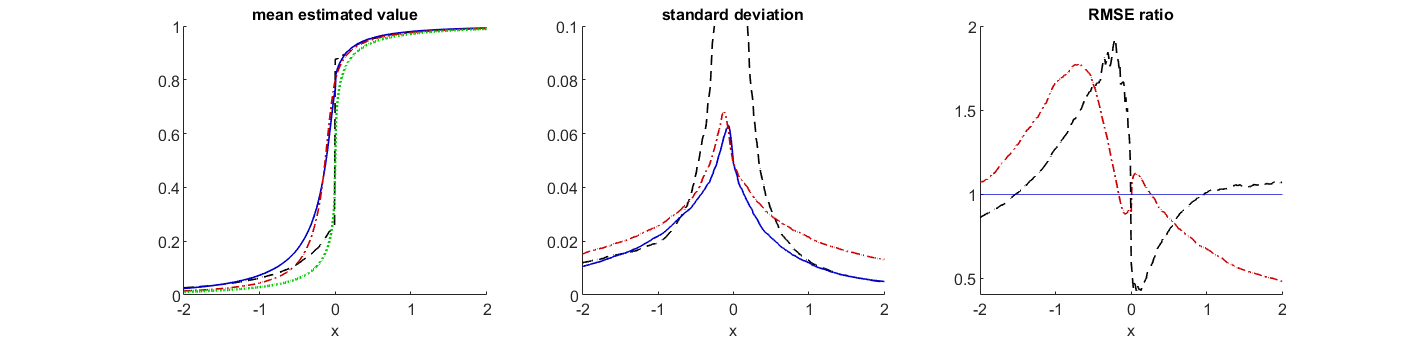}
	\includegraphics[width=1\textwidth]{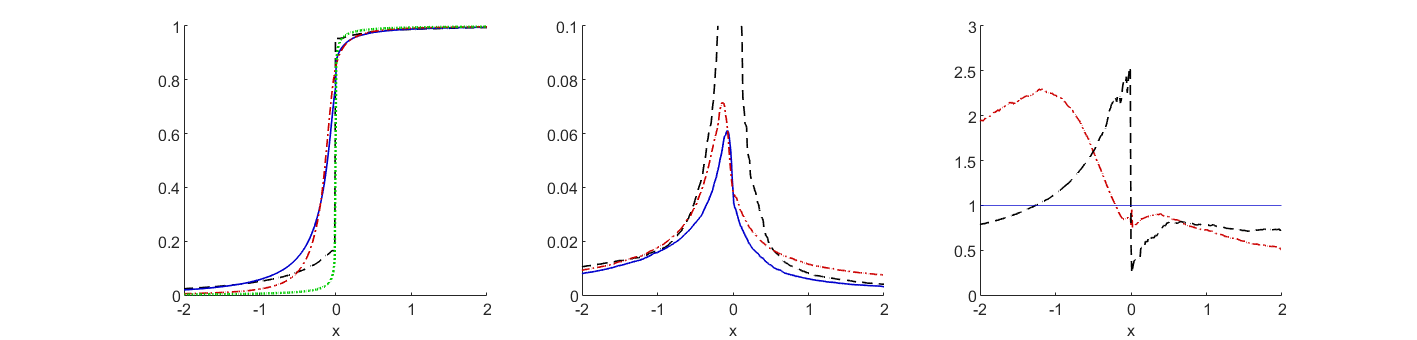}
	\caption{Mean (left), standard deviation (middle) and relative efficiency w.r.t.\  $\pnAhh$ (right), of $\pnAhh$ (blue solid line), $\pnbh$ (black dashed line) and $\pnfh$ (red dashed-dotted line) in the
SRE model for lags $i=5$ (top) and  $i=10$ (bottom); the true cdf is indicated by the green dotted line.}
	\label{figure:sr-10-95}
\end{figure}

\begin{figure}[t!]
	\includegraphics[width=1\textwidth]{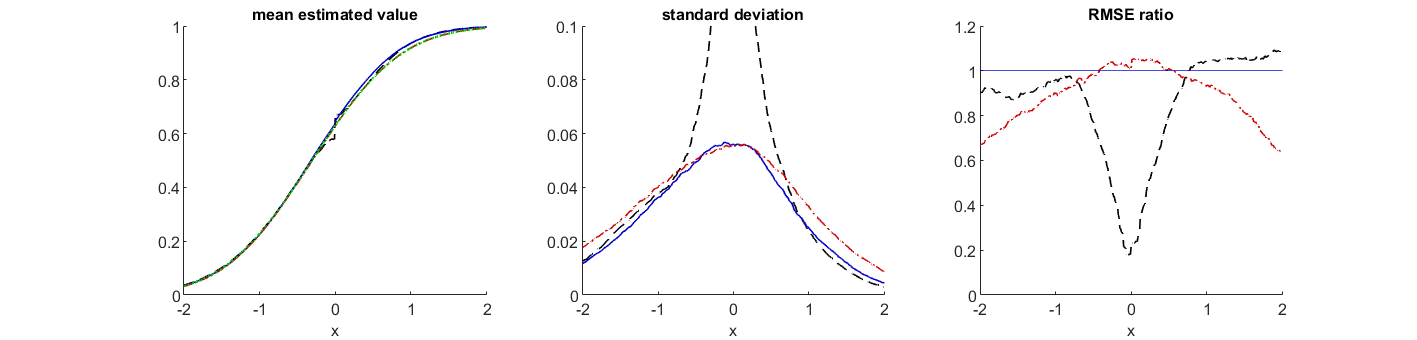}
	\caption{Mean (left), standard deviation (middle) and relative efficiency w.r.t.\  $\pnAhh$ (right), of $\pnAhh$ (blue solid line), $\pnbh$ (black dashed line) and $\pnfh$ (red dashed-dotted line) as estimators of the pre-asymptotic cdf of $X_i/|X_0|$ given $|X_0|>F_{|X|}^{\leftarrow}(0.95)$ in the
SRE model for lag $i=1$; the true cdf is indicated by the green dotted line.}
	\label{figure:sr-1-95-pre}
\end{figure}

In addition to the GARCHt and the SRE model, here we examine a model with trivial tail process:
\begin{itemize}
\item SV: Consider the stationary stochastic volatility model $X_t=\sigma_t \epsilon_t,$ $t \in \Z,$ with $\log(\sigma_t)= 0.9\log(\sigma_{t-1})+Z_t$ where $Z_t$ are i.i.d.\ standard normal random variables and $\epsilon_t$ are i.i.d.\ with Student's $t_{2.6}$-distribution, independent of $(Z_t)_{t \in Z}$. Then $(X_t)_{t \in \Z}$ is a stationary regularly varying time series with $\alpha=2.6$ (cf. \cite{davis2018}, Section 4). Since the extremal behavior of $(X_t)_{t \in \Z}$ is dominated by the i.i.d.\ heavy-tailed innovations $\epsilon_t$, we have $P\{\Theta_i=0\}=1$ for all $i\neq 0$, see \cite{davis2009extremes}.
\end{itemize}

Note that the conditions of Section 2 can only be fulfilled for the family of sets $(-\infty,x]$ if a neighborhood of 0 is omitted from the range of $x$-values. Nevertheless, here we present the results for the full range $[-2,2]$ (and lags $i\in\{1,10\}$) in Figure \ref{figure:sv-10-95}.
For lag $i=1$, the projection based estimator performs best, but the loss of efficiency of the other estimators is less than 10\%.
In contrast, for $i=10$, the $\pnfh$ has usually a moderately smaller RMSE than $\pnAhh$, which in turn is preferable to $\pnbh$ outside a tiny neighborhood of the origin.

\begin{figure}[t!]
	\includegraphics[width=1\textwidth]{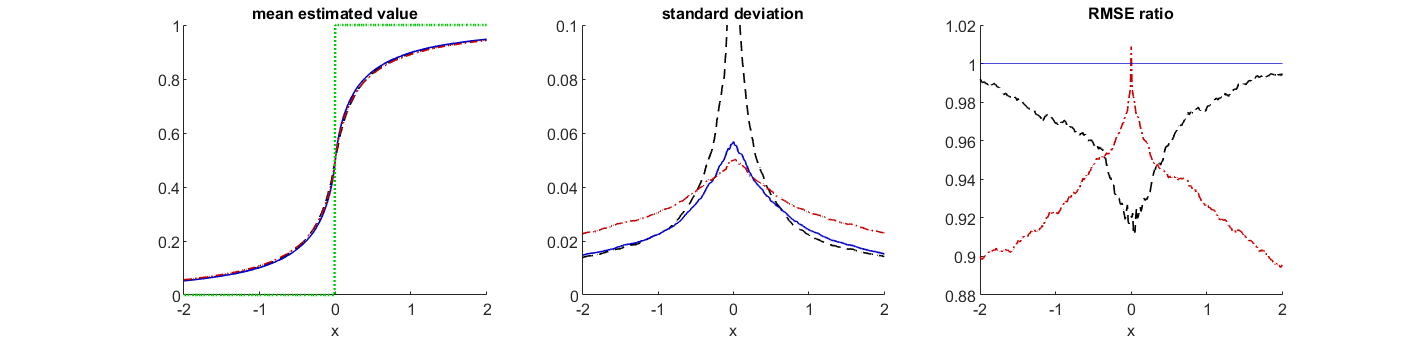}
	\includegraphics[width=1\textwidth]{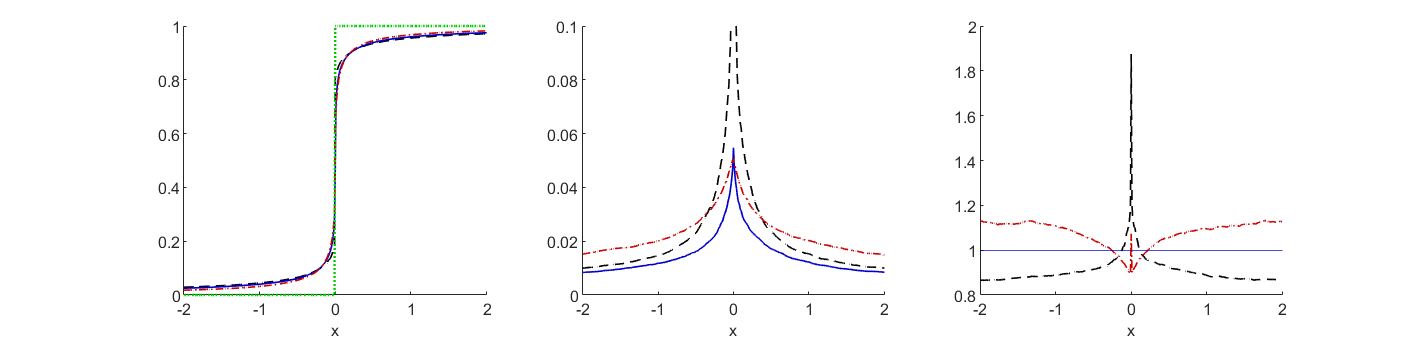}
	\caption{Mean (left), standard deviation (middle) and relative efficiency w.r.t.\  $\pnAhh$ (right), of $\pnAhh$ (blue solid line), $\pnbh$ (black dashed line) and $\pnfh$ (red dashed-dotted line) in the
SV model for lags $i=i$ (top) and  $i=10$ (bottom); the true cdf is indicated by the green dotted line.}
	\label{figure:sv-10-95}
\end{figure}

To sum up, usually the projection based estimator has the smallest variance, especially for larger lags. If pre-asymptotic probabilities significantly  differ from the corresponding limit probabilities,  $\pnAhh$ may have a larger bias and RMSE than the forward estimator and sometimes also than the backward estimator. Overall, though, $\pnAhh$ shows a more robust performance than both these estimators.

\begin{figure}[b!]
	\includegraphics[width=1\textwidth]{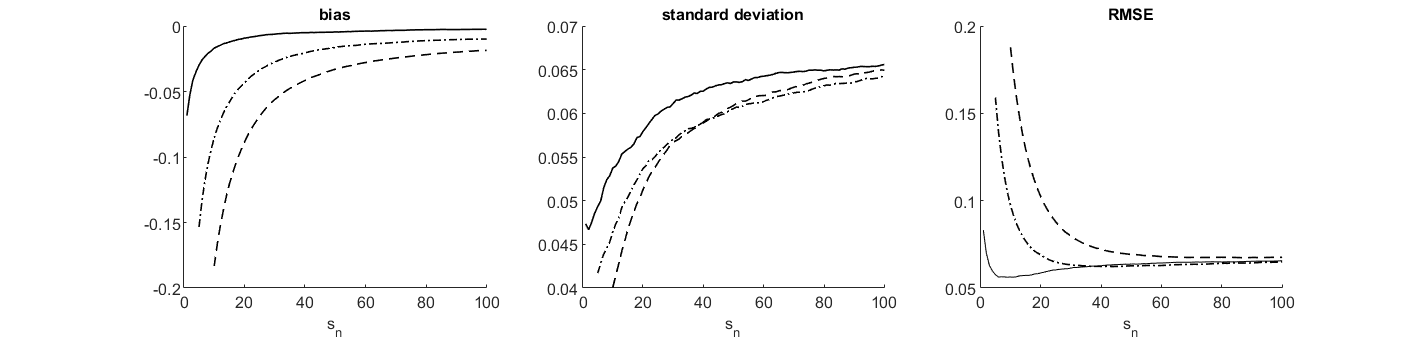}
	\includegraphics[width=1\textwidth]{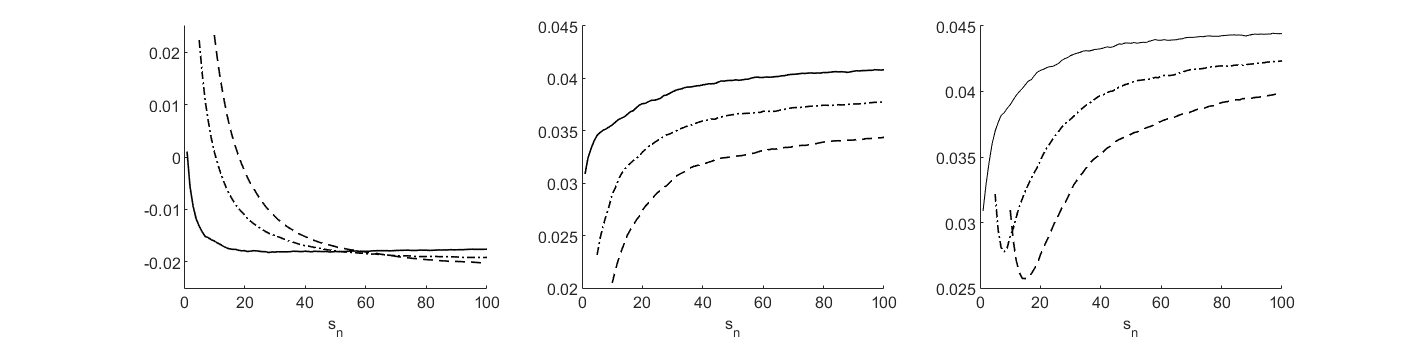}
	\caption{Bias (left), standard deviation (middle) and RMSE (right) of $\pnAhh$  versus $s_n$, for lag $1$ (solid line), lag $5$ (dashed-dotted line) and lag $10$ (dashed line) with $x=0$ (top) and $x=1/2$ (bottom) in the
 GARCHt model.}
	\label{figure:snvgl-garcht}
\end{figure}

\begin{figure}[b!]
	\includegraphics[width=1\textwidth]{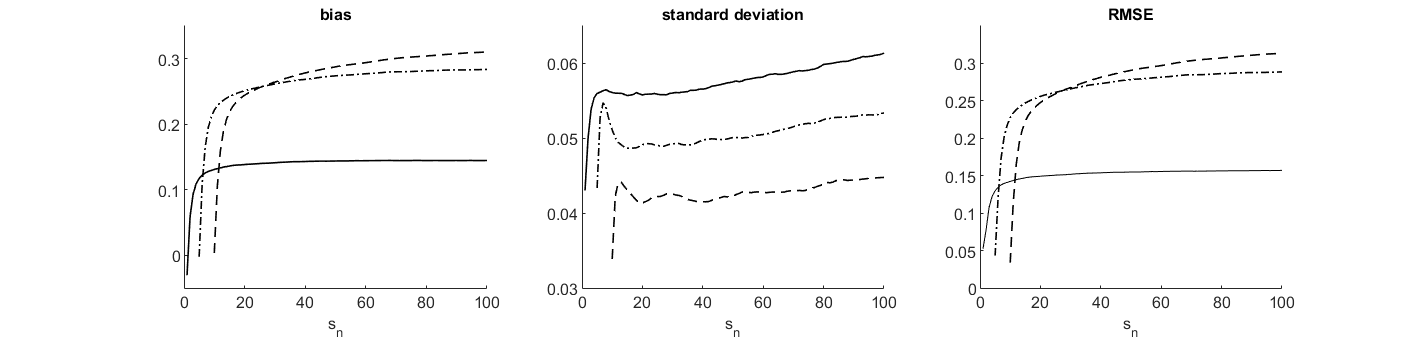}
	\includegraphics[width=1\textwidth]{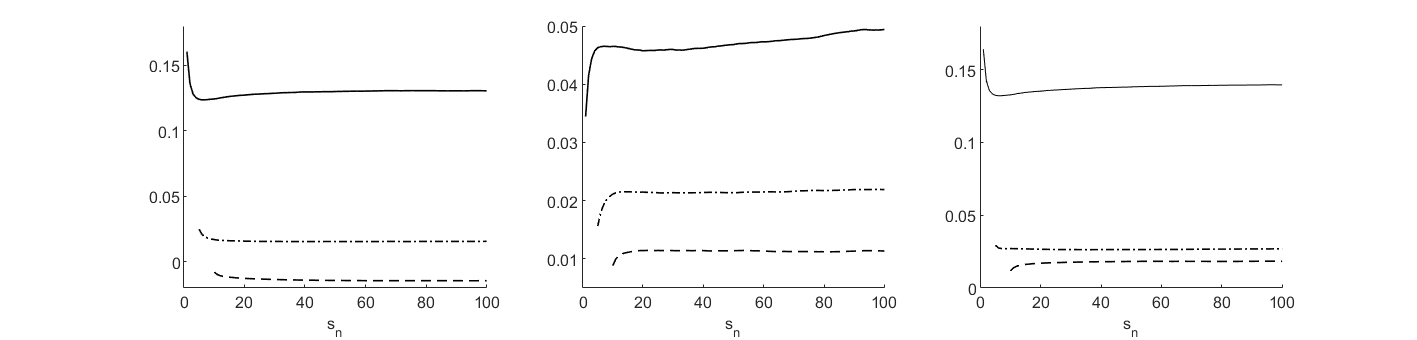}
	\caption{Bias (left), standard deviation (middle) and RMSE (right) of $\pnAhh$  versus $s_n$, for lag $1$ (solid line), lag $5$ (dashed-dotted line) and lag $10$ (dashed line) with $x=0$ (top) and $x=1/2$ (bottom) in the
 SRE model.}
	\label{figure:snvgl-sr}
\end{figure}

Recall that in the definition of the projection based estimator one has to choose $s_n>i$ which determines the length of the blocks. Next, we examine the sensitivity of the estimators to changes of the block length.

If blocks are chosen too small than they do not capture the dependence structure of a typical cluster of extremes. Moreover, a point mass at 0 is introduced as an artefact of the artificial clipping of the cluster. On the other hand, if the block length is chosen too large, then almost independent clusters of large observations are compounded, which may  lead to a bias and increase the variance, too.

Figure~\ref{figure:snvgl-garcht} shows the bias, the standard deviation and the RMSE of $\pnAhh$ as function of $s_n$ for the lags $i\in\{1,5,10\}$ and $x\in\{0,0.5\}$  for the GARCHt model and Figure~\ref{figure:snvgl-sr} does the same for the SRE model. For the GARCHt model, especially for larger lags, the estimator is clearly biased if  $s_n$ is chosen too small (see top left plot in Figure~\ref{figure:snvgl-garcht}), but for $s_n\ge 30$ the bias is small and the RMSE is quite stable. In the SRE model, in most cases the bias is caused by the deviation of the pre-asymptotic cdf from the limit cdf, discussed above. This negative bias happens to (partly) cancel a positive bias caused by too short a block out if $s_n$ is chosen small, leading to a very small RMSE. Apart from this very specific effect, the RMSE is again rather stable if $s_n$ is not too small. We thus recommend to choose the block length not too small (in particular for larger lags $i$), but unless an excessively large value is used, the performance of the estimator is not very sensitive to this tuning parameter.

To guide the selection of a reasonable value, a graph analogous to a Hill plot may be useful. In Figure~\ref{figure:snvgl-single},  for a single a time series according to the GARCHt model, the value of $\pnAhh$ is plotted versus $s_n$, for $A=(-\infty,-1]$ (left) and $A=(-\infty,1/2]$ (right). We suggest  to choose $s_n$ in a range where the plots starts to stabilize, which in this case again leads to a value of about 30.

\begin{figure}[t!]
	\includegraphics[width=1\textwidth]{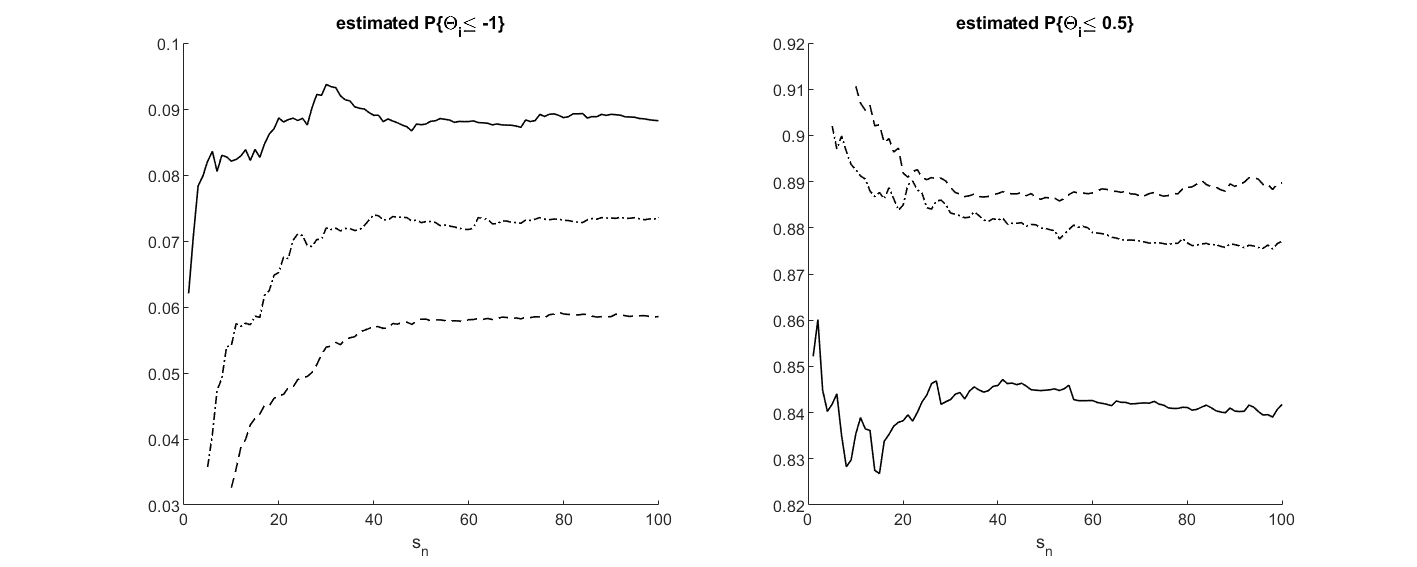}
\caption{ $\pnAhh$ versus $s_n$ for a single realization of the GARCHt model for lag $1$ (solid line), lag $5$ (dashed-dotted line) and lag $10$ (dashed line) with $x=-1$ (left) and $x=0.5$ (right). }
	\label{figure:snvgl-single}
\end{figure}

As always when the peaks-over-threshold approach is employed, all three estimators under consideration depend on the
selected threshold $u_n$. So far, we have fixed this tuning parameter to the empirical $95\%$-quantile of $|X_0|$.
Figure~\ref{figure:nivgl-x01} displays the performance of all estimators for the cdf at $x=0.1$ in the GARCHt and the SRE model for lag 1 for different quantile levels. Unless the threshold is chosen very high, the performance of the projection based estimator and forward estimator seems quite stable. To a lesser extent, this is also true for the backward estimator, but here the performance deteriorates a bit faster. (Note that here we consider a set $A$ whose boundary $x$ is close to the origin, while the backward estimator is known to perform better for sets clearly bounded away from 0.)

To conclude, apparently the newly proposed estimator is quite insensitive to the choice of the block length, and it is not more sensitive to threshold selection than the direct forward estimator.

\begin{figure}[t!]
	\includegraphics[width=1\textwidth]{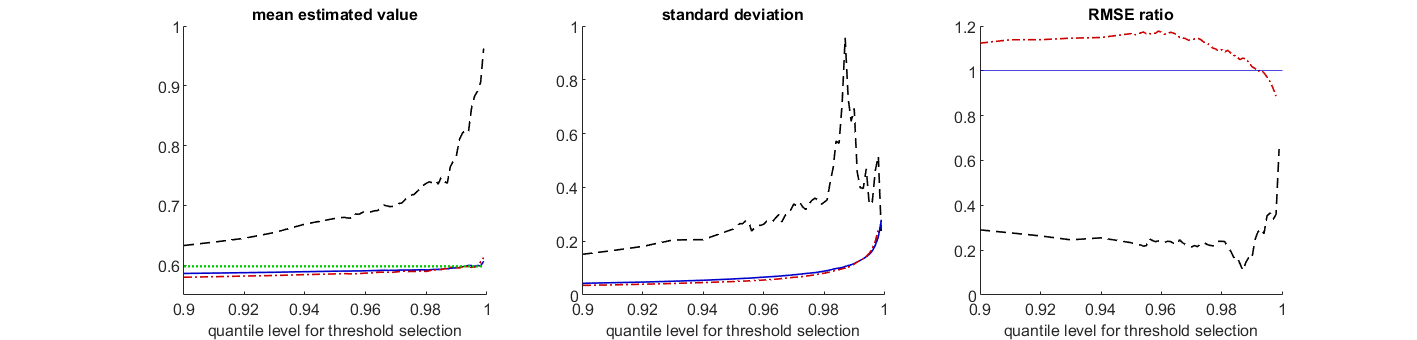}
	\includegraphics[width=1\textwidth]{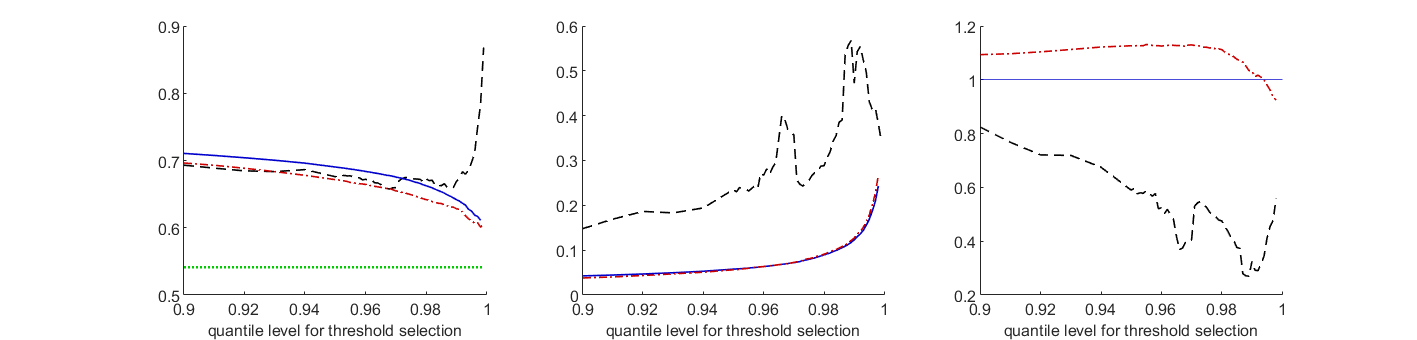}
	\caption{Mean (left), standard deviation (middle) and relative efficiency w.r.t.\  $\pnAhh$ (right) versus $F_{|X|}(u_n)$, of $\pnAhh$ (blue solid line), $\pnbh$ (black dashed line) and $\pnfh$ (red dashed-dotted line) in the
GARCHt model (top) and in the SRE model (bottom)  for lag $i=1$ and $x=0.1$; the true cdf is indicated by the green dotted line.}\label{figure:nivgl-x01}
\end{figure}

\section{Example: comparison of asymptotic variances.}
In this section we give further details for the calculation of the covariances in Example 2.4 of the main paper.
Recall that we defined the process $W=(W_t)_{t\in\mathbb{Z}}$ by
\begin{align}
	&P\{W_0=a^{-1},W_1=-1,W_t=0,\,\forall t\notin\{0,1\}\}=p,\\
	&P\{W_0=1,W_1=b^{-1},W_t=0,\,\forall t\notin\{0,1\}\}=1-p
\end{align}
for some $a> 1$, $b> 1$ and $p=[0,1]$ and we set $\alpha=1$ and define $\Theta=(\Theta_t)_{t\in\mathbb{Z}}$ such that $P^{\Theta}=(P^W)^{RS}$, which results in
\begin{align}
	P\{\Theta\in D\}
&=E\left[\frac{\|W_0\|}{\|W_0\|+\|W_1\|}\ind{D}\left(\frac{W}{\|W_0\|}\right)+\frac{\|W_1\|}{\|W_0\|+\|W_1\|}\ind{ D}\left(\frac{(W_{t+1})_{t\in\mathbb{Z}}}{\|W_1\|}\right)\right].
\end{align}
The specific choice $D=\{(y_t)_{t\in\mathbb{Z}}|y_{-1}=a^{-1},y_0=-1,y_t=0\,\forall t\notin\{-1,0\}\}$ yields
\begin{align}
	P\{\Theta_{-1}=a^{-1},\Theta_0=-1,\Theta_t=0,\,\forall t\notin\{-1,0\}\}&=p\frac{a}{a+1}.
\end{align}
Likewise, we obtain
\begin{align}
	&P\{\Theta_{0}=1,\Theta_1=-a,\Theta_t=0,\,\forall t\notin\{0,1\}\}=p\frac{1}{a+1},\\
	&P\{\Theta_{-1}=b,\Theta_0=1,\Theta_t=0,\,\forall t\notin\{-1,0\}\}=(1-p)\frac{1}{1+b},\\
	&P\{\Theta_{0}=1,\Theta_1=b^{-1},\Theta_t=0,\,\forall t\notin\{0,1\}\}=(1-p)\frac{b}{1+b}.
\end{align}
According to Theorem 2.3, for sets $A$ such that its boundary is disjoint to $B:=\{-a,0,a^{-1},b^{-1},1,b\}$,  the asymptotic variance of $\pnAhh$  equals
$$Var\big(Z(A)-(p_A-\alpha d_A)Z(\R)-\alpha^2 d_A\Zphi\big), $$
with $p_A=P\{\Theta_i\in A\}$ and $d_A, Z$ and $\Zphi$ given in Theorem 2.3 and (2.2).

Likewise, according to \cite{davis2018}, Theorem 3.1, for $A=(x,\infty)$ with $x\not\in B$, the asymptotic variance of $\pnbh$ is given by
$$Var\big(\tilde{Z}(A)-p_AZ(\mathbb{R}) + (\alpha^2\Zphi-\alpha Z(\mathbb{R}))E[\log(\|\Theta_i\|)\ind{A}(\Theta_i)]\big)$$
where $(\tilde{Z},\Zphi,Z(\mathbb{R}))$ is a centered Gaussian process with covariance specified by the formulas in Theorem 2.3 and
\begin{align}	&Var(\tilde{Z}(A))=E\left[\sum_{j\in\mathbb{Z}}(\|\Theta_j\|^\alpha\wedge 1)\left(\frac{\|\Theta_{j-i}\|^\alpha}{\|\Theta_j\|^\alpha}\ind{ A}\left(\frac{\Theta_j}{\|\Theta_{j-i}\|}\right)\right) \left(\|\Theta_{-i}\|^\alpha\ind{A}\left(\frac{\Theta_0}{\|\Theta_{-i}\|}\right)\right)\right],\\
&Cov(\tilde{Z}(A),Z(\mathbb{R}))=E\left[\sum_{j\in\mathbb{Z}}(\|\Theta_j\|^\alpha\wedge 1) \left(\|\Theta_{-i}\|^\alpha\ind{A}\left(\frac{\Theta_0}{\|\Theta_{-i}\|}\right)\right)\right]\\
	&Cov(\tilde{Z}(A),Z_\phi)=E\left[\sum_{j\in\mathbb{Z}}(\|\Theta_j\|^\alpha\wedge 1)(\log(\|\Theta_j\|\vee 1)+\alpha^{-1}) \|\Theta_{-i}\|^\alpha\ind{A}\left(\frac{\Theta_0}{\|\Theta_{-i}\|}\right)  \right].
\end{align}
Finally, from \cite{davis2018}, Theorem 3.1, it is known that in this setting the variance of the forward estimator $\pnfh$ equals
\begin{align}
	\sum_{j\in\mathbb{Z}} E\left[  ( \|\Theta_j\|^{\alpha}\wedge 1) \left(p_A- \mathds{1}_{ A}\left(\frac{\Theta_{j+i}}{\|\Theta_{j}\|}\right) \right) \left(p_A- \mathds{1}_{A}\left(\Theta_{i}\right)\right)\right].
\end{align}
All these expectations can be easily calculated numerically, since the spectral tail process $(\Theta_t)_{t\in\Z}$ vanishes for $|t|>1$ and its distribution is discrete with just 4 points of mass.

 Note that the asymptotic variances of the backward estimator and the projection based estimator  with known $\alpha$ can be calculated likewise, by omitting all terms involving $d_A$ or $E[\log(\|\Theta_i\|)\ind{\{\Theta_i\in A\}}]$.

\section{Verification of (BC), (BC') and (M) (i) for solutions to stochastic recurrence equations.}

In Subsection 2.1 of the main paper we discussed stationary solutions to  stochastic recurrence equations
\begin{align} \label{eq:SREdef}
	X_t=C_t X_{t-1}+D_t, \quad t\in\Z,
\end{align}
where  $C_t$ are  random $d\times d$-matrices with non-negative entries and $D_t$ are $[0,\infty)^d$-valued random vectors so that $(C_t,D_t)$, $t\in\Z$ are i.i.d. The Conditions (SRE) and (SRE') have been used to conclude the existence of a stationary regular varying solution and that $(X_t)_{t\in\mathbb{Z}}$ is an aperiodic, positive Harris recurrent, $P^{X_0}$-irreducible Feller process. Recall
\begin{align} \label{eq:SRErec}
	X_t=\Pi_{1,t}X_0+R_t
\end{align}
where $\Pi_{j,k}:=C_k\cdots C_j$ and $R_t:= \sum_{j=1}^t \Pi_{j+1,t}D_j$ are independent of $X_0$.

Here we give more details why Conditions (BC), (BC') and (M)~(i) are satisfied in this setting.
As in the paper, $\|\cdot\|$ denotes both the Euclidean norm of a vector and the corresponding operator norm of a matrix.
Recall that for all $q\in (0,\alpha)$ one has for sufficient large $m$
 $$\kappa:=E(\|\Pi_{1,m}\|^q)<1$$
 and
\begin{align}
	\label{eq:SREbound.pi.k}
	E[\|\Pi_{1,k}\|^q] \leq   c_m \tilde\kappa^k
\end{align}
for all $k\in\mathbb{N}$, some $c_m>0$ and $\tilde\kappa:=\kappa^{1/m}<1$.

Fix some $c\in(0,1]$ and let $v_{n,c}:=P\{\|X_0\|>cu_n\}$. In view of \eqref{eq:SRErec}
\begin{align}
		&P(\|X_k\|>cu_n\mid \|X_0\|>cu_n)\\
  & \leq \frac{1}{v_{n,c}}P\{\|X_0\|>cu_n,\|R_k\|+\|\Pi_{1,k}X_0\|>cu_n\}\\
		&\leq \frac{1}{v_{n,c}}\Big(P\Big\{\|X_0\|>cu_n,\|R_k\|>\frac{cu_n}{2}\Big\}+P\Big\{\|X_0\|>cu_n,\|\Pi_{1,k}\|\|X_0\|>\frac{cu_n}{2}\Big\}\Big)\\
		&=P\Big\{\|R_k\|>\frac{cu_n}{2}\Big\}+\frac{1}{v_{n,c}}\int_{cu_n}^\infty P\Big\{\|\Pi_{1,k}\|>\frac{cu_n}{2t}\Big\}P^{\|X_0\|}(dt).
\end{align}
Since we assume that $C$ and $D$ only have non-negative entries, the first summand can be bounded by $P\{\|X_k\|>cu_n/2\}=v_{n,c/2}$. For the second term,  the generalized Markov inequality and \eqref{eq:SREbound.pi.k} yield
\begin{align}
	P\Big\{\|\Pi_{1,k}\|>\frac{cu_n}{2t}\Big\}\leq E[\|\Pi_{1,k}\|^q]\Big(\frac{2t}{cu_n}\Big)^q\leq c_m \tilde\kappa^k \Big(\frac{2t}{cu_n}\Big)^q.
\end{align}
Hence,
\begin{align}
	\frac{1}{v_{n,c}}\int_{cu_n}^\infty P\Big\{\|\Pi_{1,k}\|>\frac{cu_n}{2t}\Big\}P^{\|X_0\|}(dt) \leq 2^q  c_m \tilde\kappa^k E\bigg(\Big(\frac{\|X_0\|}{cu_n}\Big)^q\Big|  \|X_0\|>cu_n \bigg)
\end{align}
for all $k\in \mathbb{N}$. To sum up, we have shown
\begin{align}
	P\{\|X_k\|>cu_n\mid \|X_0\|>cu_n\}\leq v_{n,c/2}+2^q  c_m\tilde\kappa^k E\bigg(\Big(\frac{\|X_0\|}{cu_n}\Big)^q\Big| \|X_0\|>cu_n \bigg)=:e_{n,c}(k)
\end{align}
which is (2.11) in the main paper.

Because of regular variation of $\|X_0\|$, $q<\alpha$ and $r_nv_n\to 0$, we have $E((\|X_0\|/cu_n)^q\mid  \|X_0\|>cu_n)\to E(\|Y_0\|^q)$ and $r_nv_{n,c/2}=r_nv_n(v_{n,c/2}/v_n)\to 0$. Therefore
$$e_{\infty,c}(k)=\lim_{n\to\infty}e_{n,c}(k)=2^q  c_m\tilde\kappa^k E(\|Y_0\|^q)<\infty $$
and
\begin{align}
	\sum_{k=1}^{r_n}e_{n,c}(k)&=r_nv_{n,c/2}+2^q  c_m E\bigg(\Big(\frac{\|X_0\|}{cu_n}\Big)^q\Big| \|X_0\|>cu_n \bigg) \sum_{k=1}^{r_n} \tilde\kappa^k\\
	&\rightarrow 2^q  c_m E(\|Y_0\|^q) \Big(\frac{1}{1-\tilde{\kappa}}-1\Big)=\sum_{k=1}^{\infty}e_{\infty,c}(k)<\infty,
\end{align}
i.e.\ condition (BC).

Next observe that, for $p\in (0,\alpha)$, from the drift condition $E(\|X_1\|^p\mid X_0=y)\le \beta \|y\|^p+b$, i.e.\ Assumption 2.1~(iii) of \cite{KSW19} which has been established in Subsection 2.1 with some $\beta\in (0,1)$ and some $b>0$, one may conclude by induction that $E(\|X_k\|^p\mid X_0=y)\le \beta^k \|y\|^p+b/(1-\beta)$ (\cite{douc2018}, Proposition 14.1.8). For all $p,\tilde p>0$ such that $p+\tilde p<\alpha$, it follows
\begin{align}
     E\Big( & \Big(\frac{\|X_k\|}{u_{n}}\Big)^p \Big(\frac{\|X_0\|}{u_{n}}\Big)^{\tilde p} \Ind{\|X_k\|>u_{n}} \ \Big|\  \|X_0\|>u_{n}\Big)\\
       & \le  v_{n}^{-1}\int_{u_{n}}^\infty u_{n}^{-(p+\tilde p)}E(\|X_k\|^p\mid \|X_0\|=y) \|y\|^{\tilde p}\ P^{\|X_0\|}(\mathrm{d}y)\\
       & \le \beta^k E\Big(\Big(\frac{\|X_0\|}{u_{n}}\Big)^{p+\tilde p}\ \Big|\ \|X_0\|>u_{n}\Big) + \frac{b}{1-\beta}u_{n}^{-p}E\Big(\Big(\frac{\|X_0\|}{u_{n}}\Big)^{\tilde p}\ \Big|\  \|X_0\|>u_{n}\Big)\\
       & \le 2\beta^k E(\|Y_0\|^{p+\tilde p}) + \frac{2b}{1-\beta}u_{n}^{-p}E(\|Y_0\|^{\tilde p})
       \label{eq:MiSRE}
  \end{align}
for sufficiently large $n$.

Define $\psi(x):=\max(\log \|x\|, \ind{[1,\infty)}(\|x\|))$, which for all $p>0$ can be bounded in absolute value by $\tilde c_p \|x\|^p \Ind{\|x\|>1}$ for some constant $\tilde c_p$. The uniform moment bound  \eqref{eq:MiSRE} readily shows that the random variables $\psi(X_0/u_n)\psi(X_k/u_n)/v_n$ are uniformly integrable, so that the definition of the tail process yields
$$ e'_n(k):=E(\psi(X_0/u_n)\psi(X_k/u_n)\mid \|X_0\|>u_n)\to e'_\infty(k) := E[\psi(Y_0)\psi(Y_k)]<\infty
$$
for all $k\in\mathbb{N}$. Moreover, the representation of the forward tail process $Y_k=\|Y_0\|\Theta_0\Pi_{1,k}$ and \eqref{eq:SREbound.pi.k} show that the $e'_\infty(k)$ are summable:
$$ \sum_{k=1}^\infty e'_\infty(k) \le \tilde c_{q/2}^2 E[\|Y_0\|^q]\sum_{k=1}^\infty E[\|\Pi_{1,k}\|^q]\le \tilde c_{q/2}^2 c_m E[\|Y_0\|^q]\sum_{k=1}^\infty \tilde\kappa^k <\infty. $$
Hence, for all $\varepsilon>0$, one can find some $L\in\N$ such that $\sum_{k=L+1}^\infty e'_\infty(k)<\varepsilon$. Using \eqref{eq:MiSRE} and convergence (2.10), one may conclude for sufficiently large $L'\ge L$
\begin{align}
  \limsup_{n\to\infty} \Big|\sum_{k=1}^{r_n} e'_n(k) - \sum_{k=1}^\infty e'_\infty(k)\Big|
    & \le  \limsup_{n\to\infty} \Big|\sum_{k=1}^{L'} e'_n(k) - \sum_{k=1}^{L'} e'_\infty(k)\Big| + 2\varepsilon
     = 2\varepsilon.
\end{align}
Let $\varepsilon$ tend to 0 to obtain (BC').

Finally, \eqref{eq:MiSRE} with  $p\in (\alpha(1+\delta)/(1+\zeta),\alpha)$, $\delta<\zeta$, and $\tilde p\in (0,\alpha-p)$ shows that the sum on the left-hand side of Condition (M) (i) can be bounded by a multiple of
\begin{align}
  & \sum_{k=1}^{r_n} \Big(\beta^k E[\|Y_0\|^{p+\tilde p}] + u_{n}^{-p}E[\|Y_0\|^{\tilde p}]\Big)^{1/(1+\delta)} \\
  & \le \sum_{k=1}^{r_n} (\beta^{1/(1+\delta)})^k (E[\|Y_0\|^{p+\tilde p}] )^{1/(1+\delta)} + r_n u_n^{-p/(1+\delta)} (E[\|Y_0\|^{\tilde p}])^{1/(1+\delta)}\\
  & = \Ord(1),
\end{align}
provided $r_n^{1+\zeta}v_n$ is bounded, because $u_n$ is of larger order than $v_n^{\eta-1/\alpha}$ for all $\eta>0$.

\section{Modified sliding blocks limit theorem.}
In this section we present a slightly modified version of a limit theorem for unbounded functions of sliding blocks introduced by \cite{DN20} (D\&N), Theorem 2.4, which is applied in the proof of Theorem 2.3 of the main paper.
To this end, assume that $(X_{n,i})_{1-s_n\leq i\leq n+s_n, n\in\mathbb{N}}$ is a triangular array of row-wise stationary random variables, consider a set $\mathcal{G}$ of functionals defined on vectors of arbitrary length and define $W_{n,j}:=(X_{n,j-s_n},...,X_{n,j+s_n})$ for some sequence $s_n$,
\begin{align}	
  V_{n,i}(g)&:= \frac{1}{b_n(g)}\sum_{j=1}^{r_n} g(W_{n,(i-1)r_n+j}),\\
 \tilde{V}_{n,i}(g) &:= \frac{1}{b_n(g)}\sum_{j=1}^{r_n-l_n} g(W_{n,(i-1)r_n+j}),\\
\bar{Z}_n(g)&:=\frac{1}{\sqrt{p_n}b_n(g)}\sum_{j=1}^{n}\left(g(W_{n,j})-E g(W_{n,j})\right),\\
	Z_n(g)	&:=\frac{1}{\sqrt{p_n}}\sum_{i=1}^{\floor{n/r_n}}\left(V_{n,i}(g)-E V_{n,i}(g)\right), \qquad g\in\mathcal{G},
\end{align}
and $p_n:=P\{\exists g\in\mathcal{G} : V_{n}(g)\neq 0\}$ for some suitable sequence $r_n=\ord(n)$; here $V_n$ denotes a random variable with the same distribution as $V_{n,1}$. Moreover, we use the following conditions taken from  Section 2 and Appendix A of D\&N:
\begin{itemize}
	\item[\bf(A1)] $(X_{n,i})_{1\leq i\leq n}$ is stationary for all $n\in\mathbb{N}$.
	
	\item[\bf(A2)] The sequences $l_n,r_n,s_n\in\N$, $p_n$, and $b_n(g)>0$, $g\in\GG$, satisfy $s_n\le l_n=\ord(r_n)$, $r_n=\ord(n)$, $p_n\to 0$ and $r_n=\ord\big(\sqrt{p_n} \inf_{g\in\GG}b_n(g)\big)$.
	\item[\bf(D0)] The processes $V_{n}$, $n\in\N$, are separable.
	\item[\bf(MX)] $m_n\beta_{n,l_n-(2s_n+1)}^{X}\to 0$ for $m_n:=\floor{n/r_n}$.
	\item[\bf(C)] There exists a function $c:\GG^2\to\R$ such that
	\begin{equation}
		\frac{m_n}{p_n}Cov\left(V_{n}(g),V_{n}(h)\right)\rightarrow c(g,h), \qquad \forall\, g,h\in\mathcal{G}.
	\end{equation}

\item[\bf($\mathbf{\Delta}$)] $\Delta_n:=V_n-\tilde V_n$ satisfies

(i)
$\displaystyle
E\left[(\Delta_n(g)-E[\Delta_n(g)])^2\mathds{1}_{\left\{ |\Delta_n(g)-E[\Delta_n(g)]|\leq \sqrt{p_n}\right\}}\right]=\ord\left(p_n/m_n\right), \qquad \forall\, g\in\mathcal{G},
$
\smallskip

(ii)
$\displaystyle
P\left\{|\Delta_n(g)-E[\Delta_n(g)]|> \sqrt{p_n}\right\}=\ord\left(1/m_n\right), \qquad \forall\, g\in\mathcal{G}.
$

\smallskip

(iii)
$\displaystyle E\left[(\Delta_n(g)-E[\Delta_n(g)])\mathds{1}_{\left\{ |\Delta_n(g)-E[\Delta_n(g)]|\leq \tau \sqrt{p_n}\right\}}\right]=\hbox{o}\left(\sqrt{p_n}/m_n\right) ,  \qquad \forall g\in\mathcal{G}$ for some $\tau>0$.
\smallskip

	\item[\bf (L)]
$\displaystyle
E\left[(V_n(g)-E[V_n(g)])^2 \mathds{1}_{\left\{ |V_n(g)-E[V_n(g)]|>\epsilon\sqrt{p_n} \right\}}\right]=\ord\left(p_n/m_n\right), \quad  \forall\, g\in\mathcal{G}, \epsilon>0.
$


\item[\bf(D1)] There exists a semi-metric $\rho$ on $\mathcal{G}$ such that $\mathcal{G}$ is totally bounded and
\begin{equation}
	\lim_{\delta\downarrow 0} \limsup_{n\to\infty} \sup_{g,h\in\mathcal{G},\rho(g,h)<\delta} \frac{m_n}{p_n } E\big[(V_n(g)-V_n(h))^2\big] = 0.
\end{equation}

\item[\bf(D2)]
\begin{equation}
	\lim_{\delta\downarrow 0}\limsup_{n\to\infty}\int_{0}^{\delta} \sqrt{\log N_{[\cdot]}(\epsilon,\mathcal{G},L_2^n) } \, d\epsilon =0,
\end{equation}
where $N_{[\cdot]}(\epsilon,\mathcal{G},L_2^n)$ denotes the $\epsilon$-bracketing number of $\GG$ w.r.t.\ $L_2^n$,  i.e.\ the smallest number $N_\epsilon$ such that for each $n\in\mathbb{N}$ there exists a partition $(\mathcal{G}_{n,k}^{\epsilon})_{1\leq k\leq N_{\epsilon}}$ of $\mathcal{G}$ satisfying
\begin{equation}
	\frac{m_n}{p_n } E^*\Big[\sup_{g,h\in\mathcal{G}_{n,k}^{\epsilon}} (V_n(g)-V_n(h))^2\Big]\leq \epsilon^2,  \qquad\forall 1\leq k\leq N_{\epsilon}.
\end{equation}
 (Here $E^*$ denotes the outer expectation.)
\item[\bf(D3)] Denote by $N(\epsilon,\mathcal{G},d_n)$ the $\epsilon$-covering number of $\GG$ w.r.t.\ the random semi-metric
\begin{equation}	d_n(g,h)=\bigg(\frac{1}{p_n}\sum_{i=1}^{m_n}(V_{n,i}^*(g)-V_{n,i}^*(h))^2\bigg)^{1/2}
\end{equation}
with  $V_{n,i}^*$, $1\leq i\leq m_n$,  independent copies of $V_{n,1}$, i.e.\ $N(\epsilon,\mathcal{G},d_n)$ is the smallest number of balls with respect to $d_n$ with radius $\epsilon$ that is needed to cover $\mathcal{G}$. We assume
\begin{equation}
	\lim_{\delta\downarrow 0} \limsup_{n\to\infty} P^*\bigg\{\int_{0}^{\delta} \sqrt{\log(N(\epsilon,\mathcal{G},d_n))}d\epsilon>\tau\bigg\}=0,\qquad\forall \tau>0.
\end{equation}

\end{itemize}
Condition ($\Delta$) is always fulfilled if
\begin{equation}
	\label{eq:bed fuer c1 zweites moment}
	E\left[(\Delta_n(g))^2\right]=\hbox{o}\left(p_n/m_n\right), \quad \forall g\in\mathcal{G}.
\end{equation}

Then we can modify Theorem 2.4 of D\&N as follows.
\begin{theorem}
	\label{cor:g unbounded sliding block con}
	\begin{itemize}
		\item[(i)] Suppose the conditions (A1), (A2), (D0), (MX) and (C) are met. Moreover, assume condition (L) is satisfied and
		\begin{equation}
			\label{eq: bed g messbar bedingung sliding blocks}
			E\bigg[\bigg( \sum_{i=1}^{r_n}|g(W_{n,i})|\bigg)^{2}\bigg]=\Ord\bigg(\frac{p_nb_n^2(g)}{m_n}\bigg),\quad \forall\, g\in\GG.
		\end{equation}
		Then the fidis of $(Z_n(g))_{g\in\mathcal{G}}$ and of $(\bar{Z}_n(g))_{g\in\mathcal{G}}$ converge to the fidis of the Gaussian process $(Z(g))_{g\in\mathcal{G}}$ defined in Theorem 2.1 of D\&N.
		\item[(ii)] If, in addition, $g_{\max}=\sup_{g\in\mathcal{G}}|g|$ is measurable, $b_n(g)=b_n$ is the same for all $g\in\GG$, \eqref{eq: bed g messbar bedingung sliding blocks} holds for $g=g_{\max}$ and the conditions (D1) and (D2) or the conditions (D1) and (D3) are fulfilled, then the processes $(Z_n(g))_{g\in\mathcal{G}}$ and $(\bar{Z}_n(g))_{g\in\mathcal{G}}$ converge weakly to $(Z(g))_{g\in\mathcal{G}}$ uniformly.
	\end{itemize}
\end{theorem}
So, in comparison with Theorem 2.4 of D\&N, basically  condition (2.8) in that paper, which is the analog to \eqref{eq: bed g messbar bedingung sliding blocks} with some exponent strictly greater than 2, is replaced with the weaker condition \eqref{eq: bed g messbar bedingung sliding blocks}, a weak technical condition is omitted, and as a compensation condition (L) is added.
\begin{proof}[Proof of Theorem \ref{cor:g unbounded sliding block con}]
	The proof is basically the same as for Theorem 2.4 in D\&N. We apply Theorem A.1 of D\&N to the observations $X'_{n,t}:=X_{n,t-s_n}$, $1\le t\le n':=n+2s_n$, and block lengths $s'_n:=2s_n+1$ and $l_n':=l_n+s_n+1$ to establish fidi-convergence of $(Z_n(g))_{g\in\mathcal{G}}$. Only Condition ($\Delta$) must be verified, because (L) is assumed and the remaining conditions follow as in the proof of Theorem 2.1 of D\&N.
Since
	\begin{align}
		E\bigg[\bigg(\sum_{i=1}^{r_n} |g(W_{n,i})|\Bigg)^2\bigg]&
		\ge \sum_{j=1}^{\floor{r_n/l_n}} E\bigg[\bigg(\sum_{i=1}^{l_n} |g(W_{n,(j-1)l_n+i})|\bigg)^2\bigg]\\
		&= \floor{r_n/l_n} E\bigg[\bigg(\sum_{i=1}^{l_n} |g(W_{n,i})|\bigg)^2\bigg],
	\end{align}
	 \eqref{eq: bed g messbar bedingung sliding blocks} and $l_n=\ord(r_n)$ imply
	\begin{align}
		E(\Delta_n(g)^2) & \le \frac 1{b_n^2(g)} E\bigg[\bigg(\sum_{i=1}^{l_n} |g(W_{n,i})|\bigg)^2\bigg]
		 \le \frac 1{b_n^2(g) \floor{r_n/l_n}}E\bigg[\bigg(\sum_{i=1}^{r_n} |g(W_{n,i})|\bigg)^2 \bigg]\\
		& = \Ord\bigg( \frac{l_n}{r_n b_n^2(g)} \frac{p_n b_n^2(g)}{m_n}\bigg)=\ord\bigg( \frac{p_n}{m_n}\bigg).
	\end{align}
	
	Hence, equation \eqref{eq:bed fuer c1 zweites moment} holds, which in turn yields ($\Delta$). Now, the convergence of the fidis of $(Z_n(g))_{g\in\mathcal{G}}$ follows from Theorem A.1 of D\&N.
	Similarly,
	\begin{align}
		&E \Big((\bar Z_n(g)-Z_n(g))^2\Big)
		\leq \frac{1}{p_n b_n^2(g)} E\bigg[\bigg( \sum_{j=r_nm_n+1}^{n-s_n} |g(W_{n,j})| \bigg)^2\bigg]
		 = \Ord\bigg( \frac{1}{m_n}\bigg)\to 0,
	\end{align}
	so that the fidi-convergence of $(\bar{Z}_n(g))_{g\in\mathcal{G}}$ follows, too.
	
	Part (ii) of the assertion can be concluded by the same arguments as in the proof of Theorem 2.4 of D\&N.
\end{proof}

\section{Proof of (2.1) and (A.4).}

First we want to establish (2.1) of the paper, i.e. the convergence
\begin{align}
	\label{eq:excong}
	E(g_A(W_{n,0})\mid \|X_0\|>u_n) \to E\big(g_A(Y)\big)=E\Big( \|\Theta\|_\alpha^{-\alpha}\sum_{h\in\Z} \|\Theta_h\|^\alpha \ind{A}(\Theta_{h+i}/\|\Theta_h\|)\Big)
\end{align}
for $g_A:l_\alpha\to [0,1]$,
\begin{align}
	g_A\big((w_h)_{h\in \mathbb{Z} }\big) & := \frac{\Ind{\|w_0\|>1}}{\sum_{h\in\mathbb{Z}} \|w_h\|^\alpha} \sum_{h\in\mathbb{Z}} \|w_h\|^\alpha \ind{A}\Big(\frac{w_{h+i}}{\|w_h\|}\Big)
\end{align}
and $W_{n,t}:=(X_{n,t+h})_{|h|\leq s_n}$, $X_{n,t}:=X_t/u_n$. To this end, we define the approximating functions $g_A^{(m)}:l_\alpha\to\mathbb{R}$,
\begin{align}
	g_A^{(m)}\big((w_h)_{h\in \mathbb{Z} }\big) & := \frac{\Ind{\|w_0\|>1}}{\sum_{|h|\leq m} \|w_h\|^\alpha} \sum_{|h|\leq m} \|w_h\|^\alpha \ind{A}\Big(\frac{w_{h+i}}{\|w_h\|}\Big)
\end{align}
for all $m\in\mathbb{N}$.
As a finite sum of continuous functions, $g_A^{(m)}$ is $P^{(Y_{t+j})_{t\in\Z}}$-a.s. continuous for $j\in\mathbb{Z}$ if $P\{\exists h\in\mathbb{Z}: Y_{h+i}/\|Y_{h}\|\in \partial A, \|Y_h\|>0\}=0$ and $P\{\|Y_j\|=1\}=0$. While the former equality is ensured by (C$\Theta$) and Lemma A.3, the latter follows from $Y_j=\Theta_j\|Y_0\|$, where $\|\Theta_j\|$ and $\|Y_0\|$ are independent and $\|Y_0\|$ has a Pareto($\alpha$)-distribution. Here we only need the continuity for $j=0$, but below the general version is used in the proof of (A.4) for the continuity of $f$.

Since $g_A^{(m)}$ is bounded by $1$, the weak convergence $\mathcal{L}((X_{n,h})_{|h|\leq m})\mid \|X_0\|>u_n)\to\mathcal{L}((Y_h)_{|h|\le m})$ defining the tail process implies
\begin{align}
	\label{eq:gam.expeccon}
	E\big(g_A^{(m)}(W_{n,0})\mid \|X_0\|>u_n\big) \to E\big(g_A^{(m)}(Y)\big).
\end{align}
One has $g_A^{(m)}(w)\to g_A(w)$ as $m\to\infty$ for all $w=(w_h)_{h\in \mathbb{Z} }\in l_\alpha$. Since $|g_A^{(m)}|\leq 1$ for all $m\in\mathbb{N}$, dominated convergence implies $E(g_A^{(m)}(Y))\to E(g_A(Y))$ as $m\to\infty$.

Next, we prove that the difference between $g_A^{(m)}(W_{n,j})$ and $g_A(W_{n,j})$ is asymptotically negligible for all $j\in\Z$ as $n$ and $m$ tend to $\infty$. While here this result is only needed for $j=0$, the general version is used in the proof of (A.4) below.
Using (A.18) of the paper, one has for sufficiently large $n$ (such that $s_n\geq m+|i|+|j|$)
\begin{align}
	& \big|g_A^{(m)}(W_{n,j})-g_A(W_{n,j})\big|  \\
	&=\bigg|\frac{\sum_{|h|\leq m} \|X_{n,h+j}\|^\alpha \ind{A}\Big(\frac{X_{h+j+i}}{\|X_{h+j}\|}\Big)\sum_{m<|k|\leq s_n} \|X_{n,k+j}\|^\alpha}{\sum_{|h|\leq m} \|X_{n,h+j}\|^\alpha \sum_{|k|\leq s_n} \|X_{n,k+j}\|^\alpha} \\
	&\hspace{0.5cm} -\frac{\sum_{m< |h|\leq s_n} \|X_{n,h+j}\|^\alpha \Big( \Ind{h\in H_n} \ind{A}\Big(\frac{X_{h+j+i}}{\|X_{h+j}\|}\Big)+ \Ind{h\in H_n^C} \ind{A}(0) \Big) \sum_{|k|\leq m} \|X_{n,k+j}\|^\alpha}{\sum_{|h|\leq m} \|X_{n,h+j}\|^\alpha \sum_{|k|\leq s_n} \|X_{n,k+j}\|^\alpha}\bigg| \\
	&\le 2 \frac{\sum_{m<|h|\leq s_n} \|X_{h+j}\|^\alpha }{ \sum_{|h|\leq s_n} \|X_{h+j}\|^\alpha}.\label{eq:abschaetzung.ga}
\end{align}
For $c\in (0,1)$ from (TC), Condition (BC) implies
\begin{align}
	E\bigg( & \frac{\sum_{m<|h|\leq s_n} \|X_{h+j}\|^\alpha\Ind{\|X_{h+j}\|>c u_n} }{ \sum_{|h|\leq s_n} \|X_{h+j}\|^\alpha}\,\Big|\, \|X_0\|>u_n\bigg)\\
	&\leq \sum_{m<|h|\leq s_n} P(\|X_{h+j}\|>cu_n\mid \|X_{0}\|>cu_n)\frac{P(\|X_0\|>cu_n)}{P(\|X_0\|>u_n)}\\
	& \leq 4c^{-\alpha}\sum_{m<|h|\leq s_n} e_{n,c}(h+j)
\end{align}
for sufficiently large $n$, due to regular variation of $\|X_0\|$. Therefore,
\begin{align}
	&E\big( |g_A^{(m)}(W_{n,j})-g_A(W_{n,j})| \mid \|X_0\|>u_n\big)\\
	&\leq 4c^{-\alpha}\sum_{m<|h|\leq s_n} e_{n,c}(h+j) +E\left( \frac{\sum_{m<|h|\leq s_n} \|X_{h+j}\|^\alpha\Ind{\|X_{h+j}\|\leq c u_n} }{ \sum_{|h|\leq s_n} \|X_{h+j}\|^\alpha}\,\Big|\, \|X_0\|>u_n\right)\label{eq:weaken.bc1}
\end{align}
for sufficiently large $n\in\mathbb{N}$ and all $j\in\mathbb{Z}$.
Thus, condition (BC) and (TC) yield
\begin{align}
	\lim_{m\to\infty}\limsup_{n\to\infty} E\big(|g_A^{(m)}(W_{n,j})-g_A(W_{n,j})|\mid \|X_0\|>u_n\big)=0.
  \label{eq:gAm_Wnj_approx}
\end{align}
Combine this with \eqref{eq:gam.expeccon} and $E(g_A^{(m)}(Y))\to E(g_A(Y))$ to conclude \eqref{eq:excong}, i.e.\ (2.1) of the paper.

Next we verify (A.4) of the paper, i.e.
\begin{align}
	\label{eq:econfm}
	E\big(f\big(W_{n,0},W_{n,j})\mid \|X_0\|>u_n\big) \to E\big(f\big((Y_t)_{t\in\Z},(Y_{t+j})_{t\in\Z}\big)\big)
\end{align}
for all $j\in\mathbb{Z}$, with $f$ defined as $f=f_{A,B}:l_{\alpha}\times l_{\alpha}\to[0,1]$ with
\begin{align}
	f\big(&(y_t)_{t\in\mathbb{Z}},(z_{t})_{t\in\mathbb{Z}}\big) :=\mathds{1}_{\{\|y_0\|>1\}}\mathds{1}_{\{\|z_0\|>1\}}\\
	& \hspace{3cm}\times \Bigg(\sum_{h\in\mathbb{Z}}\frac{\|z_{h}\|^{\alpha}}{\sum_{k\in\mathbb{Z}} \|z_{k}\|^{\alpha }} \mathds{1}_{A}\Big(\frac{z_{h+i}}{\|z_{h}\|}\Big) \bigg) \bigg(\sum_{l\in\mathbb{Z}}\frac{\|y_{l}\|^{\alpha}}{\sum_{k\in\mathbb{Z}} \|y_{k}\|^{\alpha }} \mathds{1}_{B}\Big(\frac{y_{l+i}}{\|y_{l}\|}\Big)\Bigg).
\end{align}
Again we define approximating functions $f^{(m)}:l_\alpha\times l_\alpha\to\mathbb{R}$ by
\begin{align}
	f^{(m)}\big(&(y_t)_{t\in\mathbb{Z}},(z_{t})_{t\in\mathbb{Z}}\big) :=\mathds{1}_{\{\|y_0\|>1\}}\mathds{1}_{\{\|z_0\|>1\}}\\
	& \hspace{2cm}\times \Bigg(\sum_{|h|\leq m}\frac{\|z_{h}\|^{\alpha}}{\sum_{|k|\leq m} \|z_{k}\|^{\alpha }} \mathds{1}_{A}\Big(\frac{z_{h+i}}{\|z_{h}\|}\Big) \bigg) \bigg(\sum_{|l|\leq m}\frac{\|y_{l}\|^{\alpha}}{\sum_{|k|\leq m} \|y_{k}\|^{\alpha }} \mathds{1}_{B}\Big(\frac{y_{l+i}}{\|y_{l}\|}\Big)\Bigg),
\end{align}
for all $m\in\mathbb{N}$.
Observe that  $f^{(m)}(y,z)=g_B^{(m)}(y)g_A^{(m)}(z)$ and $f(y,z)=g_B(y)g_A(z)$, so that $f^{(m)}$ is $P^{((Y_t)_{t\in\Z},(Y_{t+j})_{t\in\Z})}$-a.s.\ continuous and bounded by $1$. Thus, the definition of the tail process implies
\begin{align}
	\label{eq:fm.expecconv}
	E\big(f^{(m)}\big(W_{n,0},W_{n,j})\mid \|X_0\|>u_n\big) \to E\big(f^{(m)}\big((Y_t)_{t\in\Z},(Y_{t+j})_{t\in\Z}\big)\big)
\end{align} for all $m\in\mathbb{N}$. Moreover,
\begin{align}
	\big| & f^{(m)}(W_{n,0},W_{n,j})-f(W_{n,0},W_{n,j})\big|\\
	&\leq  g_A^{(m)}(W_{n,0})|g_B^{(m)}(W_{n,j})-g_B(W_{n,j})|
	 +|g_A^{(m)}(W_{n,0})-g_A(W_{n,0})| g_B(W_{n,j})\\
	&\leq |g_B^{(m)}(W_{n,j})-g_B(W_{n,j})|+|g_A^{(m)}(W_{n,0})-g_A(W_{n,0})|.
\end{align}
Thus, in view of \eqref{eq:gAm_Wnj_approx}, for all $j\in\Z$, $$\lim_{m\to\infty}\limsup_{n\to\infty}E\Big(\big| f^{(m)}(W_{n,0},W_{n,j})-f(W_{n,0},W_{n,j})\big|\,\Big|\, \|X_0\|>u_n\Big)=0.$$
Combining this with \eqref{eq:fm.expecconv} and $E(f^{(m)}((Y_t)_{t\in\Z},(Y_{t+j})_{t\in\Z}))\to E(f((Y_t)_{t\in\Z},(Y_{t+j})_{t\in\Z}))$ as $m\to\infty$, which holds due to dominated convergence and $|f^{(m)}|\leq 1$ for all $m\in\mathbb{N}$, yields \eqref{eq:econfm}, i.e. (A.4) of the paper.

\section{Proof of Remark 2.2.}

We start with part (a). Since $Y_h=\Theta_h\|Y_0\|$ and $|a+b|^{1+\delta}\leq 2^{1+\delta}(|a|^{1+\delta}+|b|^{1+\delta})$ for $a,b\in\R$ and $\delta\ge 0$, we  conclude
\begin{align}
	E\bigg(& \frac{ \sum_{|h|\leq m} |\log\|Y_{h}\||^{1+\delta}\|Y_{h}\|^\alpha}{\sum_{|k|\leq m}\|Y_k\|^\alpha}\bigg)\\
	& =E\bigg(\frac{ \sum_{|h|\leq m} |\log\|\Theta_{h}\|+\log\|Y_0\||^{1+\delta}\|\Theta_{h}\|^\alpha}{\sum_{|k|\leq m}\|\Theta_k\|^\alpha}\bigg)\\
	&\leq 2^{1+\delta}\bigg(E\bigg(\frac{ \sum_{|h|\leq m} |\log\|\Theta_{h}\||^{1+\delta}\|\Theta_{h}\|^\alpha}{\sum_{|k|\leq m}\|\Theta_k\|^\alpha}\bigg)+ E\bigg(\frac{ \sum_{|h|\leq m} |\log\|Y_0\||^{1+\delta}\|\Theta_{h}\|^\alpha}{\sum_{|k|\leq m}\|\Theta_k\|^\alpha}\bigg)\bigg)\\
	&=2^{1+\delta}\bigg(E\bigg(\frac{ \sum_{|h|\leq m} |\log\|\Theta_{h}\||^{1+\delta}\|\Theta_{h}\|^\alpha}{\sum_{|k|\leq m}\|\Theta_k\|^\alpha}\bigg)+ E\big( |\log\|Y_0\||^{1+\delta}\big)\bigg)<\infty,
\end{align}
where in the last expression the first expectation is finite due to (M) (ii) and the second expectation, because $\log \|Y_0\|$ is exponentially distributed. Thus, Condition (M)~(ii) implies (2.5) of the paper.

Conversely, (2.5) of the paper implies
\begin{align}
	&E\bigg(\frac{ \sum_{|h|\leq m} |\log\|\Theta_{h}\||^{1+\delta}\|\Theta_{h}\|^\alpha}{\sum_{|k|\leq m}\|\Theta_k\|^\alpha}\bigg)\\
	&\leq2^{1+\delta}\bigg(E\bigg(\frac{ \sum_{|h|\leq m} |\log\|Y_{h}\||^{1+\delta}\|Y_{h}\|^\alpha}{\sum_{|k|\leq m}\|Y_k\|^\alpha}\bigg)+ E\big( |\log\|Y_0\||^{1+\delta}\big)\bigg)<\infty.
\end{align}
Thus, Condition (M) (ii) is equivalent to (2.5) of the main paper.

Next we turn to part (b) of Remark 2.2, which states that (M') implies (M) (ii) and (TC).
According to (a), it suffices to establish (2.5) of the paper, which follows from the definition of the tail process if we prove uniform integrability of
$$ \frac{\Ind{\|X_0\|>u_n}}{v_n}\cdot \frac{\sum_{|h|\leq (s_n\wedge m)} |\log\|X_{n,h}\||^{1+\delta} \|X_{n,h}\|^\alpha}{\sum_{|k|\leq (s_n\wedge m)}\|X_{n,k}\|^\alpha}, \quad n,m\in\mathbb{N}.
$$
 This, in turn, is implied by the uniform moment bound
\begin{align}
	\label{mii-iv-new}
	\sup_{m\in\mathbb{N}}\sup_{n\in\mathbb{N}}E\bigg( \bigg(\frac{\sum_{|h|\leq (s_n\wedge m)} |\log\|X_{n,h}\||^{1+\delta}\|X_{n,h}\|^\alpha}{\sum_{|k|\leq (s_n\wedge m)} \|X_{n,k}\|^{\alpha}}\bigg)^{1+\eta}\,\Big|\, \|X_0\|>u_n\bigg)<\infty
\end{align}
for $\eta=(\delta'-\delta)/(1+\delta)>0$. By Jensen's inequality, the expectation can be bounded by
\begin{align}
	&E\bigg( \frac{\sum_{|h|\leq (s_n\wedge m)} |\log\|X_{n,h}\||^{(1+\delta)(1+\eta)}\|X_{n,h}\|^\alpha}{\sum_{|k|\leq (s_n\wedge m)} \|X_{n,k}\|^{\alpha}}\,\Big|\, \|X_0\|>u_n\bigg).
\end{align}
On the set $\{\|X_0\|>u_n\}=\{\|X_{n,0}\|>1\}$, the term in the expectation can be bounded  by
\begin{align}
	&\sup_{|h|\leq (s_n\wedge m)}(\log\|X_{n,h}\|)^{1+\delta'}\Ind{\|X_{n,h}\|> 1} +
	\sum_{|h|\leq (s_n\wedge m)} |\log\|X_{n,h}\||^{1+\delta'}\|X_{n,h}\|^{\alpha}\Ind{\|X_{n,h}\|\le 1}\\
	&\leq \sup_{|h|\leq s_n}(\log^+\|X_{n,h}\|)^{1+\delta'} +
	\sum_{h=-s_n}^{s_n} (\log^-\|X_{n,h}\|)^{1+\delta'}\|X_{n,h}\|^{\alpha}.
\end{align}
Thus, Conditions (M') (i) and (ii) imply \eqref{mii-iv-new}, and hence in turn condition (M) (ii).

To verify  (TC) under (M') (ii), choose $c<e^{-1}$ so that $\Ind{\|X_h\|\le cu_n}\leq (\log^-(\|X_h\|/u_n))^{1+\delta'}$. Since, for $s_n\geq |j|$, the denominator is at least $u_n^\alpha$, it directly follows
\begin{align}
	&E\bigg( \frac{\sum_{m<|h|\leq s_n} \|X_{h+j}\|^\alpha\Ind{\|X_{h+j}\|\leq cu_n} }{ \sum_{|h|\leq s_n} \|X_{h+j}\|^\alpha}\Big| \|X_0\|>u_n\bigg)\\
	&\leq \sum_{|k|\leq s_n+j} E\bigg( \Big(\log^-\frac{\|X_{k}\|}{u_n}\Big)^{1+\delta'}\Big(\frac{\|X_k\|}{u_n}\Big)^\alpha\,\Big|\, \|X_0\|>u_n\bigg).
\end{align}
Because $x\mapsto (\log^- x)^{1+\delta'} x^\alpha$ is a bounded function on $(0,\infty)$ and thus the sum over $k$ with $s_n<|k|\le s_n+j$ is bounded, Condition  (M')~(ii)  ensures that (TC) holds.

\section{Proof of Lemma A.2.}

Next, we give the proof for Lemma A.2 of the main paper. This lemma shows that the well-known anticlustering condition is implied by our condition (BC) which also restricts the size of clusters of extremes.
\begin{lemma}
	If (RV) and (BC) holds, then the so-called anticlustering condition
	\begin{align}
		\label{eq:anticlustering bedingung}
		\lim_{m\to\infty}\limsup_{n\to\infty}P\Big(\max_{m\leq |t|\leq r_n}\|X_t\|>cu_n \,\Big|\, \|X_0\|>cu_n \Big)=0
	\end{align}
	is satisfied for all $c\in(0,\infty)$.
\end{lemma}

\begin{proof}
	For $c\le 1$, the assertion is an immediate consequence of (BC):
	\begin{align}
		&\lim_{m\to\infty}  \limsup_{n\to\infty} P\Big(\max_{m\leq |t|\leq r_n}\|X_t\|>cu_n \,\Big|\, \|X_0\|>cu_n \Big)\\
		&\leq \lim_{m\to\infty}\limsup_{n\to\infty} \sum_{m\leq |t|\leq r_n}P\big(\|X_t\|>cu_{n} \mid \|X_0\|>cu_{n} \big)
		\leq \lim_{m\to\infty} 2 \sum_{t=m}^{\infty} e_{\infty,c}(t)
		=0.
	\end{align}
	For $c> 1$ and  sufficiently large $n$, stationarity, regular variation and (BC) yield
	\begin{align}
		P\big(\|X_k\| >cu_n \mid\|X_0\|>cu_n\big) 
		&\leq P\big(\|X_{|k|}\|>u_n \mid\|X_0\|>u_n\big) \frac{P\{\|X_0\|>u_n\}}{ P\{\|X_0\|>cu_n\} }\\
		&\leq 2c^{\alpha}e_{n,1}(|k|)
	\end{align}
	for all $k\in\mathbb{Z}$.
	Thus, (BC) implies
	\begin{align}
		&\lim_{m\to\infty} \limsup_{n\to\infty}P\Big( \max_{m\leq |t|\leq r_n} \|X_t\|>cu_n \,\Big|\, \|X_0\|>cu_n\Big)\\
		&\leq \lim_{m\to\infty}\limsup_{n\to\infty} \sum_{m\leq |t|\leq r_n}P\big(\|X_t\|>cu_n \mid \|X_0\|>cu_n\big)
		\leq \lim_{m\to\infty} 4c^{\alpha} \sum_{t=m}^{\infty} e_{\infty,1}(t)
		=0. & \qedhere
	\end{align}
\end{proof}

\section{Proof of Lemma A.9, part (ii).}
Part (i) of Lemma A.9 was already established in the main paper. Here we prove the second assertion by similar arguments. First, recall the Lemma A.9 (ii):
\begin{lemma}
	 Under Conditions (RV), (S) and (M) (ii) $$\lim_{m\to\infty}\sup_{A\in\AA}|d^{(m)}_A-d_{A}|=0$$ with $d^{(m)}_A=E(f_A^{(m)}(Y))$ and $d_A$ defined in Theorem 2.3.
	
\end{lemma}

\begin{proof}
	In a first step we prove $$E[f_A(Y)]=d_A.$$
Direct calculations using $Y_h=\Theta_h\|Y_0\|$ with $\Theta_h$ and $\|Y_0\|$ independent, $P\{\|Y_0\|>y\}=y^{-\alpha}\wedge 1$ yield
\begin{align}
	&E[f_A(Y)]\\
	&=E\left[\sum_{h\in\mathbb{Z}}\left(\frac{\log(\|Y_h\|)\|Y_h\|^{\alpha}}{\|Y\|_\alpha^\alpha}-\frac{\|Y_h\|^{\alpha}\sum_{k\in\mathbb{Z}}\log(\|Y_k\|)\|Y_k\|^\alpha}{\|Y\|_\alpha^{2\alpha}}\right)\mathds{1}_{ A}\left(\frac{Y_{h+i}}{\|Y_h\|}\right)\right]\\
	&=E\left[\sum_{h\in\mathbb{Z}}\left(\frac{\log(\|\Theta_h\|)\|\Theta_h\|^{\alpha}}{\|\Theta\|_\alpha^\alpha} +\log(\|Y_0\|)\frac{\|\Theta_h\|^{\alpha}}{\|\Theta\|_\alpha^\alpha} -\frac{\|\Theta_h\|^{\alpha}\sum_{k\in\mathbb{Z}}\log(\|\Theta_k\|)\|\Theta_k\|^\alpha}{\|\Theta\|_\alpha^{2\alpha}}\right.\right.\\
	&\hspace{4cm}\left.\left. -\log(\|Y_0\|)\frac{\|\Theta_h\|^{\alpha}\sum_{k\in\mathbb{Z}}\|\Theta_k\|^\alpha}{\|\Theta\|_\alpha^{2\alpha}}\right)\mathds{1}_{ A}\left(\frac{\Theta_{h+i}}{\|\Theta_h\|}\right)\right]\\
	&=E\left[\sum_{h\in\mathbb{Z}}\frac{\|\Theta_h\|^{\alpha}}{\|\Theta\|_\alpha^\alpha}\left(\log(\|\Theta_h\|)\frac{\sum_{k\in\mathbb{Z}}\|\Theta_k\|^\alpha}{\|\Theta\|_\alpha^\alpha} -\frac{\sum_{k\in\mathbb{Z}}\log(\|\Theta_k\|)\|\Theta_k\|^\alpha}{\|\Theta\|_\alpha^{\alpha}}\right)\mathds{1}_{ A}\left(\frac{\Theta_{h+i}}{\|\Theta_h\|}\right)\right]\\
	&=E\left[\sum_{h\in\mathbb{Z}}\frac{\|\Theta_h\|^{\alpha}}{\|\Theta\|_\alpha^\alpha}\left( -\frac{\sum_{k\in\mathbb{Z}}\log(\|\Theta_k\|/\|\Theta_h\|)(\|\Theta_k\|/\|\Theta_h\|)^\alpha}{\|\Theta/\|\Theta_h\|\|_\alpha^{\alpha}} \right)\mathds{1}_{ A}\left(\frac{\Theta_{h+i}}{\|\Theta_h\|}\right)\right]\\
	&=E\Big[ -\sum_{k\in\mathbb{Z}}\log(\|\Theta_k\|)\|\Theta_k\|^\alpha \|\Theta\|_\alpha^{-\alpha}\mathds{1}_{A}(\Theta_{i})\Big]=d_A,\label{eq:da-is-Efa}
\end{align}
because the last two expectations can be interpreted as expectations of the same function on $l_\alpha$ w.r.t.\ $(P^\Theta)^{RS}$ and $P^\Theta$, respectively, and these two distributions coincide.

The remaining part of the proof is similar to the proof of part (i) of Lemma A.9.
Direct calculations show
\begin{align}
	&|d^{(m)}(A)-d(A)|\leq  E(|f_A^{(m)}(Y)-f_A(Y)|)\\
	&\leq  E\Bigg| \frac{\sum_{|h|\leq m} \log(\|Y_{h}\|)\|Y_{h}\|^{\alpha} \ind{A} \big(\frac{Y_{h+i}}{\|Y_{h}\|}\big)}{\sum_{|k|\leq m} \|Y_{k}\|^{\alpha }}
	- \frac{\sum_{h\in \Z}\log(\|Y_{h}\|)\|Y_{h}\|^{\alpha} \ind{A}\big(\frac{Y_{h+i}}{\|Y_{h}\|}\big)}{\sum_{k\in\Z} \|Y_{k}\|^{\alpha }} \Bigg|\\
	&\hspace{0.5cm}+ E\Bigg| \frac{\sum_{|h|\leq m}\|Y_{h}\|^\alpha \ind{A} \big(\frac{Y_{h+i}}{\|Y_{h}\|}\big)\sum_{|k|\leq m}\log(\|Y_{k}\|)\|Y_{k}\|^\alpha }{(\sum_{|k|\leq m} \|Y_{k}\|^{\alpha })^2} \\
	&\hspace{5cm}-  \frac{\sum_{h\in\Z}\|Y_{h}\|^\alpha \ind{A}\big(\frac{Y_{h+i}}{\|Y_{h}\|}\big) \sum_{k\in\Z}\log(\|Y_{k}\|)\|Y_{k}\|^\alpha }{(\sum_{k\in\Z} \|Y_{k}\|^{\alpha })^2} \Bigg|\\
	&=: T_1^Y+T_2^Y.
\end{align}
Similar as for the bound of $T_1$ in the proof of part (i), we obtain
\begin{align}
	T_1^Y&\leq E\bigg( \frac{\sum_{|h|\leq m} |\log\|Y_{h}\||\|Y_{h}\|^{\alpha} \sum_{|k|>m} \|Y_{k}\|^{\alpha }}{\sum_{|k|\leq m} \|Y_{k}\|^{\alpha }\sum_{k\in\Z} \|Y_{k}\|^{\alpha }}\bigg)
	+E\bigg(\frac{\sum_{|h|>m} |\log\|Y_{h}\||\|Y_{h}\|^{\alpha} }{\sum_{k\in\mathbb{Z}} \|Y_{k}\|^{\alpha }}\bigg)\\
	&=:T_{1,1}^Y+T_{1,2}^Y.
\end{align}
Applying the H\"{o}lder inequality for expectations and for sums, we can bound $T_{1,1}^Y$ by
\begin{align}
	\Bigg( E\bigg( \frac{\sum_{|h|\leq m} |\log\|Y_{h}\||^{1+\delta}\|Y_{h}\|^{\alpha} }{\sum_{|k|\leq m} \|Y_{k}\|^{\alpha }}\bigg)\Bigg)^{1/(1+\delta)} \Bigg(E\bigg( \bigg(\frac{ \sum_{|k|>m} \|Y_{k}\|^{\alpha }}{\sum_{k\in\Z} \|Y_{k}\|^{\alpha }}\bigg)^{(1+\delta)/\delta}\bigg)\Bigg)^{\delta/(1+\delta)}.
\end{align}
According to Remark 2.2 (a), the first expectation is bounded if Condition (M) (ii) is met. The second expectation converges to $0$ as $m\to\infty$ by monotone convergence. Hence $\lim_{m\to\infty}T_{1,1}^Y=0$. Since (2.5) in combination with $|\log\|Y_h\||\le |\log\|Y_h\||^{1+\delta}+1$ also implies that $T_{1,2}^Y<\infty$ for all $m\in\N$, monotone convergence yields $\lim_{m\to\infty}T_{1,2}^Y=0$, too.

By analogous arguments as used for $T_2$ in the proof of part (i), one can show that
\begin{align}
	T_2^Y&\leq 2E\bigg( \frac{ \sum_{|k|\leq m}|\log\|Y_{k}\||\|Y_{k}\|^\alpha   \sum_{|h|>m} \|Y_{h}\|^\alpha} {\sum_{|k|\leq m} \|Y_{k}\|^{\alpha } \sum_{k\in\Z} \|Y_{k}\|^{\alpha }}     \bigg)\\
	&\hspace{1cm}+E\bigg( \frac{ \sum_{|k|>m}|\log\|Y_{k}\||\|Y_{k}\|^\alpha }{\sum_{k\in\Z } \|Y_{k}\|^{\alpha }}   \bigg)+E\bigg( \frac{\sum_{|h|>m}\|Y_{h}\|^\alpha \sum_{|k|\leq m}|\log\|Y_{k}\||\|Y_{k}\|^\alpha }{(\sum_{k\in\Z } \|Y_{k}\|^{\alpha })^2}  \bigg)\\
	&\le 3T_{1,1}^Y+T_{1,2}^Y,
\end{align}
so that $\lim_{m\to\infty}T_2^Y=0$. Since the bounds on $T_1^Y$ and $T_2^Y$ do not depend on $A$, all convergences hold uniformly in $A$, which proves assertion (ii).	
\end{proof}

\section{Proof of Lemma A.10.}

Lemma A.10 of the main paper establishes bounds on the solutions to two optimization problems. For convenience, the lemma is restated here.
\begin{lemma}
	\label{lemma:maximizing.m}
	For $m\in\N$ and $a=(a_1,\ldots,a_m)\in [0,1]^m$ let
	\begin{align}
		M(a)& :=\frac{\sum_{k=1}^{m}a_k \log^2a_k}{1+\sum_{k=1}^{m}a_k},\qquad
		\tilde M(a)  := \frac{\sum_{k=1}^{m}a_k|\log a_k|}{1+\sum_{k=1}^{m}a_k}
	\end{align}
	with $0\log^i(0):=0$ for $i\in\{1,2\}$.
	Then $\sup_{a\in[0,1]^m} M(a)=\Ord(\log^2 m)$ and $\sup_{a\in[0,1]^m} \tilde M(a)$ $=\Ord(\log m)$
	as $m\to\infty$.
\end{lemma}
\begin{proof}
	We only establish the bound on $M$, as the second assertion follows by similar arguments.
	
	First note that $x\mapsto x\log^2x$ is decreasing on the interval $[\e^{-2},1]$. Thus, if $a_h>\e^{-2}$ for some $h\in\{1,\ldots,m\}$, then  replacing $a_h$ by  $\e^{-2}$ increases $M(a)$, i.e.\ any point of maximum must belong to $[0,\e^{-2}]^m$.
	
	Check that the partial derivative
	\begin{align}
		\label{eq:maximizing-problem-2-taylor}
		\frac{\partial}{\partial a_h}M(a) &= \frac{\log^2 a_h+2\log a_h}{1+\sum_{k=1}^{m}a_k} -\frac{\sum_{k=1}^{m} a_k\log^2a_k}{(1+\sum_{k=1}^{m}a_k)^2}
	\end{align}
	tends to $\infty$ as $a_h\downarrow 0$. Hence all points of maximum $a$ must be in $(0,\e^{-2}]^m$ and all partial derivatives must vanish at such points. In particular,
	$\log^2 a_h+2\log a_h
	= \log^2 a_1 +2\log a_1$ for all $h\in\{2,\ldots,m\}$, which is equivalent to
	$
	(\log(a_h)+1)^2=(\log(a_1)+1)^2
	$
	and thus to $a_h=a_1$ or $a_h=1/(\e^2 a_1)$. Because the second case contradicts $a_1,a_h\le \e^{-2}$, we have shown so far that all coordinates of a point of maximum must be equal.
	
	The condition on the partial derivatives then reads as
	\begin{align}
		\frac{\log^2 a_1+2\log a_1}{1+ma_1}-\frac{ma_1\log^2 a_1}{(1+ma_1)^2}=0\quad
		\Leftrightarrow \quad |\log a_1|=2(1+ma_1).
		\label{eq:maximierung-m-2-taylor-endgleichung}
	\end{align}
	This equation has a unique solution $a_1^*$. It is easily seen that $a_1^*$ is of smaller order than $m^{\delta-1}$ and of larger order than $m^{-\delta-1}$ as $m\to\infty$ for all $\delta>0$. In particular, $|\log a_1^*|\sim\log m$ and thus
	\begin{equation}
		a_1^*\sim \frac{\log m}{2m}
	\end{equation}
	as $m\to\infty$. The maximum of $M$ is given by
	\begin{align}
		M(a_1^*,\ldots,a_1^*) = &\frac{m a_1^* \log^2 a_1^*}{1+ma_1^*}=4ma_1^*(1+ma_1^*)
		\sim  \log^2 m \label{eq:upper.bound.m.unknow.alpha}
	\end{align}
	as $m\to\infty$.
\end{proof}

\section{Proof of Lemma A.12.}
Lemma A.12 of the paper contains the limit behavior of the covariances between $(Z(A))_{A\in\mathcal{A}}$ and $\Zphi$ and the variance of $\Zphi$. The proof uses similar ideas as the proof of Lemma A.4.
\begin{lemma}
	\label{lemma:covaraince.pb.unknown.alpha}
	Suppose the conditions (RV), (S), (BC), (TC), (C$\Theta$), and (BC') are satisfied. Then
	\begin{enumerate}
		\item[(i)]
		\begin{align}
			\frac{1}{r_nv_n}&Var\bigg(\sum_{t=1}^{r_n}\phi(W_{n,t})\bigg)
			=\alpha^{-1}\sum_{k\in\mathbb{Z}}E\left[(1\wedge \|\Theta_k\|^\alpha)(|\log(\|\Theta_k\|)|+2\alpha^{-1})\right],
		\end{align}
		
		\item[(ii)]
		\begin{align}
			\frac{1}{r_nv_n}&Cov\bigg(\sum_{j=1}^{r_n}\phi(W_{n,j}),\sum_{t=1}^{r_n}g_A(W_{n,t})\bigg)\\
			&=\sum_{k\in\mathbb{Z}}E\left[\sum_{h\in\mathbb{Z}}\frac{\|\Theta_h\|^\alpha}{\|\Theta\|_\alpha^\alpha}\ind{ A}\left(\frac{\Theta_{h+i}}{\|\Theta_h\|}\right) (1\wedge \|\Theta_k\|^\alpha)(\log^+\|\Theta_k\|+\alpha^{-1})\right]
		\end{align}
		for all $A\in\mathcal{A}$.
	\end{enumerate}
\end{lemma}

\begin{proof}
	Regular variation  implies
	\begin{align}	
		E\big(\log^+(\|X_{n,0}\|)\log^+(\|X_{n,k}\|) \mid \|X_0\|>u_n\big)
		& \rightarrow E[\log^+(Y_0)\log^+(Y_k)]\\
		E\big(\log^+ \|X_{n,0}\| \mid \|X_0\|>u_n\big) & \rightarrow E[\log^+Y_0]
		\label{eq:regvar.moment.log}
	\end{align}
	for all $k\in\mathbb{Z}$ (cf.\ \cite{KS20}, Section 2.3.3). Therefore, by stationarity we obtain
	\begin{align}	&\frac{1}{r_nv_n}Var\bigg(\sum_{t=1}^{r_n}\phi(W_{n,t})\bigg)\\	&=\frac{1}{v_n} \sum_{k=-r_n}^{r_n} \Big(1-\frac{|k|}{r_n}\Big) E[\phi(W_{n,0})\phi(W_{n,k})]
		- \frac{(r_nv_n)^2}{r_nv_n}\left(E(\log^+\|X_{n,0}\|\mid \|X_0\|>u_n)\right)^2\\	&=\sum_{k=-r_n}^{r_n}\Big(1-\frac{|k|}{r_n}\Big)E\big(\log^+(\|X_{n,0}\|)\log^+(\|X_{n,k}\|) \mid \|X_0\|>u_n\big) + \ord(1)\\
		&\rightarrow \sum_{k\in\mathbb{Z}}E[\log^+(\|Y_0\|)\log^+(\|Y_k\|)],
	\end{align}
	where the last step follows from Pratt's Lemma, which may be applied because of (BC').
	
	Check that
	\begin{align}
		\int_u^\infty \log(y) \alpha y^{-(\alpha+1)}\, dy & = u^{-\alpha}(\log u+\alpha^{-1})\\
		\int_u^\infty \log^2(y) \alpha y^{-(\alpha+1)}\, dy & = u^{-\alpha}(\log^2 u+2\alpha^{-1}\log u + 2\alpha^{-2}).
	\end{align}
	
	Direct calculations similar to \eqref{eq:da-is-Efa}, using $Y_k=\Theta_k\|Y_0\|$ with $\Theta_k$ and $\|Y_0\|$ independent and $P\{\|Y_0\|>y\}=y^{-\alpha}\wedge 1$ yield
	\begin{align}
		&E\left[\log^+(\|Y_0\|)\log^+(\|Y_k\|)\right]\\
		& =E\left[\log^+(\|Y_0\|)\log^+(\|\Theta_k\|\|Y_0\|)\right]\\
		&=E\left[\int_{1\vee \|\Theta_k\|^{-1}}^{\infty}\log(y)\log(\|\Theta_k\|y)\alpha y^{-\alpha-1}dy\right]\\
		& = E\Big[ (1\vee \|\Theta_k\|^{-1})^{-\alpha}\big(\log\|\Theta_k\|(\log(1\vee \|\Theta_k\|^{-1})+\alpha^{-1})\\
		& \hspace{4cm}  + \log^2(1\vee \|\Theta_k\|^{-1})+2\alpha^{-1}\log(1\vee \|\Theta_k\|^{-1})+2\alpha^{-2}\big)\Big]\\
		& = E\Big[ (1\wedge \|\Theta_k\|^{\alpha})\big(\log\|\Theta_k\|(-\log(1\wedge \|\Theta_k\|)+\alpha^{-1}) \\
		& \hspace{3cm} + \log^2(1\wedge \|\Theta_k\|)-2\alpha^{-1}\log(1\wedge \|\Theta_k\|)+2\alpha^{-2}\big)\Big]\\
		&=\alpha^{-1}E\left[(1\wedge \|\Theta_k\|^\alpha)(|\log \|\Theta_k\|| +2\alpha^{-1})\right]
	\end{align}
	for all $k\in\mathbb{Z}$, which proves assertion (i).
	
	Next we prove
	\begin{align}
		E\big(g_A(W_{n,0})\log^+\|X_{n,k}\|\, \big|\, \|X_0\|>u_n\big)
		&\rightarrow E\big[g_A(Y)\log^+\|Y_k\|\big]\\
& = E\bigg[\sum_{h\in\mathbb{Z}}\frac{\|Y_h\|^\alpha}{\|Y\|_\alpha^\alpha}\ind{ A}\Big(\frac{Y_{h+i}}{\|Y_h\|}\Big)\log^+\|Y_k\|\bigg] \label{eq:gA_log_conv}
	\end{align}
	for all $k\in\mathbb{Z}$, using similar techniques as in the proof of \eqref{eq:excong}. In particular, by the same arguments and uniform integrability one obtains, for all $m\in\N, k\in\Z$,
  \begin{align}
     \lim_{n\to\infty} E\big(g_A^{(m)}(W_{n,0})\log^+\|X_{n,k}\|\, \big|\, \|X_0\|>u_n\big) &= E\big[g_A^{(m)}(Y)\log^+\|Y_k\|\big], \label{eq:A12_approx1}\\
    \lim_{m\to\infty} E\big[g_A^{(m)}(Y)\log^+\|Y_k\|\big] & = E\big[g_A(Y)\log^+\|Y_k\|\big]. \label{eq:A12_approx2}
  \end{align}
   Furthermore, because $g_A\in[0,1]$, the Cauchy-Schwarz inequality yields
   \begin{align}
     E& \big(\big|g_A^{(m)}(W_{n,0})-g_A(W_{n,0})\big|\log^+\|X_{n,k}\|\, \big|\, \|X_0\|>u_n\big)\\
     & \le \Big[E \big(\big|g_A^{(m)}(W_{n,0})-g_A(W_{n,0})\big|\, \big|\, \|X_0\|>u_n\big)E\big((\log^+\|X_{n,k}\|)^2\, \big|\, \|X_0\|>u_n\big)\Big]^{1/2}.
   \end{align}
   Because the second expectation converges to $E[(\log^+\|Y_k\|)^2]$, \eqref{eq:gAm_Wnj_approx} shows that
   $$ \lim_{m\to\infty} \limsup_{n\to\infty} E \big(\big|g_A^{(m)}(W_{n,0})-g_A(W_{n,0})\big|\log^+\|X_{n,k}\|\, \big|\, \|X_0\|>u_n\big)=0,
   $$
   which, in combination with \eqref{eq:A12_approx1} and \eqref{eq:A12_approx2}, yields \eqref{eq:gA_log_conv} by standard arguments.

 Now, by stationarity and Pratt's lemma,
	\begin{align}
		\frac{1}{r_nv_n} & Cov\bigg(\sum_{j=1}^{r_n}\phi(W_{n,j}),\sum_{t=1}^{r_n}g_A(W_{n,t})\bigg)\\ &=\sum_{k=-r_n}^{r_n}\Big(1-\frac{|k|}{r_n}\Big) E\Big(\log^+(\|X_{n,k}\|)g_A(W_{n,0}) \,\Big|\, \|X_0\|>u_n\Big) + \ord(1)\\
		&\rightarrow \sum_{k\in\mathbb{Z}}E\bigg[\log^+(\|Y_k\|)\sum_{h\in\mathbb{Z}}\frac{\|Y_h\|^\alpha}{\|Y\|_\alpha^\alpha}\ind{ A}\Big(\frac{Y_{h+i}}{\|Y_h\|}\Big)\bigg].
	\end{align}
		
	Similarly as above, we obtain
	\begin{align}
		&E\left[\sum_{h\in\mathbb{Z}}\frac{\|Y_h\|^\alpha}{\|Y\|_\alpha^\alpha}\ind{A} \left(\frac{Y_{h+i}}{\|Y_h\|}\right)\log^+(\|Y_k\|)\right]\\
		&=E\left[\sum_{h\in\mathbb{Z}}\frac{\|\Theta_h\|^\alpha}{\|\Theta\|_\alpha^\alpha}\ind{A} \left(\frac{\Theta_{h+i}}{\|\Theta_h\|}\right)\log^+(\|\Theta_k\|\|Y_0\|)\right]\\
		&=E\left[\int_{1\vee \|\Theta_k\|^{-1}}^{\infty} \sum_{h\in\mathbb{Z}}\frac{\|\Theta_h\|^\alpha}{\|\Theta\|_\alpha^\alpha}\ind{ A}\left(\frac{\Theta_{h+i}}{\|\Theta_h\|}\right)\log(\|\Theta_k\|y) \alpha y^{-\alpha-1}dy \right]\\
		&=E\left[\sum_{h\in\mathbb{Z}}\frac{\|\Theta_h\|^\alpha}{\|\Theta\|_\alpha^\alpha}\ind{ A}\left(\frac{\Theta_{h+i}}{\|\Theta_h\|}\right)(1\vee \|\Theta_k\|^{-1})^{-\alpha} \big( \log\|\Theta_k\| +  \log(1\vee\|\Theta_k\| ^{-1})+\alpha^{-1}\big)\right]\\
		&=E\left[\sum_{h\in\mathbb{Z}}\frac{\|\Theta_h\|^\alpha}{\|\Theta\|_\alpha^\alpha}\ind{ A}\left(\frac{\Theta_{h+i}}{\|\Theta_h\|}\right)(1\wedge \|\Theta_k\|^\alpha)(\log^+\|\Theta_k\|+\alpha^{-1})\right],
	\end{align}
	which concludes the proof of assertion (ii).
	
\end{proof}

\bibliography{Dissbib_070321}
\bibliographystyle{agsm}